\documentclass[a4paper,10pt,DIV=12,bibliography=totoc]{scrartcl}

\usepackage{amsmath,amssymb,amsthm}
\usepackage[T1]{fontenc}
\usepackage[utf8]{inputenc}
\usepackage{lmodern}
\usepackage{graphicx}
\usepackage{float}
\usepackage[dvipsnames,svgnames]{xcolor}
\usepackage{tikz,pgf}
\usetikzlibrary{arrows.meta}
\usetikzlibrary{calc}
\usepackage{booktabs}
\usepackage[colorlinks=true,linkcolor=RoyalBlue,citecolor=PineGreen,urlcolor=RoyalBlue]{hyperref}
\usepackage{caption}
\usepackage{subcaption}
\usepackage[nameinlink]{cleveref}
\usepackage{array}

\newcommand{\sst}[2]{\left\lbrace #1\,\textbf{\upshape\textbar}\,#2 \right\rbrace}

\newcommand{\lamsymb}{\mathbf{c}}
\newcommand{\lam}[2]{\lamsymb_{#1}(#2)}
\newcommand{\lamfunc}[1]{\lamsymb_{#1}}
\newcommand{\lamIIfunc}{\lamfunc{2}}
\newcommand{\lamIIIfunc}{\lamfunc{3}}
\newcommand{\lamII}[1]{\lam{2}{#1}}
\newcommand{\lamIII}[1]{\lam{3}{#1}}

\newcommand{\lamrsymb}{\mathbf{r}}

\newcommand{\lamrII}{\lamrsymb_2}

\newcommand{\lamSsymb}{\lamsymb^\circ}
\newcommand{\lamS}[2]{\lamSsymb_{#1}(#2)}
\newcommand{\lamSfunc}[1]{\lamSsymb_{#1}}
\newcommand{\lamSIIfunc}{\lamSfunc{2}}

\newcommand{\lamSII}[1]{\lamS{2}{#1}}

\newcommand{\maxsymb}{\mathbf{M}}

\newcommand{\maxlam}[2]{\maxsymb_{#1}(#2)}
\newcommand{\maxlamII}[1]{\maxlam{2}{#1}}
\newcommand{\maxlamIII}[1]{\maxlam{3}{#1}}
\newcommand{\maxlamIIn}{\maxlamII{n}}
\newcommand{\maxlamIIIn}{\maxlamIII{n}}

\newcommand{\maxsymbS}{\maxsymb^\circ}
\newcommand{\maxlamS}[2]{\maxsymbS_{#1}(#2)}

\newcommand{\maxlamSII}[1]{\maxlamS{2}{#1}}
\newcommand{\maxlamSIIn}{\maxlamS{2}{n}}

\newcommand{\maxsetsymb}{\mathcal{H}}
\newcommand{\maxlamset}[2]{\maxsetsymb_{#1}(#2)}

\newcommand{\maxlamdnset}{\maxlamset{d}{n}}

\newcommand{\maxsetSsymb}{\maxsetsymb^\circ}
\newcommand{\maxlamSset}[2]{\maxsetSsymb_{#1}(#2)}
\newcommand{\maxlamSIIset}[1]{\maxlamSset{2}{#1}}

\newcommand{\maxlamSdnset}{\maxlamSset{d}{n}}

\newcommand{\setsymb}{\mathcal{C}}
\newcommand{\lamset}[3]{\setsymb_{#1}(#2,#3)}

\newcommand{\setSsymb}{\setsymb^\circ}
\newcommand{\lamSset}[3]{\setSsymb_{#1}(#2,#3)}

\newcommand{\minsymb}{\mathbf{m}}
\newcommand{\minlam}[2]{\minsymb_{#1}(#2)}
\newcommand{\minlamIII}[1]{\minlam{3}{#1}}
\newcommand{\minlamIIIn}{\minlamIII{n}}
\newcommand{\minsymbS}{\minsymb^\circ}
\newcommand{\minlamS}[2]{\minsymbS_{#1}(#2)}

\newcommand{\minsetsymb}{\mathcal{S}}
\newcommand{\minlamset}[2]{\minsetsymb_{#1}(#2)}

\newcommand{\minlamdnset}{\minlamset{d}{n}}

\newcommand{\minsetSsymb}{\minsetsymb^\circ}
\newcommand{\minlamSset}[2]{\minsetSsymb_{#1}(#2)}

\newcommand{\minlamSdnset}{\minlamSset{d}{n}}

\newcommand{\modop}{\,\mathrm{mod}\,}

\newcommand{\bnrs}[3]{\tikz[baseline=3pt]{\fill[#3] (0,0)rectangle(#1*#2,0.4);\node[anchor=east] at (0,0.2) {#1};}}

\newcommand{\msolve}{\texttt{msolve}}

\renewcommand{\leq}{\leqslant}
\renewcommand{\geq}{\geqslant}

\newcommand{\R}{\mathbb R}
\newcommand{\C}{\mathbb C}

\newcommand{\setM}[2]{\mathcal M_{#1}^{#2}}
\newcommand{\setMn}{\setM{2}{n}}
\newcommand{\setD}[3]{\mathcal D_{#1}^{#2,#3}}
\newcommand{\setDn}{\setD{2}{n}{3}}

\DeclareMathOperator{\id}{id}

\newtheorem{theorem}{Theorem}
\newtheorem{proposition}[theorem]{Proposition}
\newtheorem{lemma}[theorem]{Lemma}
\newtheorem{corollary}[theorem]{Corollary}
\newtheorem{definition}[theorem]{Definition}
\theoremstyle{definition}

\newtheorem{oproblem}{Open Problem}
\newtheorem{conjecture}[oproblem]{Conjecture}

\colorlet{ncol}{green!60!black}
\colorlet{colbg}{white}
\colorlet{colfg}{black}
\colorlet{colgraphv}{colfg!75!white}
\colorlet{colgraphe}{colfg!55!white}
\colorlet{colG}{DarkSeaGreen}
\definecolor{colR}{HTML}{CC6677}
\definecolor{colO}{HTML}{DDCC77}
\definecolor{colB}{HTML}{6699CC}
\colorlet{colY}{Gold!90!black}
\colorlet{colGray}{white!60!black}
\colorlet{colGold}{Gold!90!black}
\colorlet{colP}{colY!80!white}
\colorlet{colX}{black!50!white}

\colorlet{colGw}{colG!60!white}
\colorlet{colBw}{colB!60!white}
\colorlet{colRw}{colR!60!white}
\colorlet{colPw}{colP!60!white}
\colorlet{colXw}{colX!60!white}

\colorlet{cola}{colfg!75!colbg}

\tikzstyle{vertex}=[fill=colgraphv,circle,inner sep=0pt, minimum size=4pt]
\tikzstyle{fvertex}=[vertex,minimum size=3pt,fill=colbg,draw=colgraphv,double=colbg,double distance=0.25pt,outer sep=1pt]
\tikzstyle{nvertex}=[vertex,colG]
\tikzstyle{hvertex}=[vertex,colB]
\tikzstyle{edge}=[line width=1.5pt,colgraphe]
\tikzstyle{oedge}=[edge,colR]
\tikzstyle{hedge}=[edge,colB]
\tikzstyle{nedge}=[edge,colG]
\tikzstyle{edgeq}=[edge,dashed,colGray]
\tikzstyle{dedge}=[edge,-latex]
\tikzstyle{gridl}=[black!50!white]

\tikzstyle{value}=[circle,inner sep=0pt, minimum size=2pt]
\tikzstyle{valueG}=[value,fill=colG]
\tikzstyle{valueGw}=[value,fill=colGw]
\tikzstyle{valueR}=[value,fill=colR]
\tikzstyle{valueRw}=[value,fill=colRw]
\tikzstyle{valueB}=[value,fill=colB]
\tikzstyle{valueBw}=[value,fill=colBw]
\tikzstyle{valueP}=[value,fill=colP]
\tikzstyle{valuePw}=[value,fill=colPw]
\tikzstyle{valueX}=[value,fill=colX]
\tikzstyle{valueXw}=[value,fill=colXw]

\tikzstyle{genericgraph}=[dashed,black!70!white]
\tikzstyle{moreedges}=[black!70!white,dotted,shorten <= 2pt, shorten >= 2pt]
\tikzstyle{construle}=[ultra thick,-{Classical TikZ Rightarrow[]}]
\tikzstyle{dconstrule}=[ultra thick,{Classical TikZ Rightarrow[]}-{Classical TikZ Rightarrow[]}]
\tikzstyle{eframe}=[black!30!white,rounded corners]

\tikzstyle{labelsty}=[font=\scriptsize]
\tikzstyle{axes}=[cola,-{Classical TikZ Rightarrow[]}]
\tikzstyle{indline}=[dashed,black!40!white]
\tikzstyle{gline}=[line width=1pt]

\tikzstyle{aline}=[draw=cola]
\tikzstyle{bline}=[draw=cola!10!colbg]
\tikzstyle{alabelsty}=[cola,font=\scriptsize]
\tikzstyle{legend}=[minimum size=4pt]

\allowdisplaybreaks[1]
\begin{document}

\title{Explorations on the number of realizations of minimally rigid graphs}

\author{%
	Georg Grasegger%
	\thanks{Johann Radon Institute for Computational and Applied Mathematics (RICAM), Austrian Academy of Sciences}
}
\date{}

\maketitle

\begin{abstract}
	Rigid graphs have only finitely many realizations. In the recent years significant progress was made in computing the number of such realizations.
	With this progress it was also possible for the first time to do computations on large sets of graphs.
	In this paper we show what we can conclude from the data we got from these computations.
	This includes new lower bounds on the maximal realization count for a given number of vertices, upper bounds for the minimal realization count in higher dimensions and effects of rigidity preserving construction rules on the realization number.
	In all cases we give certificate graphs which prove the respective results.
\end{abstract}

\section{Introduction}
It is well known that given three lengths we can uniquely draw a triangle in the plane with these side lengths up to isometries as long as they fulfill the triangle inequalities.
The triangle graph is therefore called rigid.
When we glue two triangles at an edge we get a graph for which there are more ways to draw it for reasonable lengths.
Still there are finitely many up to direct isometries (see \Cref{fig:count-example}).
Note, that in this paper we count reflections to be different realizations in order to follow the main references.

\begin{figure}[ht]
	\centering
	\begin{tikzpicture}[scale=0.8]
			\node[vertex] (a1) at (0,-2) {};
			\node[vertex] (b1) at (0,0) {};
			\node[vertex] (c1) at (2,0) {};
			\node[vertex] (d1) at (2,1) {};
			\node[vertex] (a2) at (3,2) {};
			\node[vertex] (b2) at (3,0) {};
			\node[vertex] (c2) at (5,0) {};
			\node[vertex] (d2) at (5,-1) {};
			\node[vertex] (a3) at (6,2) {};
			\node[vertex] (b3) at (6,0) {};
			\node[vertex] (c3) at (8,0) {};
			\node[vertex] (d3) at (8,1) {};
			\node[vertex] (a4) at (9,-2) {};
			\node[vertex] (b4) at (9,0) {};
			\node[vertex] (c4) at (11,0) {};
			\node[vertex] (d4) at (11,-1) {};
			\path[edge] (a1)edge(b1) (b1)edge(c1) (c1)edge(a1) (b1)edge(d1) (c1)edge(d1);
			\path[edge] (a2)edge(b2) (b2)edge(c2) (c2)edge(a2) (b2)edge(d2) (c2)edge(d2);
			\path[edge] (a3)edge(b3) (b3)edge(c3) (c3)edge(a3) (b3)edge(d3) (c3)edge(d3);
			\path[edge] (a4)edge(b4) (b4)edge(c4) (c4)edge(a4) (b4)edge(d4) (c4)edge(d4);
		\end{tikzpicture}
	\caption{All realizations up to direct isometries of a 4-vertex graph with given edge lengths.}
	\label{fig:count-example}
\end{figure}
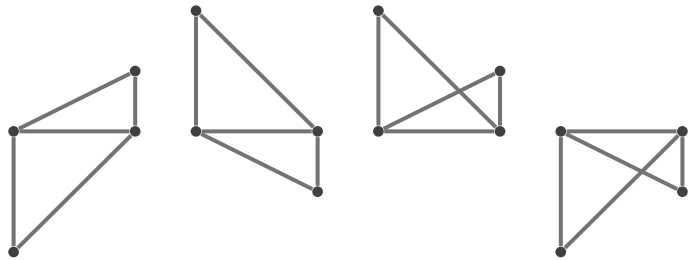

In general we are interested in the number of realizations (i.\,e., drawings) that we can get for a graph with edge lengths.
A graph together with a realization is called a framework. We call a framework rigid, if there are only finitely many other frameworks up to direct isometries with the same edge lengths.
However, we need to be careful here. In the Euclidean space the number of realizations of a framework does depend on the choice of the lengths.
It is known that rigidity is a generic property, i.\,e. we can consider it a graph property and hence we can think of a realization number of a graph using some genericity assumptions.
However, it turns out that this number is not easy to compute.
To circumvent this issue recent research has focused on counting realizations over the complex numbers.
This has the advantage that we need to care less about genericity and tools from algebraic geometry can be used to actually compute such realizations counts combinatorially.

\minisec{Previous work}

The number of realizations of rigid graphs in different dimensions has recently attained a lot of attention. The gain in interest is partially based on new combinatorial algorithms.
But even before that the realization count was analyzed algebraically \cite{SteffensTheobald,BorceaStreinu2004,Emiris2009,EmirisMoroz} and combinatorially \cite{Jackson2018}.
In \cite{PlaneCount} a combinatorial algorithm for counting the number of complex realizations of a minimally rigid graph in dimension 2 was presented. An implementation can be found in \cite{ZenodoAlg,RigiComp,CapcoCpp} and previous computing data is available at \cite{ZenodoData}. Further improvements have been presented in \cite{Calligraphs}.
Similar considerations were done on the complex sphere \cite{SphereCount,CallisphereT} with implementation at \cite{SphereAlg,RigiComp,CalligraphSoftware, CapcoCpp}.
In practice very often real realizations play a role. These have been counted for instance in \cite{Bartzos2018,BartzosISSAC}.
It is known that there are graphs that have fewer real realizations than complex ones even generically \cite{Jackson2018}.
Particularly interesting is the maximal number of realizations for graphs with a given number of vertices.
Lower bounds for this have been presented in \cite{LowerBounds,Bartzos2018,BartzosISSAC} and upper bounds were found in \cite{BorceaStreinu2004,BartzosNew,Bartzos2020,Bartzos23}.

\minisec{Our contribution}
In this paper we improve the existing lower bounds on the maximal number of realizations for a given vertex count for the complex plane, sphere and three dimensional space by further exploiting the methods from \cite{LowerBounds}. In higher dimensions we also show upper bounds for the minimal such number.
Furthermore, we show observations and results from extensive computational experiments. As such we work on an open problem from \cite{LowerBounds} to see how certain rigid graph constructions influence the number of realizations. This is interesting because these constructions can be used to generate every minimally rigid graph.

\minisec{Organization of the paper}
We repeat basic notions of rigidity in \Cref{sec:prelim}.
One main part of this paper is to obtain better lower bounds for the maximal number of realizations for graphs with a given number of vertices.
To do so we construct large graphs by gluing them on minimally rigid subgraphs (\Cref{sec:fan}).
In \Cref{sec:dim-2,sec:sphere,sec:space} we present new bounds in the plane, the sphere, and higher dimensional spaces, respectively.
Having all the computational data, we analyze thoroughly the impact of certain rigidity preserving constructions on the number of realizations (\Cref{sec:factors}).

Throughout the paper we represent graphs by integers which are obtained by
flattening the upper triangular part of their adjacency matrices and interpreting
the binary sequence as an integer.
We refer to \Cref{appendix:encoding} for further details.
In \Cref{sec:enc:sepcial} we collect the encodings
of all graphs mentioned throughout the paper.

\section{Preliminaries}\label{sec:prelim}
In this section we briefly recall the necessary definitions from rigidity theory and introduce the notation we are using in this paper.
\begin{definition}
	A realization of a graph in a space $R$ is a map $\rho\colon V \rightarrow R$.
	Two realizations $\rho$ and $\rho'$ are congruent if $||\rho(u)-\rho(v)||=||\rho'(u)-\rho'(v)||$ for all pairs of vertices $u,v\in V$.
	They are called equivalent if this only holds for pairs $u,v$ that form an edge.

	A graph $G=(V,E)$ is called (generically) $d$-rigid if for any generic\footnote{i.\,e.\ where the coordinates are algebraically independent over the rationals} realization $\rho\colon V \rightarrow \R^d$ there exists $\varepsilon>0$
	such that an equivalent realization $\rho'$ with $||\rho(v)-\rho'(v)||<\varepsilon$ for every $v\in V$ is indeed congruent.
\end{definition}
A graph is minimally $d$-rigid if it is $d$-rigid and the deletion of any edge yields a graph that is not $d$-rigid any more.
Whenever we say rigid in this paper we always mean 2-rigid. For rigidity in higher dimensions we explicitly state the $d$.

Minimally 2-rigid graphs have been classified combinatorially and lists of all minimally rigid graphs up to a certain number of vertices can be computed \cite{RigiComp,Larsson,ZenodoData}.
The rigid graphs for realizations on the $d$-dimensional sphere are indeed the $d$-rigid graphs (compare \cite{Eftekhari2019}).
Unfortunately for $d\geq3$ there is no combinatorial classification of $d$-rigidity. Nevertheless, sets of $d$-rigid graphs can be computed using probabilistic methods with no false positives.

Throughout the paper we use the following notation:
\begin{center}
	\begin{tabular}{ll}
		$\setM{d}{n}$ & set of minimally $d$-rigid graphs with $n$ vertices\\
		$\setD{d}{n}{k}$ & set of minimally $d$-rigid graphs with $n$ vertices and minimum degree $k$
	\end{tabular}
\end{center}

The definition of rigidity implies that for a (minimally) $d$-rigid graph there are only finitely many non-congruent realizations in $d$-dimensional space up to isometries.
In the rest of the paper we consider realizations over the complex numbers, and we use $||\cdot||$ to be the extension of the Euclidean norm defined by $||x||=\sum_{i=1}^n x_i^2$, where $x_i$ is the $i$-th component of the vector $x$.
\begin{definition}\label{def:laman-number-2d}
	For a minimally $d$-rigid graph $G=(V,E)$ we define $\lam{d}{G}$ to be the number of non-congruent complex realizations of $G$ in $\C^d$.
	Further, we define $\maxlam{d}{n}$ to be the largest $\lam{d}{G}$ and $\minlam{d}{n}$ to be the smallest $\lam{d}{G}$ among all minimally $d$-rigid graphs with $n$ vertices.

	Similarly, we define $\lamS{d}{G}$, $\maxlamS{d}{n}$ and $\minlamS{d}{n}$ for realizations on the $d$-dimensional sphere.

	We define $\lamset{d}{n}{c}=\sst{G\in\setM{d}{n}}{\lam{d}{G}=c}$ and respectively $\lamSset{d}{n}{c}$.

	Let $\maxlamdnset=\lamset{d}{n}{\maxlam{d}{n}}$ be the set of all minimally $d$-rigid graphs with $n$ vertices and $\lam{d}{G}=\maxlam{d}{n}$
	and let $\minlamdnset=\lamset{d}{n}{\minlam{d}{n}}$ be the set of all minimally $d$-rigid graphs with $n$ vertices and $\lam{d}{G}=\minlam{d}{G}$.
	Similarly we define $\maxlamSdnset$ and $\minlamSdnset$.
\end{definition}
For further details on the definition of the complex realization count and its relation to the rigidity map, we refer to \cite{PlaneSphere}.
Note that here, in the contrary to \cite{Jackson2018,PlaneSphere}, we do consider a reflection as non-congruent to be consistent with \cite{PlaneCount,SphereCount}.
This means the values we get for $\lam{d}{G}$ are twice the value one would get from \cite{Jackson2018,PlaneSphere}.

\section{Fan constructions}\label{sec:fan}
It is a common strategy to construct minimally rigid graphs from joining smaller ones.
Here in particular we are interested in gluing graphs on common subgraphs in such a way that we can tell the number of realizations of the new graph easily.
Fan constructions were introduced in \cite{BorceaStreinu2004} and in \cite{LowerBounds} a generalization of the fan construction was proposed.
For this purpose we take a minimally rigid graph $G=(V,E)$ with a minimally rigid subgraph $H=(W,F)$ of $G$.
Then we glue $k$ copies of $G$ along the subgraph and obtain a fan consisting of
$|W|+k(|V|-|W|)$ vertices (see \Cref{fig:fan} for an illustration).
The number of realizations of this graph is $\lamII{H}\cdot(\lamII{G}/\lamII{H})^k$ as a consequence of the following more general results,
for which we do not claim originality but have not seen them written anywhere proven.
\begin{lemma}\label{lem:subrc}
	Let $G$ be a minimally $d$-rigid graph and $H$ a minimally $d$-rigid proper subgraph with at least $d$ vertices.
	Then $\lam{d}{H}$ divides $\lam{d}{G}$.
\end{lemma}
\begin{proof}
	The number of realizations corresponds to the degree of the rigidity map $f_{G,d}$ of a graph $G$ (compare \cite{PlaneCount,PlaneSphere}).
	Let $\bar f_{G,d}\colon \C^{d|V(G)|} \longrightarrow \C^{|E(G)|} $, i.\,e.,
	$\bar f_{G,d}(p)=(\frac{1}{2}||p_v-p_w||^2)_{vw\in E(G)}$, where $||x||^2$ denotes an extension of the squared Euclidean norm to $\C^d$,
	by $||x||^2:=\sum_{i=1}^d x_i^2$, for $x=(x_1,\ldots,x_d)\in \C^d$.
	This map has infinite fibers due to rotations and translations of a given realization.
	In order to avoid them we fix some of the coordinates to be $0$.
	Let $v_1,\ldots,v_d\in V(G)$. Following the notation of \cite{PlaneSphere} we define
	\begin{align*}
		X_{G,d}:=\sst{p\in \C^{d|V(G)|}}{p(v_k)_i=0 \text{ for } j\geq k}.
	\end{align*}
	The $X_{G,d}$ has dimension $d|V(G)|-\binom{d+1}{2}$.
	By \cite[Lem~3.2]{PlaneSphere} every realization of $G$ has a congruent representative in $X_{G,d}$.
	Let now $f_{G,d}\colon X_{G,d} \longrightarrow \C^{|E(G)|} $ with $f_{G,d}(p)=\bar f_{G,d}(p)$.
	We fix $v_1,\ldots,v_d\in V(H)$ and consider the following commutative diagram.
	\begin{center}
		\begin{tikzpicture}
			\node[] (v) at (-2.5,1) {$X_{G,d}$};
			\node[] (e) at (2.5,1) {$\C^{|E(G)|}$};
			\node[] (vs) at (-2.5,-1) {$X_{H,d}\times \C^{|E(G\setminus H)|}$};
			\node[] (es) at (2.5,-1) {$\C^{|E(H)|}\times \C^{|E(G\setminus H)|}$};
			\draw[-{Classical TikZ Rightarrow}] (v) to node[above,labelsty] {$f_{G,d}$} (e);
			\draw[-{Classical TikZ Rightarrow}] (vs) to node[below,labelsty] {$(f_{H,d},\id)$} (es);
			\draw[-{Classical TikZ Rightarrow}] (v) to node[left,labelsty] {$\pi$} (vs);
			\draw[-{Classical TikZ Rightarrow}] (e) to node[right,labelsty] {$\id$} (es);
		\end{tikzpicture}
	\end{center}
	Here $\pi$ maps a realization of $G$ to its restriction to $H$ and additionally remembers the edge lengths for all the edges not in $H$.
	The $v_i$ are in $H$ and hence the restriction of a realization to $H$ is in $X_{H,d}$.
	Since $G$ and $H$ are minimally $d$-rigid, we know that $f_{G,d}$ and $f_{H,d}$ are dominant and hence generically finite (compare \cite[Lem~3.3, Prop.~3.5]{PlaneSphere}).
	Also $\pi$ is finite.
	Since the diagram is commutative we have $f_{G,d}=\pi\circ (f_{H,d},\id)$ and hence, the degree of $f_{G,d}$ is divisible by the degree of $f_{H,d}$.
	By the correspondence of the degrees of these maps to the realization count, $\lam{d}{G}$ divides $\lam{d}{H}$.
\end{proof}
From this we conclude how to count realizations of a graph obtained from gluing.
\begin{theorem}\label{thm:fanformula}
	Let $G$ be a minimally $d$-rigid graph obtained by a gluing operation with $k$ copies of a minimally $d$-rigid graph $G'$ on a minimally $d$-rigid subgraph $H$ of $G'$ with at least $d$ vertices.
	Then $\lam{d}{G}=\lam{d}{H}\cdot \left( \frac{\lam{d}{G'}}{\lam{d}{H}}\right)^k$.
\end{theorem}
\begin{proof}
	Let us assume $G$ is constructed by gluing $G'$ with $G''$ on a common minimally $d$-rigid subgraph $H$.
	By \Cref{lem:subrc} we know that $\lam{d}{G}=\lambda \lam{d}{G''}$ for some integer $\lambda>0$ and $\lam{d}{G'}=\alpha \lam{d}{H}$ for some integer $\alpha>0$.
	Indeed we know that $\lambda=\alpha$ since $G\setminus G''=G'\setminus H$ and therefore the map $\pi$ from the proof of \Cref{lem:subrc} has the same degree in both cases.
	Hence, we get
	\begin{align*}
		\lam{d}{G}=\frac{\lam{d}{G}}{\lam{d}{G'}}\lam{d}{G'}= \frac{\lam{d}{G}}{\alpha} \frac{\lam{d}{G'}}{\lam{d}{H}}= \frac{\lambda}{\alpha} \lam{d}{G''} \frac{\lam{d}{G'}}{\lam{d}{H}}
		 = \lam{d}{G''} \frac{\lam{d}{G'}}{\lam{d}{H}} = \lam{d}{H} \frac{\lam{d}{G''}}{\lam{d}{H}} \frac{\lam{d}{G'}}{\lam{d}{H}}.
	\end{align*}
	When $G''=G'$ we are done. Otherwise we do the same inductively on $G''$.
\end{proof}

Using the fan construction together with the previous lemmas, we get the following bound for $\maxlamIIn$.
This comprises the gluing of $k$ copies of $G$ at $H$ and some 0-extensions to obtain the required number of vertices.
A 0-extension adds a vertex and two edges (see \Cref{sec:factors} for details) that increases the number of realizations by a factor of 2.
\begin{equation}\label{eq:bound_genfan}
	\maxlam{d}{n} \geq 2^{(n-|W|)\modop(|V|-|W|)} \cdot\lam{d}{H} \cdot
	\left(\frac{\lam{d}{G}}{\lam{d}{H}}\right)^{\!\lfloor(n-|W|)/(|V|-|W|)\rfloor} \qquad (n\geq|W|).
\end{equation}

\begin{figure}[ht]
	\centering
	\begin{tikzpicture}[scale=0.5]
		\node[vertex] (a) at (-1,0) {};
		\node[vertex] (b) at (1,0) {};
		\node[vertex] (c) at (0,0.75) {};
		\node[vertex] (d) at ($(a)+(130:6)$) {};
		\node[vertex] (e) at ($(b)+(130:6)$) {};
		\node[vertex] (f) at ($0.5*(d)+0.5*(e)-(0,0.75)$) {};
		\node[vertex] (g) at ($(a)+(105:6)$) {};
		\node[vertex] (h) at ($(b)+(105:6)$) {};
		\node[vertex] (i) at ($0.5*(g)+0.5*(h)-(0,0.75)$) {};
		\node[vertex] (j) at ($(a)+(75:6)$) {};
		\node[vertex] (k) at ($(b)+(75:6)$) {};
		\node[vertex] (l) at ($0.5*(j)+0.5*(k)-(0,0.75)$) {};
		\node[vertex] (m) at ($(a)+(50:6)$) {};
		\node[vertex] (n) at ($(b)+(50:6)$) {};
		\node[vertex] (o) at ($0.5*(m)+0.5*(n)-(0,0.75)$) {};
		\draw[edge] (a)edge(b) (b)edge(c) (c)edge(a);
		\draw[edge] (d)edge(e) (e)edge(f) (f)edge(d);
		\draw[edge] (d)edge(a) (e)edge(b) (f)edge(c);
		\draw[edge] (g)edge(h) (h)edge(i) (i)edge(g);
		\draw[edge] (g)edge(a) (h)edge(b) (i)edge(c);
		\draw[edge] (j)edge(k) (k)edge(l)	(l)edge(j);
		\draw[edge] (j)edge(a) (k)edge(b) (l)edge(c);
		\draw[edge] (m)edge(n) (n)edge(o) (o)edge(m);
		\draw[edge] (m)edge(a) (n)edge(b) (o)edge(c);
	\end{tikzpicture}
	\caption{Fan construction with four copies of the three-prism graph glued on a triangle subgraph.}
	\label{fig:fan}
\end{figure}
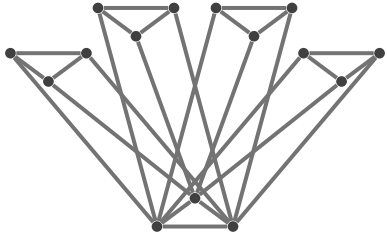

Note that the caterpillar and fan constructions from \cite{BorceaStreinu2004} are special instance of the
generalized fan.
To indicate the subgraph $H$ used in a generalized fan construction we also write $H$-fan.
Using our encoding for graphs the usual fan would be denoted by $7$-fan.
The fan fixing the 4-vertex minimally rigid graph is then denoted by $31$-fan.
Further encodings of small minimally rigid graphs can be found in \Cref{fig:smallm2rinteger}.

\section{Realizations in the plane}
\label{sec:dim-2}
The algorithmic result from \cite{PlaneCount} and its implementation \cite{ZenodoAlg,CapcoCpp} allowed large scale computations of $\lamIIfunc$ for many graphs.
A particular interest lies in graphs with a high number of realizations.
Using fan constructions, one can then construct a lower bound for the maximal number of realizations for graphs with a given number of vertices.
In the following we use these fans to derive new and better lower bounds than the
previously known ones.

Lists of all minimally rigid graphs with at most 13 vertices have been computed, so we do know also $\lamII{G}$ for every such $G$ and hence, $\maxlamIIn$ for $n\leq13$ is known (see also \cite{A306420}).
Indeed also for $n=14$ the maximum is known but there are too many graphs to keep all the list stored.
The graphs with maximum number of realizations up to $n=12$ are shown in \cite{LowerBounds}, the ones with $n=13$ and $n=14$ are listed in the Appendix (\Cref{tab:enc:1-fan}).
For higher number of vertices it is not efficiently possible to compute realization numbers for all the minimally rigid graphs because of the huge amount of them.
Nevertheless, the implementation of \cite{ZenodoAlg} allows for computing the number of realizations for single graphs up to 24 vertices, where the latter already takes a several days even with some parallelization.
Overall we have computed the realization count for several million graphs with 15 or more vertices.
From a theoretical point of view there is no limit to compute more but time and memory resources play an important role.
\Cref{tab:enc:1-fan} also lists graphs with currently largest known number of realizations for $n\leq23$.
These may be used with a 1-fan to get lower bounds for $\maxlamII{n}$. More generally \Cref{tab:bounds} summarizes the best growth rates that we get from the computations we did.
By growth rate we mean $\left(\lamII{G}/\lamII{H}\right)^{1/(|V|-|W|)}$, which gives the essential part of \labelcref{eq:bound_genfan}.
In \Cref{fig:growth} these numbers are illustrated.
Note, that we did a few computations for larger graphs as well but did not gain any better values.
The table and the figure therefore only contain values where some gain was obtained compared to fewer vertices.

\begin{table}[ht]
	\centering\small
	\begin{tabular}{rllllll}
		\toprule
		$n$ & 1-fan   & 7-fan   & 31-fan  & 254-fan & 7916-fan\\\midrule
		12 & 2.39386 & 2.43198 & 2.43006 & 2.39802 & 2.35824 \\
		13 & 2.40453 & 2.44498 & 2.46039 & 2.43006 & 2.39802 \\\hline
		14 & 2.43185 & 2.46092 & 2.46391 & 2.46039 & 2.43006 \\
		15 & 2.44695 & 2.48236 & 2.48167 & 2.46700 & 2.46039 \\
		16 & 2.46890 & 2.49802 & 2.49854 & 2.49200 & 2.46700 \\
		17 & 2.49019 & 2.51717 & 2.51269 & 2.50558 & 2.49200 \\
		18 & 2.50568 & 2.53153 & 2.53183 & 2.52210 & 2.50558 \\
		19 & 2.51640 & 2.54276 & 2.54671 & 2.54188 & 2.52210 \\
		20 & 2.52948 & 2.55428 & 2.55823 & 2.55640 & 2.54188 \\
		21 & 2.54120 & 2.56587 & 2.56829 & 2.56852 & 2.55640 \\
		22 & 2.55351 & 2.57682 & 2.57995 & 2.57690 & 2.56852 \\
		23 & 2.55643 & 2.58550 & 2.58912 & 2.59025 & 2.57690 \\
		24 &         & 2.58648 & 2.59829 & 2.59972 & 2.59025\\
		\bottomrule
	\end{tabular}
	\caption{Growth rates (rounded) of the lower bounds using different fan constructions.
			For $n\leq13$ these values are proven to be the best achievable ones.
			For $n>13$ the values are just the best we found, so it is possible that there are better ones.\\
			The encodings for the graphs can be found in \Cref{tab:enc:1-fan,tab:enc:7-fan,tab:enc:31-fan,tab:enc:254-fan,tab:enc:7916-fan}.
	}
	\label{tab:bounds}
\end{table}

In contrast to \cite{LowerBounds}, where the best bound was achieved by a 7-fan construction, we get the best bound now by a 254-fan.
Only slightly worse is the best 31-fan we found. The more general 7916-fan does seem to give shifted results from the 254-fan (and similar for more general fans not shown here) but this might be misleading due the small number of vertices and the choice of graphs in the sample sets.

\begin{figure}[ht]
	\centering
	\begin{tikzpicture}[yscale=0.9]
		\begin{scope}[yscale=20,xscale=0.6]
			\coordinate (o) at (5,2.26);
			\draw[cola,rounded corners] (2.5,2.22) rectangle (25.9,2.64);
			\node[alabelsty,rotate=90] at ($(0,2.26)!0.5!(0,2.59)+(3.25,0)$) {growth rate};
			\node[alabelsty] at ($(6,0)!0.5!(24,0)+(0,2.232)$) {$n$};
			\foreach \y in {2.3,2.4,2.5,2.6} \draw[bline] (5,\y) -- +(-0.2,0) node[left,alabelsty] {$\y$};
			\foreach \x in {6,...,24} \draw[bline] (\x,2.27) -- +(0,-0.005) node[below,alabelsty] {$\x$};

			\foreach \y [count=\i] in {2.27,2.28,...,2.6}
			{
				\draw[bline] (5,\y)--(25,\y);
			}

			\draw[indline] (13,2.27) -- (13,2.61) node[alabelsty,anchor=east] {\tikz{\draw[aline,solid,-{LaTeX[round]}] (0,0)--(-0.5,0);} complete};
			\draw[indline] (18,2.27) -- (18,2.61);
			\draw[indline] (5,2.50798) -- (25.3,2.50798) node[right,alabelsty,anchor=west,rotate=90] {new \tikz{\draw[aline,solid,-{LaTeX[round]}] (0,0)--(0.5,0);}};

			\draw[colR,gline]
				(8, 2.26772)
			-- (9, 2.30338)
			-- (10, 2.33378)
			-- (11, 2.36196)
			-- (12, 2.39386)
			-- (13, 2.40453)
			-- (14, 2.43185);
			\draw[colRw,gline]
				(14, 2.43185)
			-- (15, 2.44695)
			-- (16, 2.4689)
			-- (17, 2.49019)
			-- (18, 2.50568)
			-- (19, 2.5164)
			-- (20, 2.52948)
			-- (21, 2.5412)
			-- (22, 2.55351)
			-- (23, 2.55784);
			\node[valueR] at (8, 2.26772) {};
			\node[valueR] at (9, 2.30338) {};
			\node[valueR] at (10, 2.33378) {};
			\node[valueR] at (11, 2.36196) {};
			\node[valueR] at (12, 2.39386) {};
			\node[valueR] at (13, 2.40453) {};
			\node[valueR] at (14, 2.43185) {};
			\node[valueRw] at (15, 2.44695) {};
			\node[valueRw] at (16, 2.4689) {};
			\node[valueRw] at (17, 2.49019) {};
			\node[valueRw] at (18, 2.50568) {};
			\node[valueRw] at (19, 2.5164) {};
			\node[valueRw] at (20, 2.52948) {};
			\node[valueRw] at (21, 2.5412) {};
			\node[valueRw] at (22, 2.55351) {};
			\node[valueRw] at (23, 2.55784) {};
			\draw[colB,gline]
				(6, 2.28943)
			-- (7, 2.30033)
			-- (8, 2.32542)
			-- (9, 2.35824)
			-- (10, 2.38581)
			-- (11, 2.41159)
			-- (12, 2.43198)
			-- (13, 2.44498);
			\draw[colBw,gline]
				(13, 2.44498)
			-- (14, 2.46092)
			-- (15, 2.48236)
			-- (16, 2.49802)
			-- (17, 2.51717)
			-- (18, 2.53153)
			-- (19, 2.54276)
			-- (20, 2.55428)
			-- (21, 2.56587)
			-- (22, 2.57682)
			-- (23, 2.5855)
			-- (24, 2.58648);
			\node[valueB] at (6, 2.28943) {};
			\node[valueB] at (7, 2.30033) {};
			\node[valueB] at (8, 2.32542) {};
			\node[valueB] at (9, 2.35824) {};
			\node[valueB] at (10, 2.38581) {};
			\node[valueB] at (11, 2.41159) {};
			\node[valueB] at (12, 2.43198) {};
			\node[valueB] at (13, 2.44498) {};
			\node[valueBw] at (14, 2.46092) {};
			\node[valueBw] at (15, 2.48236) {};
			\node[valueBw] at (16, 2.49802) {};
			\node[valueBw] at (17, 2.51717) {};
			\node[valueBw] at (18, 2.53153) {};
			\node[valueBw] at (19, 2.54276) {};
			\node[valueBw] at (20, 2.55428) {};
			\node[valueBw] at (21, 2.56587) {};
			\node[valueBw] at (22, 2.57682) {};
			\node[valueBw] at (23, 2.5855) {};
			\node[valueBw] at (24, 2.58648) {};
			\draw[colG,gline]
				(7, 2.28943)
			-- (8, 2.30033)
			-- (9, 2.35216)
			-- (10, 2.35824)
			-- (11, 2.38581)
			-- (12, 2.43006)
			-- (13, 2.46039);
			\draw[colGw,gline]
				(13, 2.46039)
			-- (14, 2.46391)
			-- (15, 2.48167)
			-- (16, 2.49854)
			-- (17, 2.51269)
			-- (18, 2.53183)
			-- (19, 2.54671)
			-- (20, 2.55823)
			-- (21, 2.56829)
			-- (22, 2.57995)
			-- (23, 2.58912)
			-- (24, 2.59829);
			\node[valueG] at (7, 2.28943) {};
			\node[valueG] at (8, 2.30033) {};
			\node[valueG] at (9, 2.35216) {};
			\node[valueG] at (10, 2.35824) {};
			\node[valueG] at (11, 2.38581) {};
			\node[valueG] at (12, 2.43006) {};
			\node[valueG] at (13, 2.46039) {};
			\node[valueGw] at (14, 2.46391) {};
			\node[valueGw] at (15, 2.48167) {};
			\node[valueGw] at (16, 2.49854) {};
			\node[valueGw] at (17, 2.51269) {};
			\node[valueGw] at (18, 2.53183) {};
			\node[valueGw] at (19, 2.54671) {};
			\node[valueGw] at (20, 2.55823) {};
			\node[valueGw] at (21, 2.56829) {};
			\node[valueGw] at (22, 2.57995) {};
			\node[valueGw] at (23, 2.58912) {};
			\node[valueGw] at (24, 2.59829) {};
			\draw[colP,gline]
				(8, 2.28943)
			-- (9, 2.30033)
			-- (10, 2.35216)
			-- (11, 2.35824)
			-- (12, 2.39802)
			-- (13, 2.43006);
			\draw[colPw,gline]
				(13, 2.43006)
			-- (14, 2.46039)
			-- (15, 2.467)
			-- (16, 2.492)
			-- (17, 2.50558)
			-- (18, 2.5221)
			-- (19, 2.54188)
			-- (20, 2.5564)
			-- (21, 2.56852)
			-- (22, 2.5769)
			-- (23, 2.59025)
			-- (24, 2.59972);
			\node[valueP] at (8, 2.28943) {};
			\node[valueP] at (9, 2.30033) {};
			\node[valueP] at (10, 2.35216) {};
			\node[valueP] at (11, 2.35824) {};
			\node[valueP] at (12, 2.39802) {};
			\node[valueP] at (13, 2.43006) {};
			\node[valuePw] at (14, 2.46039) {};
			\node[valuePw] at (15, 2.467) {};
			\node[valuePw] at (16, 2.492) {};
			\node[valuePw] at (17, 2.50558) {};
			\node[valuePw] at (18, 2.5221) {};
			\node[valuePw] at (19, 2.54188) {};
			\node[valuePw] at (20, 2.5564) {};
			\node[valuePw] at (21, 2.56852) {};
			\node[valuePw] at (22, 2.5769) {};
			\node[valuePw] at (23, 2.59025) {};
			\node[valuePw] at (24, 2.59972) {};
			\draw[colX,gline]
				(9, 2.28943)
			-- (10, 2.30033)
			-- (11, 2.35216)
			-- (12, 2.35824)
			-- (13, 2.39802);
			\draw[colXw,gline]
				(13, 2.39802)
			-- (14, 2.43006)
			-- (15, 2.46039)
			-- (16, 2.467)
			-- (17, 2.492)
			-- (18, 2.50558)
			-- (19, 2.5221)
			-- (20, 2.54188)
			-- (21, 2.5564)
			-- (22, 2.56852)
			-- (23, 2.5769)
			-- (24, 2.59025);
			\node[valueX] at (9, 2.28943) {};
			\node[valueX] at (10, 2.30033) {};
			\node[valueX] at (11, 2.35216) {};
			\node[valueX] at (12, 2.35824) {};
			\node[valueX] at (13, 2.39802) {};
			\node[valueXw] at (14, 2.43006) {};
			\node[valueXw] at (15, 2.46039) {};
			\node[valueXw] at (16, 2.467) {};
			\node[valueXw] at (17, 2.492) {};
			\node[valueXw] at (18, 2.50558) {};
			\node[valueXw] at (19, 2.5221) {};
			\node[valueXw] at (20, 2.54188) {};
			\node[valueXw] at (21, 2.5564) {};
			\node[valueXw] at (22, 2.56852) {};
			\node[valueXw] at (23, 2.5769) {};
			\node[valueXw] at (24, 2.59025) {};
		\end{scope}

		\begin{scope}
			\coordinate (s) at (0,-0.3);
			\coordinate (l) at ($(o)+(10.5,3)$);
			\draw[cola,fill=white,rounded corners] ($(l)+(-0.25,0.3)$) rectangle ++(1.65,-1.8);
			\node[valueR,legend,label={[alabelsty]0:1-fan}] at (l) {};
			\node[valueB,legend,label={[alabelsty]0:7-fan}] at ($(l)+(s)$) {};
			\node[valueG,legend,label={[alabelsty]0:31-fan}] at ($(l)+2*(s)$) {};
			\node[valueP,legend,label={[alabelsty]0:254-fan}] at ($(l)+3*(s)$) {};
			\node[valueX,legend,label={[alabelsty]0:7916-fan}] at ($(l)+4*(s)$) {};
		\end{scope}
	\end{tikzpicture}
	\caption{Growth rates of the lower bounds.
		The light colors indicate values that were not found by exhaustive search and which therefore could possibly
		be improved. The horizontal dashed line indicates the lower bound known from \cite{LowerBounds}.
		The second vertical dashed line indicates the number of vertices considered in \cite{LowerBounds}
		with the respective bound found therein (horizontal dashed line).}
	\label{fig:growth}
\end{figure}

From \Cref{tab:bounds} we can see the previous bound $2.28943^n$
obtained in \cite{BorceaStreinu2004}, $2.30033^n$ from \cite{EmirisMoroz},
$2.41159^n$ from~\cite{Jackson2018}, and $2.50798^n$ from \cite{LowerBounds} as well as the current
improvements obtained in this paper.
By instantiating \labelcref{eq:bound_genfan} with the minimally rigid graph encoded by
\begin{align*}
	1621117988222861364696506244132890911843704707085263796691486054849687756801808,
\end{align*}
which has 24 vertices, $\lamII{G}=611930960$ realizations and a five vertex minimally rigid subgraph 254 for gluing,
we obtain the following theorem.
\begin{theorem}
	\label{thm:lower-bd-2d}
	The maximal number of realizations $\maxlamIIn$, for $n\geq 5$, satisfies
	\begin{equation*}
		\maxlamIIn \geq 8\cdot2^{(n-5)\modop 19}\cdot(611930960/8)^{\lfloor(n-5)/19\rfloor}.
	\end{equation*}
	This means $\maxlamIIn$ grows at least as $\bigl(\!\sqrt[19]{611930960/8}\bigr)^n$, which is approximately $2.59972^n$.
	In other words $\bigl(\!\sqrt[19]{76491370}\bigr)^n\in\mathcal O(\maxlamIIn)$.
\end{theorem}
Note that we are talking about complex realizations here. A respective bound for real realizations $\sim2.378^n$ can be found in \cite{Bartzos2018}.
To our knowledge the best known upper bound for $\maxlamIIn$ is approximately $3.4641^n$ \cite{Bartzos23}.

In \cite{LowerBounds} a set of properties, including symmetries, were presented which were satisfied by the graphs with what was back then the maximal number of realizations found.
These properties are no longer satisfied by the currently known maxima.
Still, we do not know whether we even found the graphs with the maximum number of realizations for any $n\geq 15$.
Nevertheless, the currently best known graphs do have some properties in common:
\begin{itemize}
	\item Minimum degree is 3 and maximum degree is 4, i.\,e.\ there are exactly six vertices of degree 3.
	\item No two vertices of degree 3 are adjacent.
	\item For $n\geq12$ the graph does not contain a 3-cycle subgraph.
	\item For $n\geq12$ the graph is non-planar.
	\item The graph has chromatic number 3.
	\item The graph is Hamiltonian.
	\item For each degree 3 vertex non of its neighbors are connected, i.\,e.\ in the construction of the graph a 1-extension of type E1c is needed (see \Cref{sec:factors}).
\end{itemize}
These properties however might be biased by the set of graphs for which the computations have been done.
It is unclear so far on whether a graph with $\lamII{G}=\maxlamII{G}$ would need to have any of these properties indeed.

In \cite{LowerBounds,Jackson2018} it was conjectured that $\lamII{G}\geq 2^{n-2}$. While this conjecture remains unproven, recent computations have shown that it is true at least for graphs with at most 13 vertices and non of the larger examples we computed would contradict it either.
\begin{corollary}\label{cor:lowerII}
	Let $G$ be a minimally rigid graph with $n\leq 13$ vertices. Then $\lamII{G}\geq 2^{n-2}$.
\end{corollary}
Minimally rigid graphs that can be obtained from a single edge by just using 0-extensions always have $\lamII{G}=2^{n-2}$ (compare also \Cref{sec:factors} for the definitions of $k$-extensions).
But also other graphs may have the same realization count. In \Cref{tab:low} we show how many of them there are overall and how many with minimum degree 3.
The latter are graphs that need at least a 1-extension to be constructed. They are somehow more interesting since for all other graphs we can delete a degree 2 vertex (so called 0-reduction), compute $\lamIIfunc$ for the smaller graph and multiply by two to get the realization count of the original graph.
\begin{table}[ht]
	\centering\small
	\begin{tabular}[t]{rrrl}
		\toprule
		$n$ & $|\setMn|$ & $|\lamset{2}{n}{2^{n-2}}|$ & $|\lamset{2}{n}{2^{n-2}}|/|\setMn|$ \\\midrule
		7   &         70 &       64 & \bnrs{0.914286}{2}{colG!50!white}\\
		8   &        608 &      525 & \bnrs{0.863487}{2}{colG!50!white}\\
		9   &       7222 &     5826 & \bnrs{0.806702}{2}{colG!50!white}\\
		10  &     110132 &    80912 & \bnrs{0.734682}{2}{colG!50!white}\\
		11  &    2039273 &  1338956 & \bnrs{0.656585}{2}{colG!50!white}\\
		12  &   44176717 & 25551013 & \bnrs{0.578382}{2}{colG!50!white}\\
		\bottomrule
	\end{tabular}
	\begin{tabular}[t]{rrrl}
		\toprule
		$n$ & $|\setDn|$ & $|\lamset{2}{n}{2^{n-2}}\cap \setDn|$ & $|\lamset{2}{n}{2^{n-2}}\cap\setDn|/|\setDn|$\\\midrule
		7   &        4 &      1 & \bnrs{0.25000}{5}{colB!50!white}\\
		8   &       32 &      6 & \bnrs{0.18750}{5}{colB!50!white}\\
		9   &      264 &     39 & \bnrs{0.14773}{5}{colB!50!white}\\
		10  &     3189 &    307 & \bnrs{0.09627}{5}{colB!50!white}\\
		11  &    46677 &   2867 & \bnrs{0.06142}{5}{colB!50!white}\\
		12  &   813875 &  30789 & \bnrs{0.03783}{5}{colB!50!white}\\
		13  & 16142835 & 374297 & \bnrs{0.02319}{5}{colB!50!white}\\
		\bottomrule
	\end{tabular}
	\caption{The number of graphs with $\lamII{G}=2^{n-2}$, i.\,e.\ $|\lamset{2}{n}{2^{n-2}}|$ and $|\lamset{2}{n}{2^{n-2}}\cap\setDn|$ for given number of vertices compared to the number of minimally rigid graphs with minimum degree three $|\setDn|$.}
	\label{tab:low}
\end{table}

\minisec{On Real Realizations}
In applications it is usually more interesting to know the number of real realizations.
The algorithms we are using here do not provide much information on these.
Clearly $\lamrII{G}\leq\lamII{G}$ for every minimally rigid graph.
It can be easily seen that graphs which are obtained by only 0-extensions, do have $\lamrII(G)=\lamII{G}$.
This however is not true in general for other minimally rigid graphs.
In \cite{Jackson2018} it was shown that there are graphs where $\lamrII(G)\neq\lamII{G}$;
in particular those graphs which have $4\nmid\lamII{G}$.
Note that the count in \cite{Jackson2018} is half of the count we use here.
We define the set of graphs with $n$ vertices and $4\nmid\lamII{G}$ to be $\mathcal C^4_n$.
\Cref{tab:lreal} shows how many of these graphs exist for small number of vertices. We can see that there are comparably few of them.
However, these are for sure not the only graphs for which $\lamrII(G)$ and $\lamII{G}$ differ.
For instance any 0-extension of such a graph would fulfill $4\mid\lamII{G}$ but it still has $\lamrII(G)\neq\lamII{G}$.
\begin{table}[ht]
	\centering\small
	\begin{tabular}{rrrl}
		\toprule
		$n$ & $|\setDn|$ & $|\mathcal C^4_n|$ & $|\mathcal C^4_n|/|\setDn|$\\\midrule
		8   &       32 &     1 & \bnrs{0.0313}{100}{colG!50!white}\\
		9   &      264 &     3 & \bnrs{0.0114}{100}{colG!50!white}\\
		10  &     3189 &    13 & \bnrs{0.0041}{100}{colG!50!white}\\
		11  &    46677 &   153 & \bnrs{0.0033}{100}{colG!50!white}\\
		12  &   813875 &  2077 & \bnrs{0.0026}{100}{colG!50!white}\\
		13  & 16142835 & 35858 & \bnrs{0.0022}{100}{colG!50!white}\\
		\bottomrule
	\end{tabular}
	\caption{The number of minimally rigid graphs with $4\nmid\lamII{G}$ for given number of vertices, $\mathcal C^4_n$, compared to the number of minimally rigid graphs with minimum degree three $\setDn$.}
	\label{tab:lreal}
\end{table}

In \cite[Thm~4.9]{Jackson2018} it was shown that for planar graphs we have $\lamrII(G)\geq 2^{n-2}$, where $n$ is the number of vertices of $G$.
From this we get
\begin{corollary}
	Let $G$ be a planar minimally rigid graph with $n$ vertices and $\lamII{G}=2^{n-2}$. Then $\lamrII(G)=\lamII{G}$.
\end{corollary}
In \Cref{tab:planarreal} we see how many graphs there are in $\lamset{d}{n}{2^{n-2}}\cap \mathcal P_n$, i.e.\ those graphs of $\mathcal P_n$ that have $\lamII{G}=2^{n-2}$.
\begin{table}[ht]
	\centering\small
	\begin{tabular}{rrrrl}
		\toprule
		$n$ & $|\setDn|$ & $|\mathcal P_n|$ & $|\lamset{d}{n}{2^{n-2}}\cap \mathcal P_n|$ & $|\lamset{d}{n}{2^{n-2}}\cap \mathcal P_n|/|\setDn|$\\\midrule
		7   &        4 &       3 &      1 & \bnrs{0.25000}{10}{colG!50!white}\\
		8   &       32 &      18 &      5 & \bnrs{0.15625}{10}{colG!50!white}\\
		9   &      264 &     122 &     31 & \bnrs{0.11742}{10}{colG!50!white}\\
		10  &     3189 &    1037 &    213 & \bnrs{0.06679}{10}{colG!50!white}\\
		11  &    46677 &    9884 &   1677 & \bnrs{0.03593}{10}{colG!50!white}\\
		12  &   813875 &  101848 &  14071 & \bnrs{0.01729}{10}{colG!50!white}\\
		13  & 16142835 & 1098726 & 124277 & \bnrs{0.00770}{10}{colG!50!white}\\
		\bottomrule
	\end{tabular}
	\caption{The number of graphs planar graphs $|\mathcal P_n|$ and out of those the ones with $\lamII{G}=2^{n-2}$, i.\,e.\ $|\lamset{d}{n}{2^{n-2}}\cap \mathcal P_n|$ for given number of vertices compared to the number of minimally rigid graphs with minimum degree three $|\setDn|$.}
	\label{tab:planarreal}
\end{table}

\section{Realizations on the sphere}
\label{sec:sphere}
It is known that the minimally rigid graphs on the sphere are exactly the minimally 2-rigid graphs in the plane (compare \cite{Eftekhari2019}).
However, counting the number of realizations is different.
In \cite{SphereCount} a combinatorial algorithm was presented to compute the number of complex realizations on the sphere.
This algorithm is faster than using Gröbner bases but still exponential.
Implementations can be found in \cite{SphereAlg,RigiComp,CapcoCpp}.
Now we present first major computational results from this algorithm.
For instance we were able to compute $\lamSIIfunc$ for all minimally rigid graphs with minimum degree 3 and at most 13 vertices.
Hence, with $n\leq 13$ we know the value of $\maxlamSIIn$ precisely.
For higher number of vertices we have experimental results using a large data set on minimally rigid graphs.
Unlike in the plane the graphs with maximal realization count seem to have less properties in common.
In \cite{SphereCount} it was already shown that a graph $G$ with $\lamSII{G}=\maxlamSII{n}$ is not unique, i.\,e., $\maxlamSIIset{n}$ may have more than one element.
For instance there are 5 graphs in $\maxlamSIIset{7}$, where $\lamSII{G}=64=\maxlamSII{7}$.
In \Cref{fig:maxsphere} we show the graphs with maximal realization count on the sphere for $n\in\{10,11,12,13\}$.

\begin{figure}[ht]
	\centering
	\begin{tabular}{m{0.9cm}m{10cm}c}
		\toprule
		n & graph(s) in $\maxlamSset{2}{n}$ & $\lamSIIfunc$ \\\midrule\addlinespace[10pt]
		10 &
		\begin{tikzpicture}[scale=0.7]
			\node[vertex] (1) at (2, 1) {};
			\node[vertex] (2) at (0, 1) {};
			\node[vertex] (3) at (0.5, 1.5) {};
			\node[vertex] (4) at (2.5, 1.5) {};
			\node[vertex] (5) at (1, 0) {};
			\node[vertex] (6) at (1.5, 0.6) {};
			\node[vertex] (7) at (0.5, 0.5) {};
			\node[vertex] (8) at (2.5, 0.5) {};
			\node[vertex] (9) at (1.5, 1.6) {};
			\node[vertex] (10) at (1, 1.1) {};
			\draw[edge] (1)edge(4) (1)edge(8) (1)edge(10) (2)edge(3) (2)edge(7) (2)edge(10) (3)edge(7) (3)edge(9) (4)edge(8) (4)edge(9) (5)edge(7) (5)edge(8) (5)edge(10) (6)edge(7) (6)edge(8) (6)edge(9) (9)edge(10);
		\end{tikzpicture}
		\quad
		\begin{tikzpicture}[scale=0.7]
			\node[vertex] (1) at (1.5, 0.5) {};
			\node[vertex] (2) at (0, 1) {};
			\node[vertex] (3) at (0.5, 1.5) {};
			\node[vertex] (4) at (2, 1) {};
			\node[vertex] (5) at (0.5, -0.5) {};
			\node[vertex] (6) at (1, 0.1) {};
			\node[vertex] (7) at (1, 1.1) {};
			\node[vertex] (8) at (2, 0) {};
			\node[vertex] (9) at (0.5, 0.5) {};
			\node[vertex] (10) at (0, 0) {};
			\draw[edge] (1)edge(4) (1)edge(8) (1)edge(9) (2)edge(3) (2)edge(7) (2)edge(10) (3)edge(7) (3)edge(9) (4)edge(7) (4)edge(8) (5)edge(8) (5)edge(9) (5)edge(10) (6)edge(7) (6)edge(8) (6)edge(10) (9)edge(10);
		\end{tikzpicture}
		\quad
		\begin{tikzpicture}[scale=0.7]
			\node[vertex] (1) at (0, 1.75) {};
			\node[vertex] (2) at (0.75, 0.75) {};
			\node[vertex] (3) at (-1.25, 0.1) {};
			\node[vertex] (4) at (-0.75, 0.75) {};
			\node[vertex] (5) at (1.25, 0.1) {};
			\node[vertex] (6) at (0, 0.5) {};
			\node[vertex] (7) at (-0.5, 0) {};
			\node[vertex] (8) at (0.5, 0) {};
			\node[vertex] (9) at (1, 1.25) {};
			\node[vertex] (10) at (-1, 1.25) {};
			\draw[edge] (1)edge(6) (1)edge(9) (1)edge(10) (2)edge(5) (2)edge(8) (2)edge(10) (3)edge(4) (3)edge(7) (3)edge(10) (4)edge(7) (4)edge(9) (5)edge(8) (5)edge(9) (6)edge(7) (6)edge(8) (7)edge(8) (9)edge(10);
		\end{tikzpicture}
		& 1536 \\[6ex]
		11 &
		\begin{tikzpicture}[scale=0.7]
			\node[vertex] (1) at (2.5, 1.5) {};
			\node[vertex] (2) at (2.5, 0.5) {};
			\node[vertex] (3) at (0.5, 1.5) {};
			\node[vertex] (4) at (0., 1) {};
			\node[vertex] (5) at (1.6, 0.) {};
			\node[vertex] (6) at (1.5, 1.5) {};
			\node[vertex] (7) at (1.5, 0.6) {};
			\node[vertex] (8) at (0.5, 0.5) {};
			\node[vertex] (9) at (1.6, 2) {};
			\node[vertex] (10) at (2, 1) {};
			\node[vertex] (11) at (1, 1.1) {};
			\draw[edge] (1)edge(2) (1)edge(9) (1)edge(10) (2)edge(7) (2)edge(10) (3)edge(4) (3)edge(8) (3)edge(9) (4)edge(8) (4)edge(11) (5)edge(7) (5)edge(8) (5)edge(10) (6)edge(7) (6)edge(9) (6)edge(11) (7)edge(8) (9)edge(11) (10)edge(11);
		\end{tikzpicture}
		\quad
		\begin{tikzpicture}[scale=0.7]
			\node[vertex] (1) at (2.7, 0) {};
			\node[vertex] (2) at (2.2, 0.3) {};
			\node[vertex] (3) at (0.5, 0) {};
			\node[vertex] (4) at (1, 0.5) {};
			\node[vertex] (5) at (1.4, 1.6) {};
			\node[vertex] (6) at (2.5, 1.6) {};
			\node[vertex] (7) at (1.5, 1.1) {};
			\node[vertex] (8) at (0.5, 1) {};
			\node[vertex] (9) at (1.5, 0.1) {};
			\node[vertex] (10) at (3, 0.5) {};
			\node[vertex] (11) at (1.6, 0.6) {};
			\draw[edge] (1)edge(2) (1)edge(9) (1)edge(10) (2)edge(7) (2)edge(9) (3)edge(4) (3)edge(8) (3)edge(9) (4)edge(8) (4)edge(11) (5)edge(7) (5)edge(8) (5)edge(10) (6)edge(7) (6)edge(10) (6)edge(11) (7)edge(8) (9)edge(11) (10)edge(11);
		\end{tikzpicture}
		& 4352 \\[6ex]
		12 &
		\begin{tikzpicture}[scale=0.7]
			\node[vertex] (1) at (2, 0) {};
			\node[vertex] (2) at (3, -1) {};
			\node[vertex] (3) at (0, 0) {};
			\node[vertex] (4) at (1.5, -0.5) {};
			\node[vertex] (5) at (3, 1) {};
			\node[vertex] (6) at (2.5, 0.5) {};
			\node[vertex] (7) at (1.5, 0.5) {};
			\node[vertex] (8) at (0.5, 0.6) {};
			\node[vertex] (9) at (3.1, 0) {};
			\node[vertex] (10) at (2.5, -0.5) {};
			\node[vertex] (11) at (0.5, -0.6) {};
			\node[vertex] (12) at (1, 0) {};
			\draw[edge] (1)edge(10) (1)edge(11) (1)edge(12) (2)edge(9) (2)edge(10) (2)edge(11) (3)edge(8) (3)edge(11) (3)edge(12) (4)edge(7) (4)edge(10) (4)edge(12) (5)edge(6) (5)edge(8) (5)edge(9) (6)edge(7) (6)edge(9) (7)edge(8) (7)edge(12) (8)edge(11) (9)edge(10);
		\end{tikzpicture}
		& 12288 \\[6ex]
		13 &
		\begin{tikzpicture}[scale=0.7]
			\node[vertex] (1) at (-0.5, 2) {};
			\node[vertex] (2) at (1, 2.5) {};
			\node[vertex] (3) at (0.5, 1.5) {};
			\node[vertex] (4) at (-0.5, 0.5) {};
			\node[vertex] (5) at (0, 1) {};
			\node[vertex] (6) at (0, 0) {};
			\node[vertex] (7) at (0, 2.5) {};
			\node[vertex] (8) at (0.5, 2) {};
			\node[vertex] (9) at (-1, 2.5) {};
			\node[vertex] (10) at (0.5, 0.5) {};
			\node[vertex] (11) at (1.25, 0.5) {};
			\node[vertex] (12) at (-1.25, 0.5) {};
			\node[vertex] (13) at (-0.5, 1.5) {};
			\draw[edge] (1)edge(7) (1)edge(9) (1)edge(12) (1)edge(13) (2)edge(7) (2)edge(8) (2)edge(10) (2)edge(11) (3)edge(5) (3)edge(8) (3)edge(11) (3)edge(13) (4)edge(6) (4)edge(9) (4)edge(10) (4)edge(12) (5)edge(6) (5)edge(12) (5)edge(13) (6)edge(10) (6)edge(11) (7)edge(8) (7)edge(9);
		\end{tikzpicture}
		\quad
		\begin{tikzpicture}[scale=0.7]
			\node[vertex] (1) at (0, 1.5) {};
			\node[vertex] (2) at (-0.6, 0) {};
			\node[vertex] (3) at (-0.6, 1) {};
			\node[vertex] (4) at (0.6, 0) {};
			\node[vertex] (5) at (0.6, 1) {};
			\node[vertex] (6) at (-0.6, 2) {};
			\node[vertex] (7) at (0.6, 2) {};
			\node[vertex] (8) at (-1.25, 1.75) {};
			\node[vertex] (9) at (1.25, 1.75) {};
			\node[vertex] (10) at (1.25, 1) {};
			\node[vertex] (11) at (-1.25, 1) {};
			\node[vertex] (12) at (-0.1, -0.5) {};
			\node[vertex] (13) at (0.1, 0.5) {};
			\draw[edge] (1)edge(6) (1)edge(7) (1)edge(12) (1)edge(13) (2)edge(3) (2)edge(4) (2)edge(11) (2)edge(12) (3)edge(5) (3)edge(8) (3)edge(13) (4)edge(5) (4)edge(10) (4)edge(12) (5)edge(9) (5)edge(13) (6)edge(7) (6)edge(8) (6)edge(11) (7)edge(9) (7)edge(10) (8)edge(11) (9)edge(10);
		\end{tikzpicture}
		\quad
		\begin{tikzpicture}[scale=0.7]
			\node[vertex] (1) at (0, 1.5) {};
			\node[vertex] (2) at (-0.75, 1) {};
			\node[vertex] (3) at (-0.75, 0.2) {};
			\node[vertex] (4) at (0.75, 1) {};
			\node[vertex] (5) at (0.75, 0.2) {};
			\node[vertex] (6) at (-0.5, 2) {};
			\node[vertex] (7) at (0.5, 2) {};
			\node[vertex] (8) at (-1.25, 1) {};
			\node[vertex] (9) at (1.25, 1) {};
			\node[vertex] (10) at (1.25, 1.8) {};
			\node[vertex] (11) at (-1.25, 1.8) {};
			\node[vertex] (12) at (-0.25, 0.6) {};
			\node[vertex] (13) at (0.25, 0.6) {};
			\draw[edge] (1)edge(6) (1)edge(7) (1)edge(12) (1)edge(13) (2)edge(3) (2)edge(4) (2)edge(11) (2)edge(12) (3)edge(5) (3)edge(8) (3)edge(12) (4)edge(5) (4)edge(10) (4)edge(13) (5)edge(9) (5)edge(13) (6)edge(7) (6)edge(8) (6)edge(11) (7)edge(9) (7)edge(10) (8)edge(11) (9)edge(10);
		\end{tikzpicture}
		\quad
		\begin{tikzpicture}[scale=0.7,yscale=0.9]
			\node[vertex] (1) at (0.5, 0.5) {};
			\node[vertex] (2) at (2, 0) {};
			\node[vertex] (3) at (1, 0) {};
			\node[vertex] (4) at (0, 2) {};
			\node[vertex] (5) at (3, 2) {};
			\node[vertex] (6) at (0.7, 2.5) {};
			\node[vertex] (7) at (1.5, 1.25) {};
			\node[vertex] (8) at (2.5, 0.5) {};
			\node[vertex] (9) at (3, 0.1) {};
			\node[vertex] (10) at (1.5, 0.5) {};
			\node[vertex] (11) at (0, 0.1) {};
			\node[vertex] (12) at (0.5, 1.25) {};
			\node[vertex] (13) at (1.8, 2.5) {};
			\draw[edge] (1)edge(3) (1)edge(11) (1)edge(12) (1)edge(13) (2)edge(3) (2)edge(8) (2)edge(9) (2)edge(10) (3)edge(10) (3)edge(11) (4)edge(5) (4)edge(6) (4)edge(11) (4)edge(12) (5)edge(7) (5)edge(9) (5)edge(13) (6)edge(7) (6)edge(8) (6)edge(12) (7)edge(10) (7)edge(13) (8)edge(9);
		\end{tikzpicture}
		& 34816 \\
		\bottomrule
	\end{tabular}
	\caption{Graphs with maximal realization count on the sphere $\maxlamSset{2}{n}$ with $n\in\{10,11,12,13\}$ vertices (see \Cref{tab:enc-maxsphere} for encodings).}
	\label{fig:maxsphere}
\end{figure}

Still, there are some common properties of the graphs with maximal realization count on the sphere, $\maxlamSset{2}{n}$, and $n\geq 7$:
\begin{itemize}
	\item Minimum degree is 3 and maximum degree is 4, i.\,e.\ there are exactly six vertices of degree 3.
	\item The graph has at least two triangle subgraphs.
	\item The graph is Hamiltonian.
	\item There is a degree 3 vertex with two of its neighbors being connected, i.\,e.\ in the construction of the graph the last construction step does not need to be a 1-extension of type E1c (see \Cref{sec:factors}).
		This is in contrast to the graphs with maximal realization count in the plane.
\end{itemize}
These properties also hold for the graphs that have currently the largest known number of realizations for $n\geq13$.

Using fan constructions we are able to give a lower bound for $\maxlamSIIn$.
Let $G=(V,E)$ be a minimally 2-rigid graph and $H=(W,F)$ a minimally 2-rigid subgraph. Then similarly to \labelcref{eq:bound_genfan} we get
\begin{equation}\label{eq:sphere_bound_genfan}
	\maxlamSIIn \geq 2^{(n-|W|)\modop(|V|-|W|)} \cdot\lamSII{H} \cdot
	\left(\frac{\lamSII{G}}{\lamSII{H}}\right)^{\!\lfloor(n-|W|)/(|V|-|W|)\rfloor} \qquad (n\geq|W|).
\end{equation}

\Cref{tab:spherebounds} summarizes the growth rates we obtained from experiments and \Cref{fig:spheregrowth} illustrates them.
By growth rate we mean $\left(\lamSII{G}/\lamSII{H}\right)^{1/(|V|-|W|}$, which is the important factor of \labelcref{eq:sphere_bound_genfan}.

\begin{table}[ht]
	\centering\small
	\begin{tabular}{rllllll}
		\toprule
		$n$ & 1-fan   & 7-fan   & 31-fan  & 254-fan & 7916-fan\\\midrule
		6  & 2.37841 & 2.51984 & 2.00000 & 2.00000 & -       \\
		7  & 2.29740 & 2.37841 & 2.51984 & 2.00000 & 2.00000 \\
		8  & 2.40187 & 2.49146 & 2.37841 & 2.51984 & 2.00000 \\
		9  & 2.47940 & 2.56980 & 2.63902 & 2.37841 & 2.51984 \\
		10 & 2.50207 & 2.58342 & 2.61532 & 2.63902 & 2.37841 \\
		11 & 2.53687 & 2.61341 & 2.64094 & 2.61532 & 2.63902 \\
		12 & 2.56418 & 2.63596 & 2.66704 & 2.69180 & 2.61532 \\
		13 & 2.58755 & 2.65506 & 2.68150 & 2.66704 & 2.69180 \\\hline
		14 & 2.60644 & 2.66995 & 2.70524 & 2.70213 & 2.66704 \\
		15 & 2.62024 & 2.67989 & 2.71199 & 2.72441 & 2.72158 \\
		16 & 2.64381 & 2.70117 & 2.71114 & 2.71629 & 2.72441 \\
		17 & 2.64709 & 2.70062 & 2.73051 & 2.73087 & 2.71629\\
		\bottomrule
	\end{tabular}
	\caption{Growth rates (rounded) of the lower bounds.
		For $n\leq13$ these values are proven to be the best achievable ones; for $n>13$ the values are just the best we found by experiments,
		hence it is possible that there are better ones.\\
		The encodings for the graphs can be found in \Cref{tab:Senc:1-fan,tab:Senc:31-fan,tab:Senc:254-fan,tab:Senc:7916-fan}.%
	}
	\label{tab:spherebounds}
\end{table}

In contrast to the plane we see that for certain fan constructions and certain number of vertices the growth rate does not improve compared to fewer vertices.
For instance the growth rate on the 31-fan with 10 vertices would be worse than the one with 9 vertices.

\begin{figure}[ht]
	\centering
	\begin{tikzpicture}[yscale=0.9]
		\begin{scope}[yscale=15,xscale=1]
			\coordinate (o) at (5,2.29);
			\draw[cola,rounded corners] (3.5,2.22) rectangle (18.6,2.79);
			\node[alabelsty,rotate=90] at ($(0,2.29)!0.5!(0,2.7)+(4,0)$) {growth rate};
			\node[alabelsty] at ($(6,0)!0.5!(17,0)+(0,2.24)$) {$n$};
			\foreach \y in {2.3,2.4,2.5,2.6,2.7} \draw[bline] (5,\y) -- +(-0.2,0) node[left,labelsty] {$\y$};
			\foreach \x in {6,...,17} \draw[bline] (\x,2.3) -- +(0,-0.005) node[below,labelsty] {$\x$};

			\foreach \y [count=\i] in {2.3,2.31,...,2.73}
			{
				\draw[bline] (5,\y)--(18,\y);
			}

			\draw[indline] (13,2.3) -- (13,2.74) node[alabelsty,anchor=east] {\tikz{\draw[aline,solid,-{LaTeX[round]}] (0,0)--(-0.5,0);} complete};

			\draw[colR,gline]
				(6, 2.37841)
			-- (7, 2.2974)
			-- (8, 2.40187)
			-- (9, 2.4794)
			-- (10, 2.50207)
			-- (11, 2.53687)
			-- (12, 2.56418)
			-- (13, 2.58755);
			\draw[colRw,gline]
				(13, 2.58755)
			-- (14, 2.60644)
			-- (15, 2.62024)
			-- (16, 2.64381)
			-- (17, 2.64709);
			\node[valueR] at (6, 2.37841) {};
			\node[valueR] at (7, 2.2974) {};
			\node[valueR] at (8, 2.40187) {};
			\node[valueR] at (9, 2.4794) {};
			\node[valueR] at (10, 2.50207) {};
			\node[valueR] at (11, 2.53687) {};
			\node[valueR] at (12, 2.56418) {};
			\node[valueR] at (13, 2.58755) {};
			\node[valueRw] at (14, 2.60644) {};
			\node[valueRw] at (15, 2.62024) {};
			\node[valueRw] at (16, 2.64381) {};
			\node[valueRw] at (17, 2.64709) {};
			\draw[colB,gline]
				(6, 2.51984)
			-- (7, 2.37841)
			-- (8, 2.49146)
			-- (9, 2.5698)
			-- (10, 2.58342)
			-- (11, 2.61341)
			-- (12, 2.63596)
			-- (13, 2.65506);
			\draw[colBw,gline]
				(13, 2.65506)
			-- (14, 2.66995)
			-- (15, 2.67989)
			-- (16, 2.70117)
			-- (17, 2.70062);
			\node[valueB] at (6, 2.51984) {};
			\node[valueB] at (7, 2.37841) {};
			\node[valueB] at (8, 2.49146) {};
			\node[valueB] at (9, 2.5698) {};
			\node[valueB] at (10, 2.58342) {};
			\node[valueB] at (11, 2.61341) {};
			\node[valueB] at (12, 2.63596) {};
			\node[valueB] at (13, 2.65506) {};
			\node[valueBw] at (14, 2.66995) {};
			\node[valueBw] at (15, 2.67989) {};
			\node[valueBw] at (16, 2.70117) {};
			\node[valueBw] at (17, 2.70062) {};
			\draw[colG,gline]
				(7, 2.51984)
			-- (8, 2.37841)
			-- (9, 2.63902)
			-- (10, 2.61532)
			-- (11, 2.64094)
			-- (12, 2.66704)
			-- (13, 2.6815);
			\draw[colGw,gline]
				(13, 2.6815)
			-- (14, 2.70524)
			-- (15, 2.71199)
			-- (16, 2.71114)
			-- (17, 2.73051);
			\node[valueG] at (7, 2.51984) {};
			\node[valueG] at (8, 2.37841) {};
			\node[valueG] at (9, 2.63902) {};
			\node[valueG] at (10, 2.61532) {};
			\node[valueG] at (11, 2.64094) {};
			\node[valueG] at (12, 2.66704) {};
			\node[valueG] at (13, 2.6815) {};
			\node[valueGw] at (14, 2.70524) {};
			\node[valueGw] at (15, 2.71199) {};
			\node[valueGw] at (16, 2.71114) {};
			\node[valueGw] at (17, 2.73051) {};
			\draw[colP,gline]
				(8, 2.51984)
			-- (9, 2.37841)
			-- (10, 2.63902)
			-- (11, 2.61532)
			-- (12, 2.6918)
			-- (13, 2.66704);
			\draw[colPw,gline]
				(13, 2.66704)
			-- (14, 2.70213)
			-- (15, 2.72441)
			-- (16, 2.71628)
			-- (17, 2.73087);
			\node[valueP] at (8, 2.51984) {};
			\node[valueP] at (9, 2.37841) {};
			\node[valueP] at (10, 2.63902) {};
			\node[valueP] at (11, 2.61532) {};
			\node[valueP] at (12, 2.6918) {};
			\node[valueP] at (13, 2.66704) {};
			\node[valuePw] at (14, 2.70213) {};
			\node[valuePw] at (15, 2.72441) {};
			\node[valuePw] at (16, 2.71628) {};
			\node[valuePw] at (17, 2.73087) {};
			\draw[colX,gline]
				(9, 2.51984)
			-- (10, 2.37841)
			-- (11, 2.63902)
			-- (12, 2.61532)
			-- (13, 2.6918);
			\draw[colXw,gline]
				(13, 2.6918)
			-- (14, 2.66704)
			-- (15, 2.72158)
			-- (16, 2.72441)
			-- (17, 2.71628);
			\node[valueX] at (9, 2.51984) {};
			\node[valueX] at (10, 2.37841) {};
			\node[valueX] at (11, 2.63902) {};
			\node[valueX] at (12, 2.61532) {};
			\node[valueX] at (13, 2.6918) {};
			\node[valueXw] at (14, 2.66704) {};
			\node[valueXw] at (15, 2.72158) {};
			\node[valueXw] at (16, 2.72441) {};
			\node[valueXw] at (17, 2.71628) {};
		\end{scope}

		\begin{scope}
			\coordinate (s) at (0,-0.3);
			\coordinate (l) at ($(o)+(11,3)$);
			\draw[cola,fill=white,rounded corners] ($(l)+(-0.25,0.3)$) rectangle ++(1.65,-1.8);
			\node[valueR,legend,label={[labelsty]0:1-fan}] at (l) {};
			\node[valueB,legend,label={[labelsty]0:7-fan}] at ($(l)+(s)$) {};
			\node[valueG,legend,label={[labelsty]0:31-fan}] at ($(l)+2*(s)$) {};
			\node[valueP,legend,label={[labelsty]0:254-fan}] at ($(l)+3*(s)$) {};
			\node[valueX,legend,label={[labelsty]0:7916-fan}] at ($(l)+4*(s)$) {};
		\end{scope}
	\end{tikzpicture}
	\caption{Growth rates of the lower bounds for realizations on the sphere.
		The light colors indicate values that were not found by exhaustive search and which therefore could possibly
		be improved.}
	\label{fig:spheregrowth}
\end{figure}
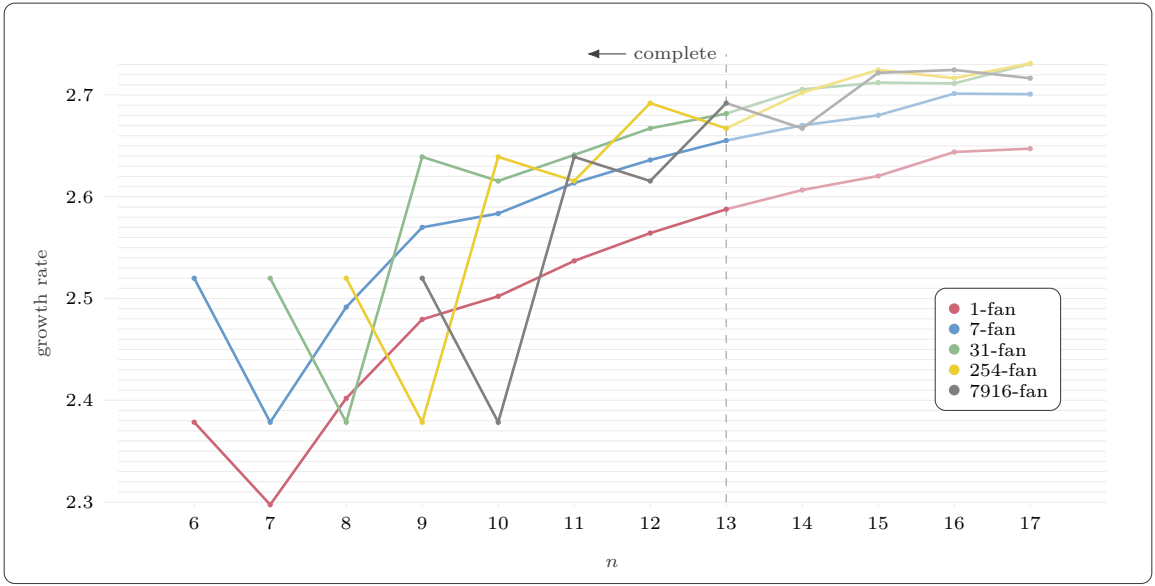

Instantiating \labelcref{eq:sphere_bound_genfan} with the minimally rigid graph 18610043923532523055425244310227943712,
which has 17 vertices and $\lamSII{G}=1376256$ realizations and a three-prism subgraph (254) for gluing,
we obtain the following theorem.

\begin{theorem}
	\label{thm:lower-bd-sphere}
	The maximal number of realizations $\maxlamSIIn$, for $n\geq 5$, satisfies
	\begin{equation*}
		\maxlamSIIn \geq 8\cdot2^{(n-5)\modop 12}\cdot(1376256/8)^{\lfloor(n-5)/12\rfloor}.
	\end{equation*}
	This means $\maxlamSIIn$ grows at least as $\bigl(\!\sqrt[12]{172032}\bigr)^n$ which is approximately $2.73087^n$.
\end{theorem}
Note that we are talking about complex realizations here. A respective bound for real realizations $\sim2.51984^n$ can be found in \cite{Bartzos2018}.
For an upper bound for $\maxlamSIIn$ we know the same as in the plane ($3.4641^n$ from \cite{Bartzos23}).
Note that on the sphere the current gap between lower and upper bound is smaller than in the plane even though we used graphs with fewer vertices.

As a consequence of \Cref{cor:lowerII} and \cite[Thm~1.1]{PlaneSphere} showing that $\lamSII{G}\geq \lamII{G}$ we get.
\begin{corollary}\label{cor:lowerS}
	Let $G$ be a minimally rigid graph with $n\leq 13$ vertices. Then $\lamSII{G}\geq 2^{n-2}$.
\end{corollary}

\subsection{Comparison of plane and sphere}
Since it is the same class of graphs that are minimally rigid on the sphere and in the plane, it is interesting to compare the results of realization counting.
In general for a minimally rigid graph $\lamII{G}$ and $\lamSII{G}$ might be different.
It was proven in \cite{PlaneSphere} that we always have $\lamII{G}\leq\lamSII{G}$.
Also a basic computational analysis has been done there.
For instance the distribution of ratios $\lamSIIfunc/\lamIIfunc$ has been illustrated and the percentage of graphs with $\lamSII{G}\neq\lamII{G}$.
Here we present the results of two more experiments: the comparison of graphs with high number of realizations in the plane and on the sphere and the differences in the distribution of realization numbers.

\subsubsection*{Graphs with high and low number of realizations}
We compare the graphs that have $\lamII{G}=\maxlamIIn$ or $\lamSII{G}=\maxlamSIIn$.
It turns out that the minimally rigid graphs with many realizations in the plane do not necessarily have many on the sphere as well and the other way round.
See \Cref{tab:compare} for a comparison.
\begin{table}[ht]
	\centering\small
	\begin{tabular}{rrrrr}
		\toprule
		$n$ & graph $G$              & $\lamII{G}=\maxlamIIn$ & $\lamSII{G}$ & $\maxlamSIIn$ \\\midrule
		6   & 7916                   & 24                     & 32           & 32    \\
		7   & 1269995                & 56                     & 64           & 64    \\
		8   & 170989214              & 136                    & 192          & 192   \\
		9   & 11177989553            & 344                    & 512          & 576   \\
		10  & 4778440734593          & 880                    & 1536         & 1536  \\
		11  & 18120782205838348      & 2288                   & 4096         & 4352  \\
		12  & 252590061719913632     & 6180                   & 8704         & 12288 \\
		13  & 2731597771584836257824 & 15536                  & 22528        & 34816 \\
		\bottomrule
	\end{tabular}

	\begin{tabular}{rrrrr}
		\toprule
		$n$ & graph $G$               & $\lamS{2}{G}=\maxlamSIIn$ & $\lamII{G}$ & $\maxlamIIn$\\\midrule
		6   & 7916                    & 32                        & 24          & 24    \\
		7   & 1269995                 & 64                        & 56          & 56    \\
		8   & 170989214               & 192                       & 136         & 136   \\
		9   & 2993854888              & 576                       & 320         & 344   \\
		10  & 4778440734593           & 1536                      & 880         & 880   \\
		11  & 18226779293308419       & 4352                      & 1920        & 2288  \\
		12  & 252695476130038944      & 12288                     & 4992        & 6180  \\
		13  & 6128220462188632473600  & 34816                     & 13440       & 15536 \\
		\bottomrule
	\end{tabular}
	\caption{Graphs which achieve the maximal number of realizations $\maxlamIIn$ or $\maxlamSIIn$. Whenever $\maxlamSIIset{n}$ has more than one element, we picked the one with largest realization count in the plane.}
	\label{tab:compare}
\end{table}

When we consider graphs with few realizations we observe the following. For graphs with less than 14 vertices it holds that if $\lam{2}{G}=2^{n-2}$ then also $\lamS{2}{G}=2^{n-2}$.
Furthermore, so far we did not find any graph with $\lam{2}{G}=2^{n-2}\neq\lamS{2}{G}$.

\subsubsection*{Distribution of realization numbers}
We now analyze which realization numbers appear.
This means for a given number of vertices we compute all realization numbers and check how many graphs do have the same number.
Indeed we only use graphs with minimum degree 3.
In the plane for graphs with 8 vertices we get 10 different values for $\lamIIfunc$.
The realization number that appears most is 96 with 10 out of 32 graphs achieving it.
In \Cref{tab:mostlam} we collect this information for all graphs up to 13 vertices.
What is interesting is, that the most common realization number is $3\cdot2^{|V|-3}$. However, the relative amount of graphs achieving this number is decreasing with the number of vertices.
\begin{table}[ht]
	\centering\small
	\begin{tabular}{rrrrrr}
		\toprule
		$n$ & $|\setDn|$ & most frequent $\lamIIfunc$ & \# graphs & \% & different $\lamIIfunc$ \\\midrule
		8   & 32         & 96       & 10      & 31.25 & 10   \\
		9   & 264        & 192      & 59      & 22.35 & 29   \\
		10  & 3189       & 384      & 571     & 17.90 & 102  \\
		11  & 46677      & 768      & 6179    & 13.24 & 401  \\
		12  & 813875     & 1536     & 77980   &  9.58 & 1529 \\
		13  & 16142835   & 3072     & 1095177 &  6.78 & 4973 \\
		\bottomrule
	\end{tabular}
	\caption{Realization numbers in the plane that are obtained most often within graphs from a given number of vertices.
	The second column shows the number of minimally rigid graphs with minimum degree 3 and the respective number of vertices $\setDn$.
	The third column shows the value of $\lamIIfunc$ that occurs most often in that set and the fourth column tells by how many graphs. The last column presents how many different realization numbers we get.}
	\label{tab:mostlam}
\end{table}

It turns out that on the sphere there are much viewer different numbers.
\Cref{tab:mostlams} shows the results for the sphere. Here the number that is obtained most often is $2^{|V|-1}$.
\begin{table}[ht]
	\centering\small
	\begin{tabular}{rrrrrr}
		\toprule
		$n$ & $|\setDn|$ & most $\lamSIIfunc$ & \# graphs & \% & different $\lamSIIfunc$ \\\midrule
		8   & 32         & 128      & 18      & 56.25 & 6   \\
		9   & 264        & 256      & 125     & 47.35 & 11  \\
		10  & 3189       & 512      & 1217    & 38.16 & 24  \\
		11  & 46677      & 1024     & 13522   & 28.97 & 47  \\
		12  & 813875     & 2048     & 174080  & 21.39 & 114 \\
		13  & 16142835   & 4096     & 2502769 & 15.50 & 259 \\
		\bottomrule
	\end{tabular}
	\caption{Realization numbers on the sphere that are obtained most often within graphs from a given number of vertices.
	The second column shows the number of minimally rigid graphs with minimum degree 3 and the respective number of vertices.
	The third column shows the $\lamIIfunc$ that occurs most often in that set and the fourth column tells by how many graphs. The last column presents how many different realization numbers we get.}
	\label{tab:mostlams}
\end{table}
In \Cref{fig:distribution} we show the overall distribution for the realization numbers in the plane and on the sphere for graphs with up to 13 vertices.
\begin{figure}[ht]
	\centering
	\includegraphics[height=4cm]{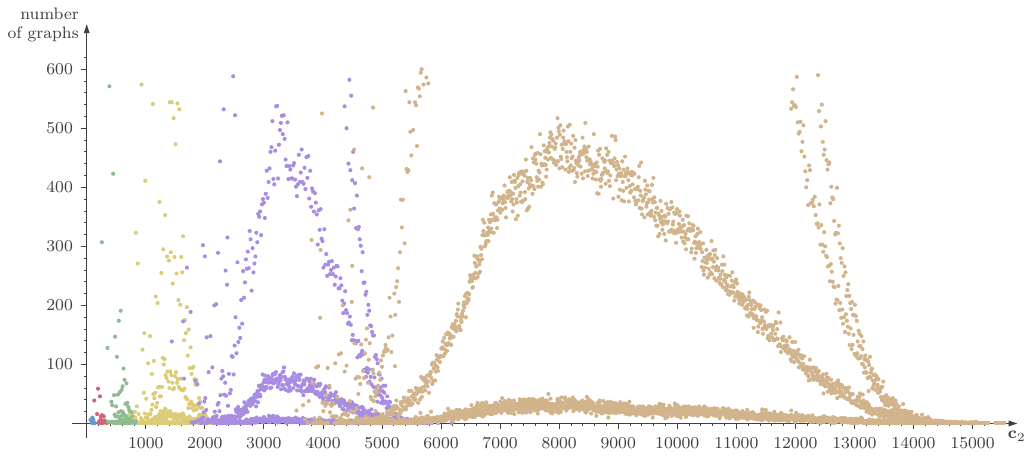}%
	\includegraphics[height=4cm]{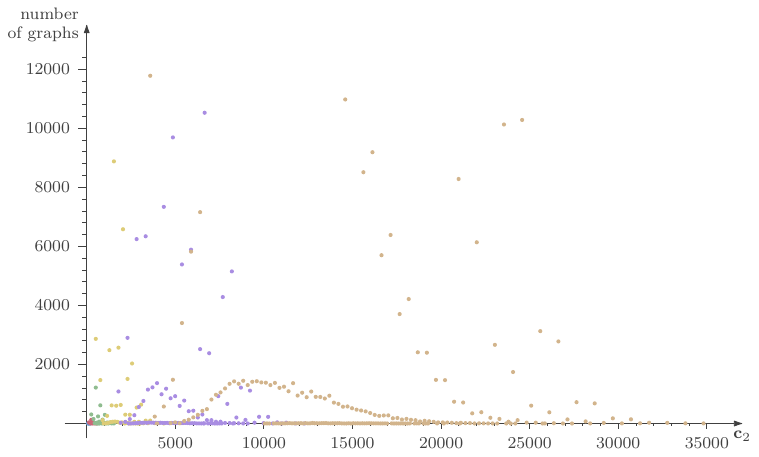}
	\caption{The distribution of different numbers of realizations in the plane (left) and on the sphere (right). The horizontal axis represents the different realization numbers while the vertical axis shows how often they appear. The colors differ between the number of vertices from 8 to 13. Note that the plots do not show the full range, in order to visualize patterns.}
	\label{fig:distribution}
\end{figure}

\subsubsection*{Ratios between plane and sphere counts}
In \cite{PlaneSphere} the ratio between $\lamfunc{2}$ and $\lamSfunc{2}$ was analyzed.
In particular the maximal possible ratio for a given number of vertices is of interest.
Let accordingly
\begin{equation*}
	\theta_d(n) := \max_{G\in\setM{d}{n}} \frac{\lamS{d}{G}}{\lam{d}{G}}\,.
\end{equation*}
It was shown in \cite{PlaneSphere} that for each positive integer $d$ there is an $\alpha$ such that $\theta_d(n)\in\mathcal{O}(\alpha^n)$.
Defining $\alpha_d$ to be the infimum of all the possible $\alpha$ they derive a bound for $\alpha_2$, to be $\alpha_2\geq(4/3)^{3/8}$ which is around $1.11391$.
This bound was again obtained by using fan constructions on a triangle subgraph.
Indeed we can, with the new data sets, improve this bound now.
We use the graph $G$ with integer representation $1290592463576136176277725405703$ which has $\lamII{G}=62208$ and $\lamSII{G}=262144$ and $15$ vertices.
A fan construction on a triangle subgraph yields the bound $\alpha_2\geq (4/3)^{5/12}$ which is around $1.12735$ and therefore a slight improvement.
Due to \Cref{thm:fanformula} we can use fans on larger minimally rigid subgraphs as well and get the following even better bound of around $1.21141$.
\begin{lemma}
	$\alpha_2\geq (4/3)^{2/3}$
\end{lemma}
\begin{proof}
	The graph $G$ with integer representation $6917588647$ has $\lamII{G}=288$ and $\lamSII{G}=512$ and $9$ vertices.
	A fan construction on a three prism subgraph on 6 vertices yields the bound
	\begin{equation*}
		\left(\frac{\lamSII{G}}{\lamII{G}}\right)^{1/(9-6)}=\left(\frac{512}{288}\right)^{1/3}=\left(\frac{4}{3}\right)^{2/3}\,.\qedhere
	\end{equation*}
\end{proof}

\section{Realizations in space}
\label{sec:space}
In fan constructions in dimension three we cannot glue on a common edge because then we would get a flexible graph. So the most basic fan construction glues two graphs by a common triangle (7-fan).
For the 63-fan we use a common tetrahedron to glue at and the 511-fan construction needs the unique minimally 3-rigid graph with $5$ vertices.
See \Cref{fig:smallm3rinteger} for pictures and graph encodings for small minimally 3-rigid graphs.

In dimension three and higher we do not have combinatorial algorithms available for computing the realization counts.
What we do instead is counting the number of solutions of the system of edge lengths equations for a given graph $G=(V,E)$:
\begin{align*}
	(x_i-x_j)^2 + (y_i-y_j)^2 &= \lambda_{ij} \forall \{i,j\}\in E,\\
	x_1=y_1=x_2&=0,
\end{align*}
where $\lambda_{ij}$ represents the edge lengths of the edge $\{i,j\}$.
We use Gröbner bases for counting these solutions.
Keeping the edge lengths symbolic would not be feasible so instead we take random edge lengths.
We need to be aware however, that this makes it a probabilistic method.
Furthermore, we do computations modulo a sufficiently large prime as in \cite{LowerBounds}.
To be confident in the computed numbers we did all computations at least ten times.
In the computation we took advantage of leading monomial computation of \msolve\ \cite{msolve} with pre- and postprocessing using \cite{RigiComp} to get the actual number of realizations.

Already in \cite{LowerBounds} the realization counts for all minimally 3-rigid graphs with 10 or less vertices have been computed.
Using \msolve\ we could go a little further.

\Cref{tab:bounds3d} summarizes the growth rates we obtained from experiments and \Cref{fig:3dgrowth} illustrates them.
By growth rate we mean $\left(\lamIII{G}/\lamIII{H}\right)^{1/(|V|-|W|}$.

\begin{table}[ht]
	\centering\small
	\begin{tabular}{rllllll}
		\toprule
		$n$ & 7-fan  & 63-fan  & 511-fan \\\midrule
		6  & 2.51984 & 2.00000 & 2.00000 \\
		7  & 2.63215 & 2.51984 & 2.00000 \\
		8  & 2.75946 & 2.63215 & 2.51984 \\
		9  & 2.93560 & 2.95155 & 2.82843 \\
		10 & 3.06826 & 3.06681 & 2.95155 \\\hline
		11 & 3.15140 & 3.16764 & 3.12314 \\
		12 & 3.35787 & 3.30791 & 3.28134\\
		\bottomrule
	\end{tabular}
	\caption{Growth rates (rounded) of the lower bounds. The encodings for the graphs on 11 and 12 vertices can be found in \Cref{tab:enc:fan3d}. The results for fewer vertices are from \cite{LowerBounds}.}
	\label{tab:bounds3d}
\end{table}
From \Cref{tab:bounds3d} we can see the bound of $2.51984^n$
obtained in \cite{Emiris2009}, $3.06825^n$ from \cite{LowerBounds} and the improvements obtained in this
paper.

\begin{figure}[ht]
	\centering
	\begin{tikzpicture}
		\begin{scope}[yscale=4]
			\coordinate (o) at (5,2.5);
			\draw[cola,rounded corners] (3.5,2.22) rectangle (13.8,3.55);
			\node[alabelsty,rotate=90] at ($(0,2.5)!0.5!(0,3.4)+(4,0)$) {growth rate};
			\node[alabelsty] at ($(5,0)!0.5!(12,0)+(0,2.3)$) {$n$};

			\foreach \y in {2.5,2.7,2.9,3.1,3.3} \draw[bline] (5,\y) -- +(-0.1,0) node[left,labelsty] {$\y$};
			\foreach \x in {5,...,12} \draw[bline] (\x,2.5) -- +(0,-0.005) node[below,labelsty] {$\x$};

			\foreach \y [count=\i] in {2.5,2.6,...,3.4}
			{
				\draw[bline] (5,\y)--(13,\y);
			}
			\draw[indline] (10,2.5) -- (10,3.45) node[alabelsty,anchor=east] {\tikz{\draw[aline,solid,-{LaTeX[round]}] (0,0)--(-0.5,0);} complete};
			\draw[indline] (5,3.06825) -- (13.25,3.06825) node[right,alabelsty,anchor=west,rotate=90] {new \tikz{\draw[aline,solid,-{LaTeX[round]}] (0,0)--(0.5,0);}};

			\draw[colR,gline]
				(6, 2.51984)
			-- (7, 2.63215)
			-- (8, 2.75946)
			-- (9, 2.9356)
			-- (10, 3.06825);
			\draw[colRw,gline]
				(10, 3.06825)
			-- (11, 3.1514)
			-- (12, 3.35787);
			\node[valueR] at (6, 2.51984) {};
			\node[valueR] at (7, 2.63215) {};
			\node[valueR] at (8, 2.75946) {};
			\node[valueR] at (9, 2.9356) {};
			\node[valueR] at (10, 3.06825) {};
			\node[valueRw] at (11, 3.1514) {};
			\node[valueRw] at (12, 3.35787) {};
			\draw[colB,gline]
				(7, 2.51984)
			-- (8, 2.63215)
			-- (9, 2.95155)
			-- (10, 3.06681);
			\draw[colBw,gline]
				(10, 3.06681)
			-- (11, 3.16764)
			-- (12, 3.30791);
			\node[valueB] at (7, 2.51984) {};
			\node[valueB] at (8, 2.63215) {};
			\node[valueB] at (9, 2.95155) {};
			\node[valueB] at (10, 3.06681) {};
			\node[valueBw] at (11, 3.16764) {};
			\node[valueBw] at (12, 3.30791) {};
			\draw[colG,gline]
				(8, 2.51984)
			-- (9, 2.82843)
			-- (10, 2.95155);
			\draw[colGw,gline]
				(10, 2.95155)
			-- (11, 3.12314)
			-- (12, 3.28134);
			\node[valueG] at (8, 2.51984) {};
			\node[valueG] at (9, 2.82843) {};
			\node[valueG] at (10, 2.95155) {};
			\node[valueGw] at (11, 3.12314) {};
			\node[valueGw] at (12, 3.28134) {};
		\end{scope}

		\begin{scope}
			\coordinate (s) at (0,-0.3);
			\coordinate (l) at ($(o)+(6.25,1.25)$);
			\draw[cola,fill=white,rounded corners] ($(l)+(-0.25,0.3)$) rectangle ++(1.65,-1.2);
			\node[valueR,label={[labelsty]0:7-fan}] at (l) {};
			\node[valueB,label={[labelsty]0:63-fan}] at ($(l)+(s)$) {};
			\node[valueG,label={[labelsty]0:511-fan}] at ($(l)+2*(s)$) {};
		\end{scope}
	\end{tikzpicture}
	\caption{Growth rates of the lower bounds for realizations in three space.
		The light colors indicate values that were not found by exhaustive search and which therefore could possibly
		be improved. The horizontal dashed line indicates the lower bound known from \cite{LowerBounds}.}
	\label{fig:3dgrowth}
\end{figure}
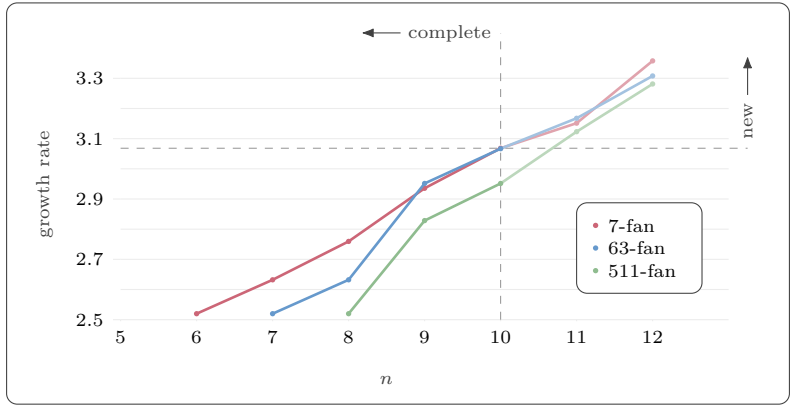

The best bound we found is achieved by a 7-fan construction on the minimally 3-rigid icosahedron graph with integer representation 1151802456431434509 which has 12 vertices and 54272 complex realizations. This value was also computed by \cite{Bartzos2020} using a homotopy solver and a combinatorial upper bound.
Using this graph we get the following rigorous result from \labelcref{eq:bound_genfan}.
\begin{theorem}
	\label{thm:lower-bd-3d}
	The maximal number of realizations $\maxlamIIIn$, for $n\geq 3$, satisfies
	\begin{equation*}
		\maxlamIIIn \geq 1\cdot2^{(n-3)\modop 9}\cdot(54272/1)^{\lfloor(n-3)/9\rfloor}.
	\end{equation*}
	Hence, $\maxlamIIIn$ grows at least as $\bigl(\!\sqrt[9]{54272}\bigr)^n$ which is approximately $3.3579^n$.
	In other words $\bigl(\!\sqrt[9]{54272}\bigr)^n\in\mathcal O(\maxlamIIIn)$.
\end{theorem}
Note that we are talking about complex realizations here. A respective bound for real realizations $\sim2.6553^n$ can be found in \cite{Bartzos2018}.
The currently known best upper bound for $\maxlamIIIn$ is approximately $6.32^n$ \cite{Bartzos23}.
Hence, there is still a huge gap to be investigated.

In \cite{LowerBounds,Jackson2018} it was conjectured that $\lamII{G}\geq 2^{n-2}$. The currently available data supports this conjecture.
One would be inclined to conjecture that $\lamIII{G}\geq 2^{n-3}$ which is wrong by the graph in \Cref{fig:counterexamplelowerbound} on 8 vertices (31965132) which has $24\leq 32=2^{8-3}$.
Such graphs can exist because there are construction steps which increase the number of realizations by a factor less than two (see \Cref{sec:factors}).
While in dimension two it seems that such construction steps only happen for graphs with reasonably high realization count, in dimension three the situation is quite different.
Hence, it does also make sense to analyze the minimum $\lamIIIfunc$ for given number of vertices $n$.

\begin{figure}[ht]
	\centering
	\begin{tikzpicture}[scale=1.5]
		\node[vertex] (1) at (0, 0.5) {};
		\node[vertex] (2) at (0, 1) {};
		\node[vertex] (3) at (1.5, 0.5) {};
		\node[vertex] (4) at (1.5, 1) {};
		\node[vertex] (5) at (0.75, 1.5) {};
		\node[vertex] (6) at (0.75, 0.5) {};
		\node[vertex] (7) at (0.75, 0.) {};
		\node[vertex] (8) at (0.75, 1) {};
		\draw[edge] (1)edge(5) (1)edge(6) (1)edge(7) (1)edge(8) (2)edge(5) (2)edge(6) (2)edge(7) (2)edge(8) (3)edge(5) (3)edge(6) (3)edge(7) (3)edge(8) (4)edge(5) (4)edge(6) (4)edge(7) (4)edge(8) (5)edge(8) (6)edge(7);
	\end{tikzpicture}
	\hspace{2cm}
	\begin{tikzpicture}[scale=0.8]
		\node[vertex] (1) at (0.75, 0) {};
		\node[vertex] (2) at (0, 0.75) {};
		\node[vertex] (3) at (-0.75, 0) {};
		\node[vertex] (4) at (0, -0.75) {};
		\foreach \i in {1,2,...,6}
		{
			\node[vertex] (v\i) at (60*\i+20:2) {};
			\draw[edge] (v\i)--(1) (v\i)--(2) (v\i)--(3) (v\i)--(4);
		}
	\end{tikzpicture}

	\caption{A minimally 3-rigid graph with $\lamIII{G}=24$ and one with $\lamIII{G}=76$.}
	\label{fig:counterexamplelowerbound}
\end{figure}
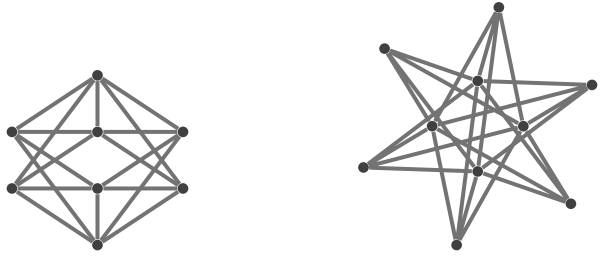
\Cref{tab:enc3min:1-fan} collects graph representatives which obtain the minimal realization number that we found so far for a given number of vertices.
Since data is complete for graphs with less than 11 vertices, we know that $\minlamIII{8}=24$, $\minlamIII{9}=48$ and $\minlamIII{10}=76$.
\begin{table}[ht]
	\centering\scriptsize
	\begin{tabular}{rrr}
		\toprule
		n  & graph encoding & $\lamIIIfunc$ \\\midrule
		8 & 31965132 & 24 \\
		9 & 1911013324 & 48 \\
		10 & 1034850648000 & 76 \\\midrule
		11 & 1725069083000768 & 152 \\
		12 & 10413379983099525195 & 304 \\
		13 & 1993222090148801645346 & 576 \\
		14 & 116992142554246520151559023 & 1152 \\
		15 & 131237666349318820678678431903 & 2304 \\
		16 & 1176451566300935565578955122609301 & 3648 \\
		17 & 1032841673513196991793263473349565121663 & 7296 \\
		18 & 2526294840809376540865170270027768169302165 & 11552 \\\bottomrule
	\end{tabular}
	\caption{Graphs with minimal number of complex realizations for given number of vertices $n\leq 10$ and graphs with a smallest known number for $n\geq 11$.}
	\label{tab:enc3min:1-fan}
\end{table}

In \cite[Thm~1.2]{PlaneSphere} it was proven that if a graph $G$ has a vertex $o$ adjacent to every other vertex then $\lamIII{G}=\lamSII{G\setminus o}$.
When $G\setminus o$ has less than 14 vertices we know that $\lamSII{G\setminus o}\geq 2^{n-3}$ when $n$ is the number of vertices of $G$.
From this we get the following corollary using \Cref{cor:lowerS}.
\begin{corollary}\label{cor:lowerIII}
	Let $G$ be a minimally 3-rigid graph with $n\leq 14$ vertices and let $o$ be a vertex with $\deg(o)=n-1$.
	Then $\lamIII{G}\geq 2^{n-3}$.
\end{corollary}

When we build a 7-fan with $k$ copies of the graph 31965132 then we know that $\lamIIIfunc$ of the resulting graph is $24^k$, which means that the graph asymptotically has around $\sqrt[5]{24}^n\sim 1.88^n$ complex realizations.
The same we can do for larger graphs.
Let $\minlamIIIn$ be the minimal realization number $\lamIIIfunc$ for all graphs with $n$ vertices.
Using a fan-construction with the graph with encoding
\begin{equation*}
	2526294840809376540865170270027768169302165
\end{equation*}
which has 18 vertices and 11552 realizations we get the following.
\begin{theorem}
	The minimal number of realizations $\minlamIIIn$ satisfies
	\begin{equation*}
		\minlamIIIn \leq 1\cdot 2^{(n-3)\modop 15}\cdot(11552/1)^{\lfloor(n-3)/15\rfloor}.
	\end{equation*}
	Hence, $\minlamIIIn$ grows at most as $\bigl(\!\sqrt[15]{11552}\bigr)^n$ which is approximately $1.86571^n$.
	\label{thm:minlamIII}
\end{theorem}
Note that in the data we have at the moment other fan constructions do not yield anything better.

Up to 10 vertices there are not many graphs with a realization number lower than $2^{|V|-3}$.
In \Cref{tab:lowlamIII} we show how many there are with minimum degree 4.
Every 0-extension of such a graph has also $\lamIII{G}<2^{|V|-3}$ but minimum degree 3.
\begin{table}[ht]
	\centering\small
	\begin{tabular}{rrr}
		\toprule
		$n$ & graphs & graphs (min deg 4)\\\midrule
		8   & 1      & 1 \\
		9   & 7      & 2 \\
		10  & 115    & 4 \\
		\bottomrule
	\end{tabular}
	\caption{Number of graphs with $\lamIII{G}< 2^{|V|-3}$ and in the last column those that have minimum degree 4.}
	\label{tab:lowlamIII}
\end{table}

\subsection{Higher dimensions}
Similarly we can do for higher dimensions. From \cite{PlaneSphere} we know that $\lam{d}{G}\leq\lamS{d}{G}$ for all graphs $G$.
Hence, $\maxlam{d}{n}\leq\maxlamS{d}{n}$ and $\minlam{d}{n}\leq\minlamS{d}{n} $ for all $n$.
In dimension $d$ the smallest reasonable fan construction is by gluing on a complete graph with $d$ vertices.
Again computations are done probabilistically using Gröbner bases on the equations given by random edge lengths.
We then get the following lower bounds for $\maxlam{d}{n}$ and $\maxlamS{d}{n}$ (\Cref{tab:higherdim}), which are clearly non-optimal.
Computations for $d\geq4$ were only done for a small subset of the minimally $d$-rigid graphs with at most $d+7$ vertices.
\begin{table}[ht]
	\centering\small
	\begin{tabular}{rrr}
		\toprule
		$d$ & bound for $\maxlam{d}{n}$ & bound for $\maxlamS{d}{n}$\\\midrule
		2   & 2.59972 & 2.73087 \\
		3   & 3.35787 & 3.30791 \\
		4   & 3.61722 & 3.66714 \\
		5   & 3.66714 & 3.66714 \\
		6   & 4.06019 & 4.07930 \\
		\bottomrule
	\end{tabular}
	\caption{Some base factors for lower bounds for $\maxlam{d}{n}$ and $\maxlamS{d}{n}$ achieved from computations with small graphs.}
	\label{tab:higherdim}
\end{table}

Let $G$ be a graph with a vertex adjacent to all other vertices.
Then we know from \cite{PlaneSphere} that $\lam{d}{G}=\lamS{d}{G}$ and furthermore $\lam{d+1}{G^*}=\lam{d}{G}$, where $G^*$ is obtained from $G$ by a coning operation, i.\,e.\ by adding a new vertex and edges from the new vertex to all the others. From this we get $\lam{d+k}{G^{*k}}=\lam{d}{G}$ for graphs $G^{*k}$ constructed by coning $k$ times.
Using this we can get general bounds for $\maxlam{d}{n}$, $\maxlamS{d}{n}$, $\minlam{d}{n}$ and $\minlamS{d}{n}$.
Note that the condition on a vertex being adjacent to all others in $G$ is crucial.
Hence, we do get the following bounds using gluing operations.
\begin{theorem}
	For $d\geq 5$ and $n\geq d+1$:
	\begin{align*}
		\maxlam{d}{n} \geq 2 \cdot 2^{(n-d-1)\modop(6)} \cdot 2432^{\!\lfloor(n-d-1)/6\rfloor}.
	\end{align*}
	Hence, $\maxlam{d}{n}$ grows at least as $\bigl(\!\sqrt[6]{2432}\bigr)^n$ which is approximately $3.66714^n$.
\end{theorem}
\begin{proof}
	The graph $G$ with integer representation 35018505495120117759 with 12 vertices has $\lam{5}{G}=4864$.
	The complete graph on 6 vertices is a subgraph of $G$ and hence by fan construction we get the result for $d=5$.
	\begin{align*}
		\maxlam{d}{n} \geq 2 \cdot 2^{(n-6)\modop(12-6)} \cdot \left(\frac{4864}{2}\right)^{\!\lfloor(n-6)/(12-6)\rfloor}.
	\end{align*}
	Furthermore, $G$ has a vertex of degree 11.
	Let $G^{*k}$ be the graph obtain from $G$ by $k$ coning operations.
	Then $\lam{d+k}{G^{*k}}=\lam{d}{G}=4864$ and $G^{*k}$ contains a complete graph with $6+k$ vertices which has 2 realizations.
	Hence, the result for $d\geq 5$ holds.
\end{proof}
Note that these are quite far apart from the upper bounds computed in \cite{Bartzos23}.
As an immediate consequence we get the same bound for $\maxlamS{d}{n}$.

Using a fan-construction of the graph with encoding
\begin{equation*}
	673666050660509883906464650904436971228255917423
\end{equation*}
which has a coned vertex, has 19 vertices and 11552 realizations in dimension $4$, we get the following.
\begin{theorem}
	The minimal number of realizations $\minlam{d}{n}$, for $d\geq 4$ and $n\geq d$, satisfies
	\begin{equation*}
		\minlam{d}{n} \leq 2^{(n-d)\modop 15}\cdot(11552/1)^{\lfloor(n-d)/15\rfloor}.
	\end{equation*}
	Hence, $\minlam{d}{n}$ grows at most as $\bigl(\!\sqrt[15]{11552}\bigr)^n$ which is approximately $1.86571^n$.
	\label{thm:minlamd}
\end{theorem}

We do know that for certain dimensions a reasonably small graph needs to have a coned vertex.
From this we get the following precise result.
\begin{lemma}
	Let $d$ be the dimension. Then
	\begin{align*}
		\maxlam{d}{d+3}&=\maxlamS{d}{d+3}=16  & \text{for } d\geq 3,\\
		\maxlam{d}{d+4}&=\maxlamS{d}{d+4}=256 & \text{for } d\geq 8.
	\end{align*}
\end{lemma}
\begin{proof}
	We only need to consider minimally $d$-rigid graphs with minimum degree $d+1$.
	It is easy to show that such a graph with $d+3$ vertices has a vertex of degree $d+2$ as long as $d>3$.
	Computations show that $\maxlam{4}{7}=16$.
	Hence, it is enough to know that the graph $G$ with integer representation 1048059 has $\lam{4}{G}=16$.
	Indeed already the graph $G'$ obtained by deleting the coned vertex of $G$ has $\lam{3}{G'}=16$.

	Similarly a minimally $d$-rigid graph with $d+4$ vertices does always have a vertex of degree $d+3$ when $d>8$.
	We can compute $\lam{9}{G}$ for all minimally $9$-rigid graphs with 13 vertices and get that $\maxlam{9}{13}=256$,
	obtained by the graph $G$ with integer representation 151115709437395776568827.
	Again the graph $G'$ obtained by deleting the coned vertex of $G$ has $\lam{8}{G'}=256$.

	The same graphs certify for the spherical count.
\end{proof}
In a similar way the computations yield results on the minimal number of realizations.
\begin{lemma}
	Let $d$ be the dimension. Then
	\begin{align*}
		\minlam{d}{d+3}&=\minlamS{d}{d+3}=8  & \text{for } d\geq 2,\\
		\minlam{d}{d+4}&=\minlamS{d}{d+4}=16 & \text{for } d\geq 2.
	\end{align*}
\end{lemma}
\begin{proof}
	Here the graph $G_1$ with integer representation 507903 has a coned vertex.
	It achieves $\lam{4}{G_1}=\lamS{4}{G_1}=8=\minlam{4}{7}$ and works as a certificate.
	All computations for $d\leq 3$ and $6$ vertices show the remaining part.
	The graph $G_2$ with integer representation 75520970211344265510911 has a coned vertex and $\lam{9}{G_2}=\lamS{9}{G_2}=16=\minlam{9}{13}$.
	The full data for graphs with $d+4$ vertices for $d<9$ yields the remaining results.
\end{proof}

\section{Changing factors of construction steps}\label{sec:factors}
In this section we summarize experimental results on the increase of the number of realizations after certain construction steps.
In particular we look at all possibilities of constructions with a low number of vertices.
For each construction step and each number of vertices we picked one example with the lowest and one example with the largest factor for the increase of the number of realizations.
The chosen construction steps for the plane and the sphere are naturally the same, while the factors differ.
Vertex and spider splits are defined for any dimension.
Extension constructions can be generalized to higher dimensions, though they loose their classification properties.
Nevertheless, we can check the factors for low number of vertices.

Bounds on these factors of increase have been described an open problem in \cite{LowerBounds,Jackson2018}. We are not solving this problem in the current paper but push the limits and go into more details.

All the results in this section are complete in a sense, that all possible cases for the given number of vertices have been checked. This, however does not indicate what might happen for graphs with more vertices.
Sometimes several graphs and operations yield the same minimum or maximum. In this case we just record a representative.
In this section we collect the resulting statements on the factors that the construct steps may yield in the number of realizations.
The certificate graphs proving the statements can be found in \Cref{sec:factor-cert}.

\subsection{Plane}
It is known that all minimally rigid graphs in the plane can be constructed by a sequence of extension constructions (also called Henneberg steps) \cite{Henneberg}.
We have seen that different types of these steps yield a different increase of $\lamIIfunc$.
While it is known and easy to see that 0-extensions always increase the number of realizations by a factor of two \cite{BorceaStreinu2004},
little is known for the 1-extensions and hence even a general lower bound that is more than constant is an open problem (compare \cite{LowerBounds}).
In \cite{Jackson2018} is is shown that vertex splitting increases $\lamIIfunc$ by a factor of at least two and as a consequence some 1-extensions also have this property.
Some experimental results on the factors have been presented in \cite{LowerBounds}.
Here we present some more detailed experiments for several construction steps and low number of vertices.
All the results in the following tables are exhaustive meaning we checked all possible constructions.
\subsubsection{Extension Constructions}
In order to see more properties we distinguish 1-extension into subclasses E1a, E1b and E1c depending on how many edges exist within the chosen vertices (see \Cref{fig:Henneberg}).
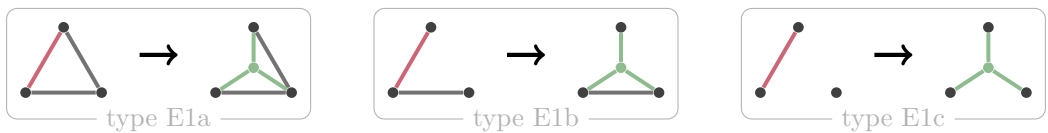
\begin{figure}[ht]
	\centering
	\begin{tikzpicture}[scale=1]
		\node[eframe] (ef) at (1.75,-0.35) {type E1a};
		\draw[eframe] (ef) -- ++(2,0) -- ++ (0,1.466) -- ++(-4,0) -- ++(0,-1.466) -- (ef);
		\node[vertex] (a) at (0,0) {};
		\node[vertex] (b) at (1,0) {};
		\node[vertex] (c) at (0.5,0.866025) {};
		\node[vertex] (d) at (2.5,0) {};
		\node[vertex] (e) at (3.5,0) {};
		\node[vertex] (f) at (3,0.866025) {};
		\node[nvertex] (g) at (3,0.33) {};
		\draw[edge] (a)edge(b) (b)edge(c) (d)edge(e) (e)edge(f);
		\draw[oedge] (a)edge(c);
		\draw[nedge] (d)edge(g) (e)edge(g) (f)edge(g);
		\draw[construle] (1.5,0.5) -- (2,0.5);
	\end{tikzpicture}
	\qquad
	\begin{tikzpicture}[scale=1]
		\node[eframe] (ef) at (1.75,-0.35) {type E1b};
		\draw[eframe] (ef) -- ++(2,0) -- ++ (0,1.466) -- ++(-4,0) -- ++(0,-1.466) -- (ef);
		\node[vertex] (a) at (0,0) {};
		\node[vertex] (b) at (1,0) {};
		\node[vertex] (c) at (0.5,0.866025) {};
		\node[vertex] (d) at (2.5,0) {};
		\node[vertex] (e) at (3.5,0) {};
		\node[vertex] (f) at (3,0.866025) {};
		\node[nvertex] (g) at (3,0.33) {};
		\draw[edge] (a)edge(b) (d)edge(e);
		\draw[oedge] (a)edge(c);
		\draw[nedge] (d)edge(g) (e)edge(g) (f)edge(g);
		\draw[construle] (1.5,0.5) -- (2,0.5);
	\end{tikzpicture}
	\qquad
	\begin{tikzpicture}[scale=1]
		\node[eframe] (ef) at (1.75,-0.35) {type E1c};
		\draw[eframe] (ef) -- ++(2,0) -- ++ (0,1.466) -- ++(-4,0) -- ++(0,-1.466) -- (ef);
		\node[vertex] (a) at (0,0) {};
		\node[vertex] (b) at (1,0) {};
		\node[vertex] (c) at (0.5,0.866025) {};
		\node[vertex] (d) at (2.5,0) {};
		\node[vertex] (e) at (3.5,0) {};
		\node[vertex] (f) at (3,0.866025) {};
		\node[nvertex] (g) at (3,0.33) {};
		\draw[oedge] (a)edge(c);
		\draw[nedge] (d)edge(g) (e)edge(g) (f)edge(g);
		\draw[construle] (1.5,0.5) -- (2,0.5);
	\end{tikzpicture}
	\caption{1-extensions of different types in dimension~2. Green vertices are new vertices.}
	\label{fig:Henneberg}
\end{figure}

When $G'$ is obtained from $G$ by a 1-extension of type E1a, we know that $\lamII{G'}/\lamII{G}\geq 2$ from \cite[Lem~4.6]{Jackson2018}.
Experiments indicate that extensions of type E1a maybe only yield a fixed factor of two for the increase of $\lamIIfunc$.
We might therefore conjecture
\begin{conjecture}
	Let $G$ be a minimally rigid graph and $G'$ be obtained from $G$ by a 1-extension of type E1a.
	Then $\frac{\lamII{G'}}{\lamII{G}}=2$.
\end{conjecture}
The conjecture is true for graphs $G$ with at most 12 vertices.

When $G'$ is obtained from $G$ by a 1-extension of type E1b, we know that $\lamII{G'}/\lamII{G}\geq 2$ from \cite[Lem~4.6]{Jackson2018} but it might be larger as \Cref{tab:Laman_Increase_H2b} shows. For the upper bound we get.
\begin{proposition}
	Assume that there is a universal constant $k_2$ bounding $\lamII{G'}/\lamII{G}$ from above for all minimally rigid graphs $G$ and $G'$ where $G'$ is obtained from $G$ by a 1-extension of type E1b.
	Then $k_2 \geq 9.75$.
\end{proposition}
\Cref{tab:Laman_Increase_H2b} collects a representative graph $G$ for each number of vertices and a graph $G'$ that is obtained from $G$ by an extension of type~E1b with the respective number of realizations.

Finally, from \Cref{tab:Laman_Increase_H2c} we can see that type E1c steps can increase $\lamIIfunc$ be a factor of less than two.
The table collects a representative graph $G$ for each number of vertices and a graph $G'$ that is obtained from $G$ by a step of type~E1c with the respective number of realizations.
The lowest factor we found is $\sim1.71$.
Since, we have a complete data set for up to 13 vertices, where this lowest factor is the same for all $7\leq n \leq 12$ we might conjecture that it is indeed the lowest possible.
What we do know for sure instead is the following.
\begin{proposition}
	Assume that there are universal constants $k_1$ and $k_2$ bounding $\lamII{G'}/\lamII{G}$ from below and above respectively for all minimally rigid graphs $G$ and $G'$ where $G'$ is obtained from $G$ by a 1-extension of type E1c.
	Then $k_1\leq 12/7$ and $k_2 \geq 11.59$.
\end{proposition}

\subsubsection{Splittings}
There are two well known types of splittings where vertices are split in order to get a larger minimally rigid graph.
We refer to \cite{NixonRoss} for and overview and \Cref{fig:splitting} for an illustration.

\begin{figure}[ht]
	\centering
	\begin{tikzpicture}[]
		\node[eframe] (ev) at (0,-1) {vertex splitting};
		\draw[eframe] (ev) -- ++(3.5,0) -- ++ (0,2) -- ++(-7,0) -- ++(0,-2) -- (ev);
		\begin{scope}[xshift=-2cm,yshift=-0.5cm]
			\begin{scope}[scale=1.25]
				\node[vertex] (a1) at (-0.5,0) {};
				\node[vertex] (a2) at (-0.9,0) {};
				\node[vertex] (b1) at (0.5,0) {};
				\node[vertex] (b2) at (0.9,0) {};
				\node[hvertex] (c) at (0,1) {};
				\node[vertex] (d) at (0,0) {};
				\draw[genericgraph] (-0.7,0) circle[x radius=0.35cm, y radius=0.15cm];
				\draw[genericgraph] (0.7,0) circle[x radius=0.35cm, y radius=0.15cm];
				\draw[moreedges] (a1)edge(a2);
				\draw[moreedges] (b1)edge(b2);
				\draw[edge] (a1)edge(c) (a2)edge(c);
				\draw[hedge] (d)edge(c);
				\draw[oedge] (b1)edge(c) (b2)edge(c);
			\end{scope}
		\end{scope}
		\draw[construle] (-0.25,0) -- (0.25,0);
		\begin{scope}[xshift=2cm,yshift=-0.5cm]
			\begin{scope}[scale=1.25]
				\node[vertex] (a1) at (-0.5,0) {};
				\node[vertex] (a2) at (-0.9,0) {};
				\node[vertex] (b1) at (0.5,0) {};
				\node[vertex] (b2) at (0.9,0) {};
				\node[hvertex] (c) at (-0.2,1) {};
				\node[nvertex] (c2) at (0.2,1) {};
				\node[vertex] (d) at (0,0) {};
				\draw[genericgraph] (-0.7,0) circle[x radius=0.35cm, y radius=0.15cm];
				\draw[genericgraph] (0.7,0) circle[x radius=0.35cm, y radius=0.15cm];
				\draw[moreedges] (a1)edge(a2);
				\draw[moreedges] (b1)edge(b2);
				\draw[hedge] (d)edge(c);
				\draw[edge] (a1)edge(c) (a2)edge(c);
				\draw[nedge] (b1)edge(c2) (b2)edge(c2) (d)edge(c2) (c)edge(c2);
			\end{scope}
		\end{scope}
	\end{tikzpicture}
	\qquad
	\begin{tikzpicture}[scale=1]
		\node[eframe] (es) at (0,-1) {spider splitting};
		\draw[eframe] (es) -- ++(3.9,0) -- ++ (0,2) -- ++(-7.1,0) -- ++(0,-2) -- (es);
		\begin{scope}[xshift=-2cm]
			\begin{scope}[scale=1.5]
				\node[hvertex] (v) at (0,0) {};
				\node[vertex] (a1) at (-0.5,0.2) {};
				\node[vertex] (a2) at (-0.5,-0.2) {};
				\node[vertex] (b1) at (0.5,0.2) {};
				\node[vertex] (b2) at (0.5,-0.2) {};
				\node[vertex] (w1) at (0,-0.5) {};
				\node[vertex] (w2) at (0,0.5) {};
				\draw[genericgraph] (-0.5,0) circle[x radius=0.15cm, y radius=0.35cm];
				\draw[genericgraph] (0.5,0) circle[x radius=0.15cm, y radius=0.35cm];
				\draw[moreedges] (a1)edge(a2);
				\draw[moreedges] (b1)edge(b2);
				\draw[edge] (a1)edge(v) (a2)edge(v);
				\draw[hedge] (w1)edge(v) (w2)edge(v);
				\draw[oedge] (b1)edge(v) (b2)edge(v);
			\end{scope}
		\end{scope}
		\draw[construle] (-0.25,0) -- (0.25,0);
		\begin{scope}[xshift=2.4cm]
			\begin{scope}[scale=1.5]
				\node[hvertex] (v) at (-0.2,0) {};
				\node[nvertex] (vp) at (0.2,0) {};
				\node[vertex] (a1) at (-0.7,0.2) {};
				\node[vertex] (a2) at (-0.7,-0.2) {};
				\node[vertex] (b1) at (0.7,0.2) {};
				\node[vertex] (b2) at (0.7,-0.2) {};
				\node[vertex] (w1) at (0,-0.5) {};
				\node[vertex] (w2) at (0,0.5) {};
				\draw[genericgraph] (-0.7,0) circle[x radius=0.15cm, y radius=0.35cm];
				\draw[genericgraph] (0.7,0) circle[x radius=0.15cm, y radius=0.35cm];
				\draw[moreedges] (a1)edge(a2);
				\draw[moreedges] (b1)edge(b2);
				\draw[edge] (a1)edge(v) (a2)edge(v);
				\draw[hedge] (w1)edge(v) (w2)edge(v);
				\draw[nedge] (b1)edge(vp) (b2)edge(vp) (w1)edge(vp) (w2)edge(vp);
			\end{scope}
		\end{scope}
	\end{tikzpicture}
	\caption{Splitting vertices to get larger minimally rigid graphs. The blue vertex is split, green vertices are new. Dots indicate that there might be more vertices and edges.}
	\label{fig:splitting}
\end{figure}
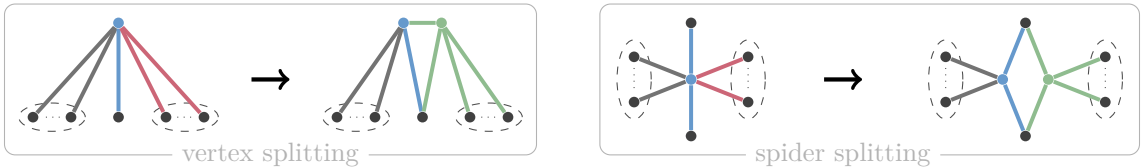

It is known from \cite[Lem~4.6]{Jackson2018} that the increase of the number of realizations by vertex splitting is at least two.
\Cref{tab:Laman_Increase_Vsplit} collects a representative graph $G$ for each number of vertices and a graph $G'$ that is obtained from $G$ by vertex splitting with the respective number of realizations.
We can see that the factor can be significantly higher than the minimum.
\begin{proposition}
	Assume that there is a universal constant $k_2$ bounding $\lamII{G'}/\lamII{G}$ from above for all minimally rigid graphs $G$ and $G'$ where $G'$ is obtained from $G$ by a vertex splitting.
	Then $k_2 \geq 11.34$.
\end{proposition}

\begin{proposition}
	Assume that there are universal constants $k_1$ and $k_2$ bounding $\lamII{G'}/\lamII{G}$ from below and above respectively for all minimally rigid graphs $G$ and $G'$ where $G'$ is obtained from $G$ by a spider splitting.
	Then $k_1\leq 1.97$ and $k_2 \geq 10.72$.
\end{proposition}

\subsection{Sphere}
Since it is the same class of graphs that are minimally rigid as in the plane, also the constructions steps are the same.
However, since the graphs that achieve high number of realizations are different, it is not surprising that also certain construction steps show different behavior on the sphere.

It is known that 0-extensions increase the realization count by a factor of 2 (compare \cite[Lem~7.1]{PlaneSphere}).

\subsubsection{Extension constructions}
Similarly to the plane the number of spherical realizations of a graph $G'$ after a 1-extension step of type~E1a on a graph $G$ is increased by a factor of two for all graphs $G'$ with at most 13 vertices.
\begin{conjecture}
	Let $G$ be a minimally rigid graph and $G'$ be obtained from $G$ by a 1-extension of type E1a.
	Then $\frac{\lamSII{G'}}{\lamSII{G}}=2$.
\end{conjecture}

While the minimal factor for a type~E1b step is again two in all the computations, the maximal ones differ from those in the plane (see \Cref{tab:SLaman_Increase_H2b}).
\begin{proposition}
	Assume that there are universal constants $k_1$ and $k_2$ bounding $\lamSII{G'}/\lamSII{G}$ from below and above respectively for all minimally rigid graphs $G$ and $G'$ where $G'$ is obtained from $G$ by a 1-extension of type E1b.
	Then $k_1\leq 2$ and $k_2 \geq 21$.
\end{proposition}

Also for steps of type~E1c the factors and representative graphs are different to the plane.
Indeed we do get even smaller minimal factors (see \Cref{tab:SLaman_Increase_H2c}).
\begin{proposition}
	Assume that there are universal constants $k_1$ and $k_2$ bounding $\lamSII{G'}/\lamSII{G}$ from below and above respectively for all minimally rigid graphs $G$ and $G'$ where $G'$ is obtained from $G$ by a 1-extension of type E1c.
	Then $k_1\leq 3/2$ and $k_2 \geq 26$.
\end{proposition}
Again since we have complete data for graph with less than 13 vertices and since $1.5$ is the lowest factor for all graphs with $6\leq n \leq 12$ we may conjecture that $k_1=1.5$.

\subsubsection{Splittings}
The results for vertex splitting on the sphere can be seen in \Cref{tab:SLaman_Increase_Vsplit}.
\begin{proposition}
	Assume that there are universal constants $k_1$ and $k_2$ bounding $\lamSII{G'}/\lamSII{G}$ from below and above respectively for all minimally rigid graphs $G$ and $G'$ where $G'$ is obtained from $G$ by a vertex splitting.
	Then $k_1\leq 2$ and $k_2 \geq 20.5$.
\end{proposition}

The results for spider splitting on the sphere can be seen in \Cref{tab:SLaman_Increase_Ssplit}.
\begin{proposition}
	Assume that there are universal constants $k_1$ and $k_2$ bounding $\lamSII{G'}/\lamSII{G}$ from below and above respectively for all minimally rigid graphs $G$ and $G'$ where $G'$ is obtained from $G$ by a spider splitting.
	Then $k_1\leq 2$ and $k_2 \geq 20.5$.
\end{proposition}

\subsection{Space}
For computations in three dimensional space we need to restrict ourselves to applying constructions on graphs with at most nine vertices since afterwards the probabilistic Gröbner basis approach is not efficient any more.
Nevertheless, we can make some interesting observations.
Note, however, that all computations are done with a random choice of edge lengths and are therefore probabilistic even after having checked several times.

\subsubsection{Extension constructions}
There is a generalization of the extension constructions to three dimensions.
In this case it is known that a third type is needed to construct all minimally 3-rigid graphs but with this third step rigidity is not always preserved.
Here, we restrict ourselves to certain subclasses of the construction steps which seem to be interesting.
Note however, that these subclasses do not suffice to construct all minimally 3-rigid graphs.

\Cref{fig:Henneberg3d:I} shows some rigidity preserving constructions.
It can be easily seen that the type~E0 extensions increase the number of realizations by a factor of two always (compare also \cite[Lem~7.1]{PlaneSphere}).
\begin{figure}[ht]
	\centering
	\begin{tikzpicture}[scale=1.25]
		\node[eframe] (ef) at (1.75,-0.35) {type E0};
		\draw[eframe] (ef) -- ++(2,0) -- ++ (0,1.466) -- ++(-4,0) -- ++(0,-1.466) -- (ef);
		\node[vertex] (a) at (0,0) {};
		\node[vertex] (b) at (1,0) {};
		\node[vertex] (c) at (0.5,0.866025) {};
		\node[vertex] (d) at (2.5,0) {};
		\node[vertex] (e) at (3.5,0) {};
		\node[vertex] (f) at (3,0.866025) {};
		\node[nvertex] (g) at (3,0.33) {};
		\draw[edgeq] (a)edge(b) (a)edge(c) (b)edge(c) (d)edge(e) (d)edge(f) (e)edge(f);
		\draw[nedge] (d)edge(g) (e)edge(g) (f)edge(g);
		\draw[construle] (1.5,0.5) -- (2,0.5);
	\end{tikzpicture}

	\begin{tikzpicture}[scale=1.25]
		\node[eframe] (ef) at (1.75,-0.35) {type E1s1};
		\draw[eframe] (ef) -- ++(2,0) -- ++ (0,1.6) -- ++(-4,0) -- ++(0,-1.6) -- (ef);
		\node[vertex] (a) at (0,0) {};
		\node[vertex] (b) at (1,0) {};
		\node[vertex] (c) at (1,1) {};
		\node[vertex] (d) at (0,1) {};
		\node[vertex] (e) at (2.5,0) {};
		\node[vertex] (f) at (3.5,0) {};
		\node[vertex] (g) at (3.5,1) {};
		\node[vertex] (h) at (2.5,1) {};
		\node[nvertex] (i) at (3,0.25) {};
		\draw[oedge] (a)edge(b);
		\draw[nedge] (e)edge(i) (f)edge(i) (g)edge(i) (h)edge(i);
		\draw[construle] (1.5,0.5) -- (2,0.5);
	\end{tikzpicture}
	\begin{tikzpicture}[scale=1.25]
		\node[eframe] (ef) at (1.75,-0.35) {type E1s30};
		\draw[eframe] (ef) -- ++(2,0) -- ++ (0,1.6) -- ++(-4,0) -- ++(0,-1.6) -- (ef);
		\node[vertex] (a) at (0,0) {};
		\node[vertex] (b) at (1,0) {};
		\node[vertex] (c) at (1,1) {};
		\node[vertex] (d) at (0,1) {};
		\node[vertex] (e) at (2.5,0) {};
		\node[vertex] (f) at (3.5,0) {};
		\node[vertex] (g) at (3.5,1) {};
		\node[vertex] (h) at (2.5,1) {};
		\node[nvertex] (i) at (3,0.25) {};
		\draw[oedge] (a)edge(b);
		\draw[edge] (a)edge(d) (b)edge(c) (c)edge(d);
		\draw[edge] (e)edge(h) (f)edge(g) (g)edge(h);
		\draw[nedge] (e)edge(i) (f)edge(i) (g)edge(i) (h)edge(i);
		\draw[construle] (1.5,0.5) -- (2,0.5);
	\end{tikzpicture}
	\begin{tikzpicture}[scale=1.25]
		\node[eframe] (ef) at (1.75,-0.35) {type E1s63};
		\draw[eframe] (ef) -- ++(2,0) -- ++ (0,1.6) -- ++(-4,0) -- ++(0,-1.6) -- (ef);
		\node[vertex] (a) at (0,0) {};
		\node[vertex] (b) at (1,0) {};
		\node[vertex] (c) at (1,1) {};
		\node[vertex] (d) at (0,1) {};
		\node[vertex] (e) at (2.5,0) {};
		\node[vertex] (f) at (3.5,0) {};
		\node[vertex] (g) at (3.5,1) {};
		\node[vertex] (h) at (2.5,1) {};
		\node[nvertex] (i) at (3,0.25) {};
		\draw[oedge] (a)edge(b);
		\draw[edge] (a)edge(c) (a)edge(d) (b)edge(c) (b)edge(d) (c)edge(d);
		\draw[edge] (e)edge(g) (e)edge(h) (f)edge(g) (f)edge(h) (g)edge(h);
		\draw[nedge] (e)edge(i) (f)edge(i) (g)edge(i) (h)edge(i);
		\draw[construle] (1.5,0.5) -- (2,0.5);
	\end{tikzpicture}
	\caption{Extension constructions of different types in dimension three where no or just one edge is deleted.}
	\label{fig:Henneberg3d:I}
\end{figure}
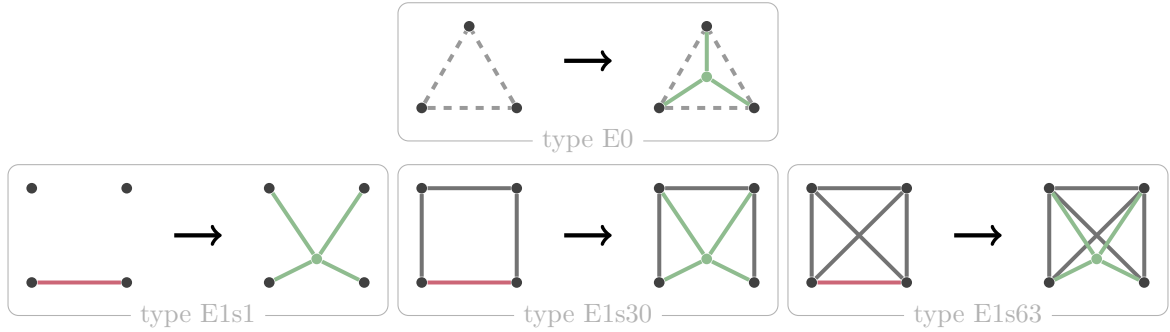

The extensions of type~E1s63 indeed did increase the number of realizations always by a factor of two in our experiments.
\begin{conjecture}
	Let $G$ be a minimally 3-rigid graph and $G'$ be obtained from $G$ by a 1-extension of type E1s63.
	Then $\frac{\lamIII{G'}}{\lamIII{G}}=2$.
\end{conjecture}
The conjecture is true for graphs $G$ with at most 9 vertices.
For the other types of 1-extensions the results can be found in \Cref{tab:3dLaman_Increase_H2s1,tab:3dLaman_Increase_H2s30}.

\begin{proposition}
	Assume that there are universal constants $k_1$ and $k_2$ bounding $\lamIII{G'}/\lamIII{G}$ from below and above respectively for all minimally 3-rigid graphs $G$ and $G'$ where $G'$ is obtained from $G$ by a 1-extension.\\
	If the 1-extension is of type~E1s1 then $k_1\leq 1.4$ and $k_2 \geq 16$.\\
	If the 1-extension is of type~E1s30 then $k_1\leq 1.5$ and $k_2 \geq 20$.
\end{proposition}

There are two major construction steps that delete two edges and add a vertex of degree five.
In one case the two edges do not share a vertex. It is therefore often called X-replacement.
In the other case they do share a vertex and it is called V-replacement.
We want to be a bit more specific and distinguish classes depending on which other edges are in the induced subgraph on the five chosen vertices (see \Cref{fig:Henneberg3d:II}).
Note that the classes we investigate are not complete.
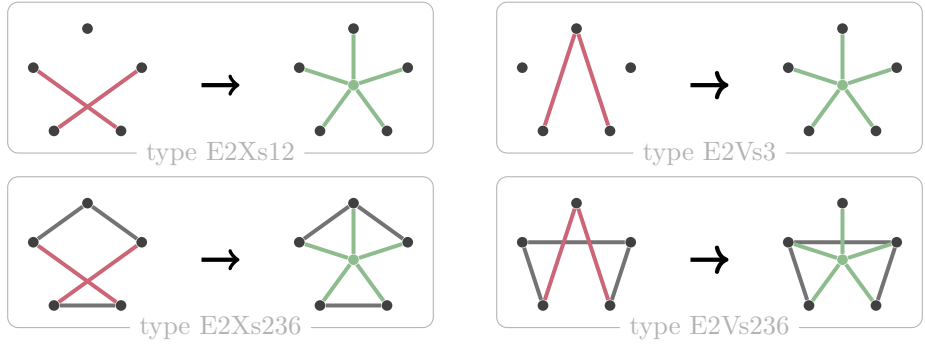
\begin{figure}[ht]
	\centering
	\begin{tikzpicture}[scale=1]
		\node[eframe] (ef) at ($(270:0.9)+(1.75,0)$) {type E2Xs12};
		\draw[eframe] (ef) -- ++(2.8,0) -- ++ (0,2) -- ++(-5.6,0) -- ++(0,-2) -- (ef);
		\begin{scope}[scale=0.75]
			\node[vertex] (a) at (0.587785, -0.809017) {};
			\node[vertex] (b) at (0.951057, 0.309017) {};
			\node[vertex] (c) at (0., 1.) {};
			\node[vertex] (d) at (-0.951057,0.309017) {};
			\node[vertex] (e) at (-0.587785, -0.809017) {};

			\draw[oedge] (a)edge(d) (b)edge(e);
		\end{scope}
		\draw[construle] (1.5,0.) -- (2,0.);
		\begin{scope}[xshift=3.5cm,scale=0.75]
			\node[vertex] (a) at (0.587785, -0.809017) {};
			\node[vertex] (b) at (0.951057, 0.309017) {};
			\node[vertex] (c) at (0., 1.) {};
			\node[vertex] (d) at (-0.951057,0.309017) {};
			\node[vertex] (e) at (-0.587785, -0.809017) {};
			\node[nvertex] (f) at (0,0) {};

			\draw[nedge] (a)edge(f) (b)edge(f) (c)edge(f) (d)edge(f) (e)edge(f);
		\end{scope}
	\end{tikzpicture}
	\qquad
	\begin{tikzpicture}[scale=1]
		\node[eframe] (ef) at ($(270:0.9)+(1.75,0)$) {type E2Vs3};
		\draw[eframe] (ef) -- ++(2.8,0) -- ++ (0,2) -- ++(-5.6,0) -- ++(0,-2) -- (ef);
		\begin{scope}[scale=0.75]
			\node[vertex] (a) at (0.587785, -0.809017) {};
			\node[vertex] (b) at (0.951057, 0.309017) {};
			\node[vertex] (c) at (0., 1.) {};
			\node[vertex] (d) at (-0.951057,0.309017) {};
			\node[vertex] (e) at (-0.587785, -0.809017) {};

			\draw[oedge] (a)edge(c) (c)edge(e);
		\end{scope}
		\draw[ultra thick,->] (1.5,0.) -- (2,0.);
		\begin{scope}[xshift=3.5cm,scale=0.75]
			\node[vertex] (a) at (0.587785, -0.809017) {};
			\node[vertex] (b) at (0.951057, 0.309017) {};
			\node[vertex] (c) at (0., 1.) {};
			\node[vertex] (d) at (-0.951057,0.309017) {};
			\node[vertex] (e) at (-0.587785, -0.809017) {};
			\node[nvertex] (f) at (0,0) {};

			\draw[nedge] (a)edge(f) (b)edge(f) (c)edge(f) (d)edge(f) (e)edge(f);
		\end{scope}
	\end{tikzpicture}

	\begin{tikzpicture}[scale=1]
		\node[eframe] (ef) at ($(270:0.9)+(1.75,0)$) {type E2Xs236};
		\draw[eframe] (ef) -- ++(2.8,0) -- ++ (0,2) -- ++(-5.6,0) -- ++(0,-2) -- (ef);
		\begin{scope}[scale=0.75]
			\node[vertex] (a) at (0.587785, -0.809017) {};
			\node[vertex] (b) at (0.951057, 0.309017) {};
			\node[vertex] (c) at (0., 1.) {};
			\node[vertex] (d) at (-0.951057,0.309017) {};
			\node[vertex] (e) at (-0.587785, -0.809017) {};

			\draw[oedge] (a)edge(d) (b)edge(e);
			\draw[edge] (b)edge(c) (c)edge(d) (a)edge(e);
		\end{scope}
		\draw[construle] (1.5,0.) -- (2,0.);
		\begin{scope}[xshift=3.5cm,scale=0.75]
			\node[vertex] (a) at (0.587785, -0.809017) {};
			\node[vertex] (b) at (0.951057, 0.309017) {};
			\node[vertex] (c) at (0., 1.) {};
			\node[vertex] (d) at (-0.951057,0.309017) {};
			\node[vertex] (e) at (-0.587785, -0.809017) {};
			\node[nvertex] (f) at (0,0) {};

			\draw[nedge] (a)edge(f) (b)edge(f) (c)edge(f) (d)edge(f) (e)edge(f);
			\draw[edge] (b)edge(c) (c)edge(d) (a)edge(e);
		\end{scope}
	\end{tikzpicture}
	\qquad
	\begin{tikzpicture}[scale=1]
		\node[eframe] (ef) at ($(270:0.9)+(1.75,0)$) {type E2Vs236};
		\draw[eframe] (ef) -- ++(2.8,0) -- ++ (0,2) -- ++(-5.6,0) -- ++(0,-2) -- (ef);
		\begin{scope}[scale=0.75]
			\node[vertex] (a) at (0.587785, -0.809017) {};
			\node[vertex] (b) at (0.951057, 0.309017) {};
			\node[vertex] (c) at (0., 1.) {};
			\node[vertex] (d) at (-0.951057,0.309017) {};
			\node[vertex] (e) at (-0.587785, -0.809017) {};

			\draw[edge] (a)edge(b) (b)edge(d) (d)edge(e);
			\draw[oedge] (a)edge(c) (c)edge(e);
		\end{scope}
		\draw[ultra thick,->] (1.5,0.) -- (2,0.);
		\begin{scope}[xshift=3.5cm,scale=0.75]
			\node[vertex] (a) at (0.587785, -0.809017) {};
			\node[vertex] (b) at (0.951057, 0.309017) {};
			\node[vertex] (c) at (0., 1.) {};
			\node[vertex] (d) at (-0.951057,0.309017) {};
			\node[vertex] (e) at (-0.587785, -0.809017) {};
			\node[nvertex] (f) at (0,0) {};

			\draw[edge] (a)edge(b) (b)edge(d) (d)edge(e);
			\draw[nedge] (a)edge(f) (b)edge(f) (c)edge(f) (d)edge(f) (e)edge(f);
		\end{scope}
	\end{tikzpicture}

	\caption{2-extensions of different types in dimension three where two edges are deleted.}
	\label{fig:Henneberg3d:II}
\end{figure}

The resulting factors can be seen in \Cref{tab:3dLaman_Increase_H3Xs12,tab:3dLaman_Increase_H3Vs3,tab:3dLaman_Increase_H3Xs236,tab:3dLaman_Increase_H3Vs236}.

We can consider similar classes of construction steps as in \Cref{fig:Henneberg3d:II} where the subgraph on the five chosen vertices is
ismorphic to other graphs, for instance the complete graph on five vertices minus one edge (511 in integer notation).
We call this classes E2Vs511 and E2Xs511 depending on whether the deleted edges share a vertex or not.
Note that, there are several choices how the edges that are deleted are placed within this subgraph.
For this reason we do not put a figure here.
Note also that the complete graph on five vertices can never be a subgraph of a minimally 3-rigid graph in dimension three because it has too many edges.
In that sense 511 is the largest possible.
Similarly we do so with the induced subgraphs on five vertices with integer representation 239 (see \Cref{fig:smallm2rinteger}).
We use this example because it gives the smallest factor we have found so far.
The results can be seen in \Cref{tab:3dLaman_Increase_H3Xs239,tab:3dLaman_Increase_H3Vs239}.

Steps of type E2Xs511 and E2Vs511 on graphs with at most nine vertices always increase the number of realizations by a factor of two.
\begin{conjecture}
	Let $G$ be a minimally 3-rigid graph and $G'$ be obtained from $G$ by a 1-extension of type E2Xs511 or E2Vs511.
	Then $\frac{\lamIII{G'}}{\lamIII{G}}=2$.
\end{conjecture}
The conjecture is true for graphs $G$ with at most 9 vertices.

\begin{proposition}
	Assume that there are universal constants $k_1$ and $k_2$ bounding $\lamIII{G'}/\lamIII{G}$ from below and above respectively for all minimally 3-rigid graphs $G$ and $G'$ where $G'$ is obtained from $G$ by a 2-extension.\\
	If the 2-extension is of type~E2Xs12 then $k_1\leq 0.875$ and $k_2 \geq 14$.\\
	If the 2-extension is of type~E2Vs3 then $k_1\leq 0.72$ and $k_2 \geq 14$.\\
	If the 2-extension is of type~E2Xs236 then $k_1\leq 0.95$ and $k_2 \geq 20$.\\
	If the 2-extension is of type~E2Vs236 then $k_1\leq 0.43$ and $k_2 \geq 26$.\\
	If the 2-extension is of type~E2Xs239 then $k_1\leq 1$ and $k_2 \geq 20$.\\
	If the 2-extension is of type~E2Vs239 then $k_1\leq 0.25$ and $k_2 \geq 22$.
\end{proposition}

\subsubsection{Splittings}
Vertex and spider splitting can be generalized to any dimension (compare \cite{NixonRoss}).
Vertex splitting preserves $d$-rigidity \cite{WhiteleyVertex}. Spider splitting seems to preserve 3-rigidity as well.
\Cref{fig:splitting3d} shows an illustration for dimension three.

\begin{figure}[ht]
	\centering
	\begin{tikzpicture}[]
		\node[eframe] (es) at (0,-1) {vertex splitting};
		\draw[eframe] (es) -- ++(3.9,0) -- ++ (0,2) -- ++(-7.1,0) -- ++(0,-2) -- (es);
		\begin{scope}[xshift=-2cm]
			\begin{scope}[scale=1.5]
				\node[hvertex] (v) at (0,0) {};
				\node[vertex] (a1) at (-0.5,0.2) {};
				\node[vertex] (a2) at (-0.5,-0.2) {};
				\node[vertex] (b1) at (0.5,0.2) {};
				\node[vertex] (b2) at (0.5,-0.2) {};
				\node[vertex] (w1) at (0,-0.5) {};
				\node[vertex] (w2) at (0,0.5) {};
				\draw[genericgraph] (-0.5,0) circle[x radius=0.15cm, y radius=0.35cm];
				\draw[genericgraph] (0.5,0) circle[x radius=0.15cm, y radius=0.35cm];
				\draw[moreedges] (a1)edge(a2);
				\draw[moreedges] (b1)edge(b2);
				\draw[edge] (a1)edge(v) (a2)edge(v);
				\draw[hedge] (w1)edge(v) (w2)edge(v);
				\draw[oedge] (b1)edge(v) (b2)edge(v);
			\end{scope}
		\end{scope}
		\draw[construle] (-0.25,0) -- (0.25,0);
		\begin{scope}[xshift=2.4cm]
			\begin{scope}[scale=1.5]
				\node[hvertex] (v) at (-0.2,0) {};
				\node[nvertex] (vp) at (0.2,0) {};
				\node[vertex] (a1) at (-0.7,0.2) {};
				\node[vertex] (a2) at (-0.7,-0.2) {};
				\node[vertex] (b1) at (0.7,0.2) {};
				\node[vertex] (b2) at (0.7,-0.2) {};
				\node[vertex] (w1) at (0,-0.5) {};
				\node[vertex] (w2) at (0,0.5) {};
				\draw[genericgraph] (-0.7,0) circle[x radius=0.15cm, y radius=0.35cm];
				\draw[genericgraph] (0.7,0) circle[x radius=0.15cm, y radius=0.35cm];
				\draw[moreedges] (a1)edge(a2);
				\draw[moreedges] (b1)edge(b2);
				\draw[edge] (a1)edge(v) (a2)edge(v);
				\draw[hedge] (w1)edge(v) (w2)edge(v);
				\draw[nedge] (b1)edge(vp) (b2)edge(vp) (w1)edge(vp) (w2)edge(vp) (v)edge(vp);
			\end{scope}
		\end{scope}
	\end{tikzpicture}
	\qquad
	\begin{tikzpicture}[scale=1]
		\node[eframe] (es) at (0,-1) {spider splitting};
		\draw[eframe] (es) -- ++(3.9,0) -- ++ (0,2) -- ++(-7.1,0) -- ++(0,-2) -- (es);
		\begin{scope}[xshift=-2cm]
			\begin{scope}[scale=1.5]
				\node[hvertex] (v) at (0,0) {};
				\node[vertex] (a1) at (-0.5,0.2) {};
				\node[vertex] (a2) at (-0.5,-0.2) {};
				\node[vertex] (b1) at (0.5,0.2) {};
				\node[vertex] (b2) at (0.5,-0.2) {};
				\node[vertex] (w1) at (0,-0.5) {};
				\node[vertex] (w2) at (-0.25,0.5) {};
				\node[vertex] (w3) at (0.25,0.5) {};
				\draw[genericgraph] (-0.5,0) circle[x radius=0.15cm, y radius=0.35cm];
				\draw[genericgraph] (0.5,0) circle[x radius=0.15cm, y radius=0.35cm];
				\draw[moreedges] (a1)edge(a2);
				\draw[moreedges] (b1)edge(b2);
				\draw[edge] (a1)edge(v) (a2)edge(v);
				\draw[hedge] (w1)edge(v) (w2)edge(v) (w3)edge(v);
				\draw[oedge] (b1)edge(v) (b2)edge(v);
			\end{scope}
		\end{scope}
		\draw[construle] (-0.25,0) -- (0.25,0);
		\begin{scope}[xshift=2.4cm]
			\begin{scope}[scale=1.5]
				\node[hvertex] (v) at (-0.2,0) {};
				\node[nvertex] (vp) at (0.2,0) {};
				\node[vertex] (a1) at (-0.7,0.2) {};
				\node[vertex] (a2) at (-0.7,-0.2) {};
				\node[vertex] (b1) at (0.7,0.2) {};
				\node[vertex] (b2) at (0.7,-0.2) {};
				\node[vertex] (w1) at (0,-0.5) {};
				\node[vertex] (w2) at (-0.25,0.5) {};
				\node[vertex] (w3) at (0.25,0.5) {};
				\draw[genericgraph] (-0.7,0) circle[x radius=0.15cm, y radius=0.35cm];
				\draw[genericgraph] (0.7,0) circle[x radius=0.15cm, y radius=0.35cm];
				\draw[moreedges] (a1)edge(a2);
				\draw[moreedges] (b1)edge(b2);
				\draw[edge] (a1)edge(v) (a2)edge(v);
				\draw[hedge] (w1)edge(v) (w2)edge(v) (w3)edge(v);
				\draw[nedge] (b1)edge(vp) (b2)edge(vp) (w1)edge(vp) (w2)edge(vp) (w3)edge(vp);
			\end{scope}
		\end{scope}
	\end{tikzpicture}
	\caption{Splitting vertices to get larger minimally 3-rigid graphs. The blue vertex is split, green vertices are new. Dots indicate that there might be more vertices and edges.}
	\label{fig:splitting3d}
\end{figure}
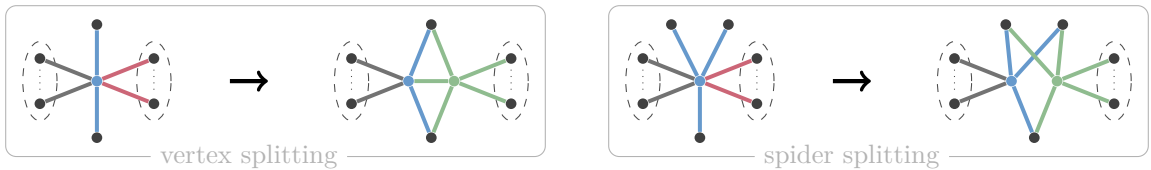

\begin{proposition}
	Assume that there are universal constants $k_1$ and $k_2$ bounding $\lamIII{G'}/\lamIII{G}$ from below and above respectively for all minimally 3-rigid graphs $G$ and $G'$ where $G'$ is obtained from $G$ by a vertex splitting operation.
	Then $k_1\leq 2$ and $k_2 \geq 26$.
\end{proposition}

\begin{proposition}
	Assume that there are universal constants $k_1$ and $k_2$ bounding $\lamIII{G'}/\lamIII{G}$ from below and above respectively for all minimally 3-rigid graphs $G$ and $G'$ where $G'$ is obtained from $G$ by a spider splitting operation.
	Then $k_1\leq 2$ and $k_2 \geq 36$.
\end{proposition}

\section{Conclusion}
We have improved the bounds for the maximal number of realizations for a given number of vertices in the plane, on the sphere and in higher dimensional spaces.
Additionally we have given bounds on the minimal such number for higher dimensions.
Since a combinatorial algorithm for higher dimensions is missing, the computations are done probabilistically using Gröbner bases with random edge lengths.
All the values presented in this paper have been computed several times to gain some level of trust.
Still it is important to gain further combinatorial knowledge of minimally rigid graphs in higher dimensions.

The results we obtained from the computations leave a basis for many conjectures on the realization numbers and the influence of construction steps on them.
These shall be subject to further research.

\addcontentsline{toc}{section}{Acknowledgments}
\section*{Acknowledgments}
G.\ Grasegger was supported by the Austrian Science Fund (FWF): 10.55776/P31888.

\appendix
\newpage
\section{Appendix --- Graph Encodings}\label{appendix:encoding}
In this section we present details on the graph encoding used in this paper and in \cite{LowerBounds,ZenodoData,ZenodoAlg}.

We encode a graph by the integer that is obtained in the following way:
\begin{itemize}
	\item flatten the upper right triangle of its adjacency matrix and
	\item interpret this binary sequence as an integer
\end{itemize}
Note that diagonal entries of the adjacency matrix are always zero since we only allow simple graphs without loops.
Hence, we consider only entries above the main diagonal.
\begin{center}
	\begin{tikzpicture}[scale=1.5]
		\begin{scope}
		\node[vertex] (a) at (0,0) {};
		\node[vertex] (b) at (1,0) {};
		\node[vertex] (c) at (0.5,0.866025) {};
		\draw[edge] (a)edge(b) (a)edge(c) (b)edge(c);
		\draw[dconstrule] (1.5,0.433) -- (2.5,0.433);
		\draw[dconstrule] (4.75,0.433) -- (5.75,0.433);
		\node at (3.7,0.433) {
			$
			\begin{pmatrix}
				0 & 1 & 1 \\
				1 & 0 & 1 \\
				1 & 1 & 0
			\end{pmatrix}
			$
		};
		\node at (6.75,0.433) {$(111)_2=7$};
	\end{scope}
	\begin{scope}[yshift=-1.75cm]
		\node[vertex,label={[labelsty]180:1}] (a) at (-0.5,0) {};
		\node[vertex,label={[labelsty]0:3}] (b) at (0.5,0) {};
		\node[vertex,label={[labelsty]180:4}] (c) at (0,0.866025) {};
		\node[vertex,label={[labelsty]0:2}] (d) at (1,0.866025) {};
		\draw[edge] (a)edge(b) (a)edge(c) (b)edge(c) (b)edge(d) (c)edge(d);
		\draw[dconstrule] (1.5,0.433) -- (2.5,0.433);
		\draw[dconstrule] (4.75,0.433) -- (5.75,0.433);
		\node at (3.7,0.433) {
			$
			\begin{pmatrix}
				0 & 0 & 1 & 1\\
				0 & 0 & 1 & 1\\
				1 & 1 & 0 & 1\\
				1 & 1 & 1 & 0
			\end{pmatrix}
			$
		};
		\node at (6.75,0.433) {$(011111)_2=31$};
		\end{scope}
	\end{tikzpicture}
\end{center}
With this encoding isomorphic graphs might be represented by different numbers.

\subsection{Small graphs}
\begin{figure}[ht]
	\centering
	\begin{tabular}{ccccccc}
		\begin{tikzpicture}[scale=0.5]
			\node[vertex] (1) at (1., 0.) {};
			\node[vertex] (2) at (-1., 0.) {};
			\draw[edge] (1)edge(2);
		\end{tikzpicture}
		&
		\begin{tikzpicture}[scale=0.5]
			\node[vertex] (1) at (210:1) {};
			\node[vertex] (2) at (0,1) {};
			\node[vertex] (3) at (330:1) {};
			\draw[edge] (1)edge(2) (1)edge(3) (2)edge(3);
		\end{tikzpicture}
		&
		\begin{tikzpicture}[]
			\node[vertex] (1) at (0.9, 0) {};
			\node[vertex] (2) at (-0.9, 0) {};
			\node[vertex] (3) at (0, -0.4) {};
			\node[vertex] (4) at (0, 0.4) {};
			\draw[edge] (1)edge(3) (1)edge(4) (2)edge(3) (2)edge(4) (3)edge(4);
		\end{tikzpicture}
		&
		\begin{tikzpicture}[]
			\node[vertex] (1) at (0, -0.6) {};
			\node[vertex] (2) at (0, 0.6) {};
			\node[vertex] (3) at (1.7, 0) {};
			\node[vertex] (4) at (0.7, -0.3) {};
			\node[vertex] (5) at (0.7, 0.3) {};
			\draw[edge] (1)edge(4) (1)edge(5) (2)edge(4) (2)edge(5) (3)edge(4) (3)edge(5) (4)edge(5);
		\end{tikzpicture}
		&
		\begin{tikzpicture}[]
			\node[vertex] (1) at (1, 0.7) {};
			\node[vertex] (2) at (-1, 0.7) {};
			\node[vertex] (3) at (-0.5, 0) {};
			\node[vertex] (4) at (0.5, 0) {};
			\node[vertex] (5) at (0, 0.7) {};
			\draw[edge] (1)edge(4) (1)edge(5) (2)edge(3) (2)edge(5) (3)edge(4) (3)edge(5) (4)edge(5);
		\end{tikzpicture}
		&
		\begin{tikzpicture}[]
			\node[vertex] (1) at (0, 0) {};
			\node[vertex] (2) at (1.7, 0.4) {};
			\node[vertex] (3) at (1.7, -0.4) {};
			\node[vertex] (4) at (0.9, -0.4) {};
			\node[vertex] (5) at (0.9, 0.4) {};
			\draw[edge] (1)edge(4) (1)edge(5) (2)edge(3) (2)edge(4) (2)edge(5) (3)edge(4) (3)edge(5);
		\end{tikzpicture}
		&
		\begin{tikzpicture}[]
			\node[vertex] (1) at (0, 0.5) {};
			\node[vertex] (2) at (1.64, 0) {};
			\node[vertex] (3) at (1.14, -0.5) {};
			\node[vertex] (4) at (0, -0.5) {};
			\node[vertex] (5) at (0.5, 0) {};
			\node[vertex] (6) at (1.14, 0.5) {};
			\draw[edge] (1)edge(4) (1)edge(5) (1)edge(6) (2)edge(3) (2)edge(5) (2)edge(6) (3)edge(4) (3)edge(6) (4)edge(5);
		\end{tikzpicture}\\
		1 & 7 & 31 & 223 & 239 & 254 & 7916
	\end{tabular}
	\caption{Small minimally 2-rigid graphs and their integer representations}
	\label{fig:smallm2rinteger}
\end{figure}

\begin{figure}[ht]
	\centering
	\begin{tabular}{ccccccc}
		\begin{tikzpicture}[scale=0.5]
			\node[vertex] (1) at (210:1) {};
			\node[vertex] (2) at (0,1) {};
			\node[vertex] (3) at (330:1) {};
			\draw[edge] (1)edge(2) (1)edge(3) (2)edge(3);
		\end{tikzpicture}
		&
		\begin{tikzpicture}[scale=0.75]
			\node[vertex] (1) at (-1, 0) {};
			\node[vertex] (2) at (0, 1) {};
			\node[vertex] (3) at (1, 0) {};
			\node[vertex] (4) at (-0.35, -0.65) {};
			\draw[edge] (1)edge(2) (1)edge(3) (1)edge(4) (2)edge(3) (2)edge(4) (3)edge(4);
		\end{tikzpicture}
		&
		\begin{tikzpicture}[]
			\node[vertex] (1) at (1.8, 0) {};
			\node[vertex] (2) at (0., 0) {};
			\node[vertex] (3) at (0.8, 0.6) {};
			\node[vertex] (4) at (1.1, 0) {};
			\node[vertex] (5) at (0.8, -0.6) {};
			\draw[edge] (1)edge(3) (1)edge(4) (1)edge(5) (2)edge(3) (2)edge(4) (2)edge(5) (3)edge(4) (3)edge(5) (4)edge(5);
		\end{tikzpicture}\\
		7 & 63 & 511
	\end{tabular}
	\caption{Small minimally 3-rigid graphs and their integer representations}
	\label{fig:smallm3rinteger}
\end{figure}

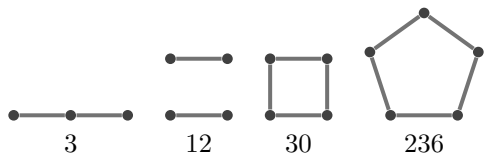
\begin{figure}[ht]
	\centering
	\begin{tabular}{ccccccc}
		\begin{tikzpicture}[scale=0.75]
			\node[vertex] (1) at (0,0) {};
			\node[vertex] (2) at (1,0) {};
			\node[vertex] (3) at (2,0) {};
			\draw[edge] (1)edge(2)(2)edge(3);
		\end{tikzpicture}
		&
		\begin{tikzpicture}[scale=0.75]
			\node[vertex] (1) at (0, 0) {};
			\node[vertex] (2) at (1, 0) {};
			\node[vertex] (3) at (1, 1) {};
			\node[vertex] (4) at (0, 1) {};
			\draw[edge] (1)edge(2) (3)edge(4);
		\end{tikzpicture}
		&
		\begin{tikzpicture}[scale=0.75]
			\node[vertex] (1) at (0, 0) {};
			\node[vertex] (2) at (1, 0) {};
			\node[vertex] (3) at (1, 1) {};
			\node[vertex] (4) at (0, 1) {};
			\draw[edge] (1)edge(2) (2)edge(3) (3)edge(4) (4)edge(1);
		\end{tikzpicture}
		&
		\begin{tikzpicture}[scale=0.75]
			\node[vertex] (1) at (90:1) {};
			\node[vertex] (2) at (90+72:1) {};
			\node[vertex] (3) at (90+2*72:1) {};
			\node[vertex] (4) at (90+3*72:1) {};
			\node[vertex] (5) at (90+4*72:1) {};
			\draw[edge] (1)edge(2) (2)edge(3) (3)edge(4) (4)edge(5) (5)edge(1);
		\end{tikzpicture}\\
		3 & 12 & 30 & 236
	\end{tabular}
	\caption{Small graphs and their integer representations}
	\label{fig:smallinteger}
\end{figure}

\subsection{Encodings for special graphs}\label{sec:enc:sepcial}
In this section we collect encodings for the graphs that were used to obtain the results.

Encodings for 2-rigid graphs on less than 13 vertices important for realization bounds in the plane can be found in \cite{LowerBounds}.
Hence, we omit them here and focus on graph with at least 13 vertices.
Similarly for 3-rigid graphs on less than 11 vertices.

\begin{table}[H]
	\centering\scriptsize
	\begin{tabular}{rrr}
		\toprule
		n  & graph encoding & $\lamIIfunc$ \\\midrule
		13 & 2731597771584836257824 & 15536 \\
		14 & 3932631430916370534240769 & 42780 \\
		15 & 94091005932357252120217796609 & 112752 \\
		16 & 892527555716690691964688718172672 & 312636 \\
		17 & 33232241523797605273285475596937609216 & 877960 \\
		18 & 609860480316548750422025574107604147273728 & 2414388 \\
		19 & 24856974395124526078390940692717703630776284260353 & 6505210 \\
		20 & 969788454592744849800428797243356986925558902858752256 & 17971016 \\
		21 & 12968494532131258973493494663506962165005447310533060506813712 & 49630180 \\
		22 & 33713110476986119025028298068249028967893417998630030702195417023572 & 138913994 \\
		23 & 71557306413135081116091456327584755953673370689478276983012212630869839968 & 367574230 \\\bottomrule
	\end{tabular}
	\caption{Graphs with high number of realizations used for getting the first column of \Cref{tab:bounds}.}
	\label{tab:enc:1-fan}
\end{table}

\begin{table}[H]
	\centering\scriptsize
	\begin{tabular}{rrr}
		\toprule
		n  & graph encoding & $\lamIIfunc$ \\\midrule
		13 & 517844367551685511200 & 15268 \\
		14 & 3940891540403122035892993 & 40096 \\
		15 & 61904862714746864303321923585 & 109496 \\
		16 & 2028424292758336805279710135194112 & 294964 \\
		17 & 33232403864928934634066282513749360672 & 819988 \\
		18 & 1219637218659582589200662236858040990044672 & 2247900 \\
		19 & 9499941923504027159504793171824935136863824564224 & 6108404 \\
		20 & 4645400869097374823966494459590862173654520880913719617 & 16771820 \\
		21 & 841159220643262909846596018548625207381025120635924158386256 & 46479708 \\
		22 & 26992863437232946557425749373646012807725030614773535644519493534898 & 129305872 \\
		23 & 7516002197842324262732702438237139053241300938418728228570663606823288844 & 356386152 \\
		24 & 1186886590055951433948405571819937144644542807550453790223319248355507845603841 & 928830292 \\\bottomrule
	\end{tabular}
	\caption{Graphs with high number of realizations and a triangle subgraph used for getting \Cref{tab:bounds}.}
	\label{tab:enc:7-fan}
\end{table}

\begin{table}[H]
	\centering\scriptsize
	\begin{tabular}{rrr}
		\toprule
		n  & graph encoding & $\lamIIfunc$ \\\midrule
		13 & 1919470438696836485632 & 13216 \\
		14 & 106387392391517896392056867 & 32984 \\
		15 & 101535565889081346135698980865 & 87952 \\
		16 & 1054779394035008997129070783324160 & 236752 \\
		17 & 2043029311514306583969055346820823801857 & 636584 \\
		18 & 15333150053920178852347345525835627097559563 & 1778872 \\
		19 & 2375165956056111156778838587761972160510430298113 & 4917616 \\
		20 & 299319965592281883846087342823084798786086393695174944 & 13461568 \\
		21 & 175762302721380444691882816903810933661859536729624543496459 & 36814256 \\
		22 & 13743255016879215026359790585006300077192451062341070022867148867723 & 102580944 \\
		23 & 14300418759653373184331761931428003983606606927352291848347223111275972680 & 283111704 \\
		24 & 477063471291857555444014842231845351516224351757402206829357538163373848135729228 & 786712432 \\\bottomrule
	\end{tabular}
	\caption{Graphs with high number of realizations and a four vertex minimally rigid subgraph used for getting \Cref{tab:bounds}.}
	\label{tab:enc:31-fan}
\end{table}

\begin{table}[H]
	\centering\scriptsize
	\begin{tabular}{rrr}
		\toprule
		n  & graph encoding & $\lamIIfunc$ \\\midrule
		13 & 12995965739298874613921 & 9728 \\
		14 & 3930558837259231061885444 & 26432 \\
		15 & 103994519603634743595901386912 & 66800 \\
		16 & 1014368078774038833854590442577956 & 184128 \\
		17 & 17285846084153567892608510949622751232 & 489776 \\
		18 & 2439997381536868085951326631278682611662860 & 1336576 \\
		19 & 570905523907304499408862045351699438966424994337 & 3760640 \\
		20 & 1730078881661256116680977603720159619508975088476226051 & 10411888 \\
		21 & 59632693970547596625962262828713950408378399537682010685504 & 28709504 \\
		22 & 1750823378745286454968405754173395962173256608787492181501934714913 & 77935744 \\
		23 & 179446670409862114796636902481179502800869665176711932792232782596415572 & 220411552 \\
		24 & 1621117988222861364696506244132890911843704707085263796691486054849687756801808 & 611930960 \\\bottomrule
	\end{tabular}
	\caption{Graphs with high number of realizations and the five vertex minimally rigid subgraph 254 used for getting \Cref{tab:bounds}.}
	\label{tab:enc:254-fan}
\end{table}

\begin{table}[H]
	\centering\scriptsize
	\begin{tabular}{rrr}
		\toprule
		n  & graph encoding & $\lamIIfunc$ \\\midrule
		13 & 152150553443130661671424 & 10944 \\
		14 & 20259183535670335770018976 & 29184 \\
		15 & 32192788694337816251917943300 & 79296 \\
		16 & 1014137821927435072109531763867776 & 200400 \\
		17 & 101022220227035701532226246627656466962 & 552384 \\
		18 & 103141290816217564540696429003780764122218528 & 1469328 \\
		19 & 5115257647184278818181495355725235437833809891841 & 4009728 \\
		20 & 7087793566411652377916979289267325306059305780156383250 & 11281920 \\
		21 & 10043363267502127349251608648315257503005775992801180904742913 & 31235664 \\
		22 & 3527962398608386801203237161357185274033369718060390247382088159424 & 86128512 \\
		23 & 3699337307276164664254297200614860849733905113696802834819080497156639233 & 233807232 \\
		24 & 60269783027825379385089984094477284578849409328099245478975316545880239834007628 & 661234656 \\\bottomrule
	\end{tabular}
	\caption{Graphs with high number of realizations and a three-prism subgraph 7916 used for getting \Cref{tab:bounds}.}
	\label{tab:enc:7916-fan}
\end{table}

\begin{table}[ht]
	\centering\scriptsize
	\begin{tabular}{rr}
		\toprule
		n  & graph encodings \\\midrule
		10 & 4778440734593, 4847160401729, 1315755596577 \\
		11 & 18226779293308419, 18226916732259843 \\
		12 & 252695476130038944 \\
		13 & 6128220462188632473600, 14444026969064381092352, 14444027004180033704448, 76113284109793682046976\\\bottomrule
	\end{tabular}
	\caption{Graph encodings for \Cref{fig:maxsphere}.}
	\label{tab:enc-maxsphere}
\end{table}

\begin{table}[H]
	\centering\scriptsize
	\begin{tabular}{rrr}
		\toprule
		n  & graph encoding & $\lamSIIfunc$ \\\midrule
		6 & 7916 & 32 \\
		7 & 112525 & 64 \\
		8 & 170957470 & 192 \\
		9 & 2993854888 & 576 \\
		10 & 1315755596577 & 1536 \\
		11 & 18226779293308419 & 4352 \\
		12 & 252695476130038944 & 12288 \\
		13 & 14444026969064381092352 & 34816 \\
		14 & 24487910449801970454110476 & 98304 \\
		15 & 188197112811157555364613865540 & 274432 \\
		16 & 3083362077516059701951221258133504 & 815104 \\
		17 & 9307395696008843988318577033335947264 & 2195456 \\\bottomrule
	\end{tabular}
	\caption{Graphs with high number of spherical realizations used for getting the first column of \Cref{tab:spherebounds}. All of them have a triangle subgraph and are therefore also used for getting column two.}
	\label{tab:Senc:1-fan}
\end{table}

\begin{table}[H]
	\centering\scriptsize
	\begin{tabular}{rrr}
		\toprule
		n  & graph encoding & $\lamSIIfunc$ \\\midrule
		6 & 3326 & 16 \\
		7 & 1256267 & 64 \\
		8 & 104400062 & 128 \\
		9 & 11987422577 & 512 \\
		10 & 4778713554625 & 1280 \\
		11 & 1236508778848775 & 3584 \\
		12 & 1801944545326387724 & 10240 \\
		13 & 153699228822191169081856 & 28672 \\
		14 & 15720174686046767213974028 & 83968 \\
		15 & 277343309388923844643716079653 & 233472 \\
		16 & 4543606606848062640178932869234706 & 630784 \\
		17 & 74437903980450196378912688547861479460 & 1875968 \\\bottomrule
	\end{tabular}
	\caption{Graphs with high number of spherical realizations and a four vertex minimally rigid subgraph used for getting \Cref{tab:spherebounds}.}
	\label{tab:Senc:31-fan}
\end{table}

\begin{table}[H]
	\centering\scriptsize
	\begin{tabular}{rrr}
		\toprule
		n  & graph encoding & $\lamSIIfunc$ \\\midrule
		6 & 3326 & 16 \\
		7 & 101630 & 32 \\
		8 & 117884055 & 128 \\
		9 & 1008905132 & 256 \\
		10 & 220302198846 & 1024 \\
		11 & 111370751706302 & 2560 \\
		12 & 270850004139176705 & 8192 \\
		13 & 1041106309634028363937 & 20480 \\
		14 & 11192470356822440426349062 & 61440 \\
		15 & 52006190731483567881425789216 & 180224 \\
		16 & 1055010897268843451626836271841792 & 475136 \\
		17 & 18610043923532523055425244310227943712 & 1376256 \\\bottomrule
	\end{tabular}
	\caption{Graphs with high number of spherical realizations and a five vertex minimally rigid subgraph 254 used for getting \Cref{tab:spherebounds}.}
	\label{tab:Senc:254-fan}
\end{table}

\begin{table}[H]
	\centering\scriptsize
	\begin{tabular}{rrr}
		\toprule
		n  & graph encoding & $\lamSIIfunc$ \\\midrule
		7 & 112525 & 64 \\
		8 & 10821356 & 128 \\
		9 & 43562960283 & 512 \\
		10 & 10059930403935 & 1024 \\
		11 & 1778189899030543 & 4096 \\
		12 & 1010360038673469964 & 10240 \\
		13 & 152150553443130661683232 & 32768 \\
		14 & 12699920118400119757561920 & 81920 \\
		15 & 1290592463576136176277725405703 & 262144 \\
		16 & 1135918983098992950261555449569796 & 720896 \\
		17 & 186094759147859265798923586845205397536 & 1900544 \\\bottomrule
	\end{tabular}
	\caption{Graphs with high number of spherical realizations and a six vertex minimally rigid subgraph 7916 used for getting \Cref{tab:spherebounds}.}
	\label{tab:Senc:7916-fan}
\end{table}

\begin{table}[H]
	\centering\scriptsize
	\begin{tabular}[t]{rrr}
		\toprule
		fan     & graph encoding   & $\lamIIIfunc$ \\\midrule
		7-fan   & 9264031572838635 & 9728 \\
		63-fan  & 2027871741807451 & 6400 \\
		511-fan & 1611981460276203 & 3712 \\\bottomrule
	\end{tabular}
	\qquad
	\begin{tabular}[t]{rrr}
		\toprule
		fan     & graph encoding   & $\lamIIIfunc$ \\\midrule
		7-fan   & 1151802456431434509 & 54272 \\
		63-fan  & 1132848041084674134 & 28672 \\
		511-fan & 1150729667545415175 & 16384 \\\bottomrule
	\end{tabular}
	\caption{Minimally 3-rigid graphs with 11 and 12 vertices and a high number of realizations used for getting \Cref{tab:bounds3d}.}
	\label{tab:enc:fan3d}
\end{table}

\section{Appendix --- Certificate graphs}\label{sec:factor-cert}
In this section we collect some certificates for the statements in \Cref{sec:factors}.
Note that these certificates are not necessarily unique.

\begin{table}[ht]
	\centering\scriptsize
	\begin{tabular}{lrrrrr}
		\toprule
		$|V|$ & $G=(V,E)$       & $\lamII{G}$ & $G'$            & $\lamII{G'}$ & Factor\\
		\midrule
		5  & 254                & 8    & 4011                   & 16   & 2 \\
		6  & 3326               & 16   & 167646                 & 32   & 2 \\
		7  & 120478             & 48   & 11357278               & 96   & 2 \\
		8  & 7510520            & 88   & 2462293336             & 176  & 2 \\
		9  & 833010561          & 208  & 2274775308581          & 416  & 2 \\
		10 & 207528715668       & 432  & 9042728074405044       & 864  & 2 \\
		11 & 107484831387142    & 512  & 504896642393621664     & 1024 & 2 \\
		12 & 109648827346584766 & 1024 & 1735798858112968179744 & 2048 & 2 \\
		\bottomrule
	\end{tabular}
	\begin{tabular}{lrrrrr}
		\toprule
		$|V|$ & $G=(V,E)$       & $\lamII{G}$ & $G'$             & $\lamII{G'}$ & Factor\\
		\midrule
		5  & 254                & 8    & 7916                    & 24   & 3.00 \\
		6  & 3934               & 16   & 1269995                 & 56   & 3.50 \\
		7  & 186013             & 32   & 170989214               & 136  & 4.25 \\
		8  & 11357293           & 64   & 5724735646              & 312  & 4.88 \\
		9  & 2621607781         & 128  & 5646831424844           & 736  & 5.75 \\
		10 & 1191880862997      & 256  & 19212043457692912       & 1728 & 6.75 \\
		11 & 607773613166125    & 512  & 23365097235260968172    & 4176 & 8.16 \\
		12 & 619316832844764209 & 1024 & 13999944567768887722001 & 9984 & 9.75 \\
		\bottomrule
	\end{tabular}
	\caption{Minimal and maximal increase of $\lamIIfunc$ by a 1-extension step of type E1b based on graphs with at most twelve vertices.}
	\label{tab:Laman_Increase_H2b}
\end{table}

\begin{table}[ht]
	\centering\scriptsize
	\begin{tabular}{lrrrrr}
		\toprule
		$|V|$ & $G=(V,E)$         & $\lamII{G}$ & $G'$            & $\lamII{G'}$ & Factor\\
		\midrule
		5  & 254                  & 8    & 7672                   & 16   & 2.00 \\
		6  & 7916                 & 24   & 481867                 & 44   & 1.83 \\
		7  & 1269995              & 56   & 31004235               & 96   & 1.71 \\
		8  & 6739377              & 112  & 1651611000             & 192  & 1.71 \\
		9  & 1361485524           & 224  & 415060176074           & 384  & 1.71 \\
		10 & 206991714068         & 448  & 634567474384650        & 768  & 1.71 \\
		11 & 18120782212031500    & 896  & 1371065508440933536    & 1536 & 1.71 \\
		12 & 41541803727340775650 & 2688 & 4702767505345224900620 & 4608 & 1.71 \\
		\bottomrule
	\end{tabular}
	\begin{tabular}{lrrrrr}
		\toprule
		$|V|$ & $G=(V,E)$        & $\lamII{G}$ & $G'$            & $\lamII{G'}$ & Factor\\
		\midrule
		5  & 254                 & 8    & 7672                   & 16    &  2.00 \\
		6  & 4011                & 16   & 1269995                & 56    &  3.50 \\
		7  & 190686              & 32   & 170989214              & 136   &  4.25 \\
		8  & 20042142            & 64   & 11177989553            & 344   &  5.38 \\
		9  & 2794620126          & 128  & 1813573113164          & 808   &  6.31 \\
		10 & 1248809223262       & 256  & 2960334732174949       & 1976  &  7.72 \\
		11 & 1710909647295913    & 512  & 15006592507478215906   & 4816  &  9.41 \\
		12 & 4649551155295838770 & 1024 & 8564720917032382554112 & 11872 & 11.59 \\
		\bottomrule
	\end{tabular}
	\caption{Minimal and maximal increase of $\lamIIfunc$ by a 1-extension of type E1c based on graphs with at most twelve vertices.}
	\label{tab:Laman_Increase_H2c}
\end{table}

\begin{table}[ht]
	\centering\scriptsize
	\begin{tabular}{lrrrrr}
		\toprule
		$|V|$ & $G=(V,E)$       & $\lamII{G}$ & $G'$            & $\lamII{G'}$ & Factor\\
		\midrule
		5  & 239                & 8    & 5791                   & 16   & 2 \\
		6  & 5791               & 16   & 567671                 & 32   & 2 \\
		7  & 567671             & 32   & 38835902               & 64   & 2 \\
		8  & 35982711           & 64   & 6513859183             & 128  & 2 \\
		9  & 19396904311        & 128  & 9354758752151          & 256  & 2 \\
		10 & 207528715668       & 432  & 178265086659764        & 864  & 2 \\
		11 & 18719774180030477  & 1728 & 505203673902716576     & 3456 & 2 \\
		12 & 252589376453374080 & 4088 & 1110700780390419841024 & 8176 & 2 \\
		\bottomrule
	\end{tabular}
	\begin{tabular}{lrrrrr}
		\toprule
		$|V|$ & $G=(V,E)$        & $\lamII{G}$ & $G'$              & $\lamII{G'}$ & Factor\\
		\midrule
		5  & 254                 & 8    & 7916                     & 24    &  3.00 \\
		6  & 3934                & 16   & 1269995                  & 56    &  3.50 \\
		7  & 186013              & 32   & 170989214                & 136   &  4.25 \\
		8  & 11636216            & 64   & 15177289073              & 336   &  5.25 \\
		9  & 2463307347          & 128  & 18795858309901           & 808   &  6.31 \\
		10 & 1172082332811       & 256  & 6914422794743947         & 1920  &  7.50 \\
		11 & 601024461194379     & 512  & 3749679080285441171      & 4744  &  9.27 \\
		12 & 1189938522902667411 & 1024 & 160027824476021084790788 & 11616 & 11.34 \\
		\bottomrule
	\end{tabular}
	\caption{Minimal and maximal increase of $\lamIIfunc$ by a vertex split based on graphs with at most twelve vertices.}
	\label{tab:Laman_Increase_Vsplit}
\end{table}

\begin{table}[ht]
	\centering\scriptsize
	\begin{tabular}{lrrrrr}
		\toprule
		$|V|$ & $G=(V,E)$         & $\lamII{G}$ & $G'$              & $\lamII{G'}$ & Factor\\
		\midrule
		5  & 254                  & 8    & 3326                     & 16   & 2.00 \\
		6  & 7672                 & 16   & 127198                   & 32   & 2.00 \\
		7  & 400857               & 32   & 6405034                  & 64   & 2.00 \\
		8  & 211042527            & 96   & 864467169                & 192  & 2.00 \\
		9  & 1016348895           & 192  & 207174779103             & 384  & 2.00 \\
		10 & 12146438253357       & 624  & 18124496647243558        & 1232 & 1.97 \\
		11 & 22817628108265694    & 1456 & 46279135492693888222     & 2864 & 1.97 \\
		12 & 43829316383861314782 & 3248 & 115519230969917969863816 & 6384 & 1.97 \\
		\bottomrule
	\end{tabular}
	\begin{tabular}{lrrrrr}
		\toprule
		$|V|$ & $G=(V,E)$        & $\lamII{G}$ & $G'$            & $\lamII{G'}$ & Factor\\
		\midrule
		5  & 239                 & 8    & 7916                   & 24    &  3.00 \\
		6  & 3934                & 16   & 1269995                & 56    &  3.50 \\
		7  & 167773              & 32   & 170989214              & 136   &  4.25 \\
		8  & 36738719            & 64   & 11177989553            & 344   &  5.38 \\
		9  & 17458998761         & 128  & 4778432477057          & 840   &  6.56 \\
		10 & 378251097553        & 256  & 18367623658392579      & 2168  &  8.47 \\
		11 & 9043558346904627    & 512  & 938475301372209920     & 5492  & 10.73 \\
		12 & 9277452071444257955 & 1024 & 9159515851062901770240 & 14214 & 13.88 \\
		\bottomrule
	\end{tabular}
	\caption{Minimal and maximal increase of $\lamIIfunc$ by a spider split based on graphs with at most ten vertices.}
	\label{tab:Laman_Increase_Ssplit}
\end{table}

\begin{table}[ht]
	\centering\scriptsize
	\begin{tabular}{lrrrrr}
		\toprule
		$|V|$ & $G=(V,E)$       & $\lamSII{G}$ & $G'$           & $\lamSII{G'}$ & Factor\\
		\midrule
		5  & 254                & 8    & 4011                   & 16   & 2 \\
		6  & 3326               & 16   & 167646                 & 32   & 2 \\
		7  & 120478             & 64   & 11357278               & 128  & 2 \\
		8  & 7510520            & 96   & 2462293336             & 192  & 2 \\
		9  & 833010561          & 224  & 2274775308581          & 448  & 2 \\
		10 & 207528715668       & 512  & 9042728074405044       & 1024 & 2 \\
		11 & 18719774448464909  & 1024 & 504902440985347472     & 2048 & 2 \\
		12 & 109642230153347189 & 3072 & 1113237338078023041024 & 6144 & 2 \\
		\bottomrule
	\end{tabular}
	\begin{tabular}{lrrrrr}
		\toprule
		$|V|$ & $G=(V,E)$        & $\lamSII{G}$ & $G'$            & $\lamSII{G'}$ & Factor\\
		\midrule
		5  & 254                 & 8    & 7916                    & 32    &  4 \\
		6  & 3326                & 16   & 120478                  & 64    &  4 \\
		7  & 183548              & 32   & 170957470               & 192   &  6 \\
		8  & 20093843            & 64   & 38945331569             & 512   &  8 \\
		9  & 2563805029          & 128  & 1796952196188           & 1280  & 10 \\
		10 & 2275289926222       & 256  & 22677431911835790       & 3072  & 12 \\
		11 & 1163020275982483    & 512  & 20897195286320951041    & 8192  & 16 \\
		12 & 1189954226918428849 & 1024 & 18635940629762920218656 & 21504 & 21 \\
		\bottomrule
	\end{tabular}
	\caption{Minimal and maximal increase of $\lamSIIfunc$ by a 1-extension of type E1b based on graphs with at most twelve vertices.}
	\label{tab:SLaman_Increase_H2b}
\end{table}

\begin{table}[ht]
	\centering\scriptsize
	\begin{tabular}{lrrrrr}
		\toprule
		$|V|$ & $G=(V,E)$       & $\lamSII{G}$ & $G'$           & $\lamSII{G'}$ & Factor\\
		\midrule
		5  & 254                & 8    & 7672                   & 16   & 2.0 \\
		6  & 7916               & 32   & 481867                 & 48   & 1.5 \\
		7  & 120478             & 64   & 7122342                & 96   & 1.5 \\
		8  & 6739377            & 128  & 1651611000             & 192  & 1.5 \\
		9  & 1361485524         & 256  & 415060176074           & 384  & 1.5 \\
		10 & 968052098124       & 1024 & 1764064806086260       & 1536 & 1.5 \\
		11 & 775070073152624    & 2048 & 507877685090371348     & 3072 & 1.5 \\
		12 & 254886745716355856 & 6144 & 2142833417622946590848 & 9216 & 1.5 \\
		\bottomrule
	\end{tabular}
	\begin{tabular}{lrrrrr}
		\toprule
		$|V|$ & $G=(V,E)$        & $\lamSII{G}$ & $G'$            & $\lamSII{G'}$ & Factor\\
		\midrule
		5  & 254                 & 8    & 7672                    & 16    &  2 \\
		6  & 4011                & 16   & 1269995                 & 64    &  4 \\
		7  & 190686              & 32   & 170989214               & 192   &  6 \\
		8  & 20042142            & 64   & 11177989553             & 512   &  8 \\
		9  & 4710608114          & 128  & 6702858835404           & 1408  & 11 \\
		10 & 1240861945058       & 256  & 496052904236256         & 3584  & 14 \\
		11 & 2290021342225298    & 512  & 1801933210823741964     & 10240 & 20 \\
		12 & 1297917820948026633 & 1024 & 15835665609699362931712 & 26624 & 26 \\
		\bottomrule
	\end{tabular}
	\caption{Minimal and maximal increase of $\lamSIIfunc$ by a 1-extension of type E1c based on graphs with at most twelve vertices.}
	\label{tab:SLaman_Increase_H2c}
\end{table}

\begin{table}[ht]
	\centering\scriptsize
	\begin{tabular}{lrrrrr}
		\toprule
		$|V|$ & $G=(V,E)$       & $\lamSII{G}$ & $G'$           & $\lamSII{G'}$ & Factor\\
		\midrule
		5  & 239                & 8    & 5791                   & 16    & 2 \\
		6  & 7916               & 32   & 127575                 & 64    & 2 \\
		7  & 112525             & 64   & 11881039               & 128   & 2 \\
		8  & 12885740           & 128  & 2720478547             & 256   & 2 \\
		9  & 833010561          & 224  & 1172019000513          & 448   & 2 \\
		10 & 207528715668       & 512  & 178265086659764        & 1024  & 2 \\
		11 & 119028075245345    & 2048 & 3531332902197577092    & 4096  & 2 \\
		12 & 252589376453374080 & 6912 & 1110700780390419841024 & 13824 & 2 \\
		\bottomrule
	\end{tabular}
	\begin{tabular}{lrrrrr}
		\toprule
		$|V|$ & $G=(V,E)$       & $\lamSII{G}$ & $G'$             & $\lamSII{G'}$ & Factor\\
		\midrule
		5  & 254                & 8    & 7916                     & 32    &  4 \\
		6  & 3326               & 16   & 120478                   & 64    &  4 \\
		7  & 183548             & 32   & 170957470                & 192   &  6 \\
		8  & 20093843           & 64   & 38945331569              & 512   &  8 \\
		9  & 2456982772         & 128  & 1315755596577            & 1536  & 12 \\
		10 & 637026913490       & 256  & 1338811453900388         & 3584  & 14 \\
		11 & 602981484498725    & 512  & 9945923421951266565      & 10240 & 20 \\
		12 & 630609743901537668 & 1024 & 163667169794306279997956 & 28672 & 28 \\
		\bottomrule
	\end{tabular}
	\caption{Minimal and maximal increase of $\lamSIIfunc$ by a vertex split based on graphs with at most twelve vertices.}
	\label{tab:SLaman_Increase_Vsplit}
\end{table}

\begin{table}[ht]
	\centering\scriptsize
	\begin{tabular}{lrrrrr}
		\toprule
		$|V|$ & $G=(V,E)$       & $\lamSII{G}$ & $G'$           & $\lamSII{G'}$ & Factor\\
		\midrule
		5  & 254                & 8    & 3326                   & 16   & 2 \\
		6  & 7672               & 16   & 127198                 & 32   & 2 \\
		7  & 400857             & 32   & 6405034                & 64   & 2 \\
		8  & 211042527          & 128  & 864467169              & 256  & 2 \\
		9  & 833010561          & 224  & 206991440769           & 448  & 2 \\
		10 & 207528715668       & 512  & 176128877450420        & 1024 & 2 \\
		11 & 10181617802924402  & 1024 & 504902440985347472     & 2048 & 2 \\
		12 & 109073487874526211 & 1792 & 1107348008565084258304 & 3584 & 2 \\
		\bottomrule
	\end{tabular}
	\begin{tabular}{lrrrrr}
		\toprule
		$|V|$ & $G=(V,E)$        & $\lamSII{G}$ & $G'$             & $\lamSII{G'}$ & Factor\\
		\midrule
		5  & 239                 & 8    & 7916                     & 32    &  4.0 \\
		6  & 4011                & 16   & 1256267                  & 64    &  4.0 \\
		7  & 560509              & 32   & 170957470                & 192   &  6.0 \\
		8  & 10800350            & 64   & 2993854888               & 576   &  9.0 \\
		9  & 4570175595          & 128  & 6702858835404            & 1408  & 11.0 \\
		10 & 8865630770537       & 256  & 776962027132128          & 3584  & 14.0 \\
		11 & 317862752457964     & 512  & 9421847768666934496      & 10496 & 20.5 \\
		12 & 4648880717118489771 & 1024 & 160007275596729620152336 & 27648 & 27.0 \\
		\bottomrule
	\end{tabular}
	\caption{Minimal and maximal increase of $\lamSIIfunc$ by a spider split based on graphs with at most ten vertices.}
	\label{tab:SLaman_Increase_Ssplit}
\end{table}

\begin{table}[ht]
	\centering\scriptsize
	\begin{tabular}{lrrrrr}
		\toprule
		$|V|$ & $G=(V,E)$ & $\lamIII{G}$ & $G'$ & $\lamIII{G'}$ & Factor\\
		\midrule
		7  & 237311      & 16  & 23576519       & 32  & 2.0 \\
		8  & 252810751   & 64  & 27594639355    & 96  & 1.5 \\
		9  & 19208142832 & 160 & 31895307006974 & 224 & 1.4 \\
		\bottomrule
	\end{tabular}
	\begin{tabular}{lrrrrr}
		\toprule
		$|V|$ & $G=(V,E)$ & $\lamIII{G}$ & $G'$ & $\lamIII{G'}$ & Factor\\
		\midrule
		7  & 237563     & 16 & 251895291      & 40   &  2.5 \\
		8  & 44875511   & 32 & 60694390687    & 256  &  8.0 \\
		9  & 5271681403 & 64 & 28731623567711 & 1024 & 16.0 \\
		\bottomrule
	\end{tabular}
	\caption{Minimal and maximal increase of $\lamIIIfunc$ by a 1-extension step of type E1s1 based on graphs up to nine vertices.}
	\label{tab:3dLaman_Increase_H2s1}
\end{table}

\begin{table}[ht]
	\centering\scriptsize
	\begin{tabular}{lrrrrr}
		\toprule
		$|V|$ & $G=(V,E)$ & $\lamIII{G}$ & $G'$ & $\lamIII{G'}$ & Factor\\
		\midrule
		6  & 16350      & 16  & 981215       & 24  & 1.5 \\
		7  & 260603     & 32  & 15187945     & 48  & 1.5 \\
		8  & 14940667   & 64  & 1894236122   & 96  & 1.5 \\
		9  & 1893988859 & 128 & 482930573274 & 192 & 1.5 \\
		\bottomrule
	\end{tabular}
	\begin{tabular}{lrrrrr}
		\toprule
		$|V|$ & $G=(V,E)$ & $\lamIII{G}$ & $G'$ & $\lamIII{G'}$ & Factor\\
		\midrule
		6  & 8187       & 8  & 1497823        & 32   &  4 \\
		7  & 384943     & 16 & 49524604       & 128  &  8 \\
		8  & 40365871   & 32 & 41115180723    & 384  & 12 \\
		9  & 9995933947 & 64 & 12175735772830 & 1280 & 20 \\
		\bottomrule
	\end{tabular}
	\caption{Minimal and maximal increase of $\lamIIIfunc$ by a 1-extension of type E1s30 based on graphs up to nine vertices.}
	\label{tab:3dLaman_Increase_H2s30}
\end{table}

\begin{table}[ht]
	\centering\scriptsize
	\begin{tabular}{lrrrrr}
		\toprule
		$|V|$ & $G=(V,E)$ & $\lamIII{G}$ & $G'$ & $\lamIII{G'}$ & Factor\\
		\midrule
		8 & 15195641   & 32 & 4058496961    & 52 & 1.625 \\
		9 & 4177288696 & 96 & 1095248625600 & 84 & 0.875 \\
		\bottomrule
	\end{tabular}
	\begin{tabular}{lrrrrr}
		\toprule
		$|V|$ & $G=(V,E)$ & $\lamIII{G}$ & $G'$ & $\lamIII{G'}$ & Factor\\
		\midrule
		8 & 23322107   & 32 & 1911013321    &  80 &  2.5 \\
		9 & 3114753523 & 64 & 2068893654067 & 896 & 14.0 \\
		\bottomrule
	\end{tabular}
	\caption{Minimal and maximal increase of $\lamIIIfunc$ by a 2-extension of type E2Xs12 based on graphs up to nine vertices.}
	\label{tab:3dLaman_Increase_H3Xs12}
\end{table}

\begin{table}[ht]
	\centering\scriptsize
	\begin{tabular}{lrrrrr}
		\toprule
		$|V|$ & $G=(V,E)$ & $\lamIII{G}$ & $G'$ & $\lamIII{G'}$ & Factor\\
		\midrule
		8 & 14940667    &  64 & 2984755142     &  80 & 1.25 \\
		9 & 27429496251 & 144 & 13270976003147 & 104 & 0.72 \\
		\bottomrule
	\end{tabular}
	\begin{tabular}{lrrrrr}
		\toprule
		$|V|$ & $G=(V,E)$ & $\lamIII{G}$ & $G'$ & $\lamIII{G'}$ & Factor\\
		\midrule
		8 & 14921726   & 32 & 2984755142    & 80  &  2.5 \\
		9 & 2968255347 & 64 & 2068893654067 & 896 & 14.0 \\
		\bottomrule
	\end{tabular}
	\caption{Minimal and maximal increase of $\lamIIIfunc$ by a 2-extension of type E2Vs3 based on graphs up to nine vertices.}
	\label{tab:3dLaman_Increase_H3Vs3}
\end{table}

\begin{table}[ht]
	\centering\scriptsize
	\begin{tabular}{lrrrrr}
		\toprule
		$|V|$ & $G=(V,E)$ & $\lamIII{G}$ & $G'$ & $\lamIII{G'}$ & Factor\\
		\midrule
		7  & 515806      & 48  & 31972856       &  48 & 1.00 \\
		8  & 16103411    & 96  & 1936586163     &  96 & 1.00 \\
		9  & 52359970796 & 352 & 26734869998966 & 336 & 0.95 \\
		\bottomrule
	\end{tabular}
	\begin{tabular}{lrrrrr}
		\toprule
		$|V|$ & $G=(V,E)$ & $\lamIII{G}$ & $G'$ & $\lamIII{G'}$ & Factor\\
		\midrule
		7  & 450271     & 16 & 62103033      &   64 &  4 \\
		8  & 40812151   & 32 & 23948082611   &  288 &  9 \\
		9  & 3678789086 & 64 & 6956268515550 & 1280 & 20 \\
		\bottomrule
	\end{tabular}
	\caption{Minimal and maximal increase of $\lamIIIfunc$ by a 2-extension step of type E2Xs236 based on graphs up to nine vertices.}
	\label{tab:3dLaman_Increase_H3Xs236}
\end{table}

\begin{table}[ht]
	\centering\scriptsize
	\begin{tabular}{lrrrrr}
		\toprule
		$|V|$ & $G=(V,E)$ & $\lamIII{G}$ & $G'$ & $\lamIII{G'}$ & Factor\\
		\midrule
		7  & 515806      &  48 & 15196152      &  48 & 1.00 \\
		8  & 49724126    & 160 & 40818832819   &  96 & 0.60 \\
		9  & 37880438494 & 448 & 9660195076318 & 192 & 0.43 \\
		\bottomrule
	\end{tabular}
	\begin{tabular}{lrrrrr}
		\toprule
		$|V|$ & $G=(V,E)$ & $\lamIII{G}$ & $G'$ & $\lamIII{G'}$ & Factor\\
		\midrule
		7  & 450271      & 16 & 16162809      &   64 &  4 \\
		8  & 75372541    & 32 & 4089036272    &  416 & 13 \\
		9  & 18910895727 & 64 & 3559486584453 & 1664 & 26 \\
		\bottomrule
	\end{tabular}
	\caption{Minimal and maximal increase of $\lamIIIfunc$ by a 2-extension of type E2Vs236 based on graphs up to nine vertices.}
	\label{tab:3dLaman_Increase_H3Vs236}
\end{table}

\begin{table}[ht]
	\centering\scriptsize
	\begin{tabular}{lrrrrr}
		\toprule
		$|V|$ & $G=(V,E)$ & $\lamIII{G}$ & $G'$ & $\lamIII{G'}$ & Factor\\
		\midrule
		6  & 7935       &   8 & 237563       &  16 & 2.00 \\
		7  & 515806     &  48 & 62103033     &  64 & 1.33 \\
		8  & 63430223   &  64 & 14527489615  &  64 & 1.00 \\
		9  & 3010182951 & 128 & 974469719631 & 128 & 1.00 \\
		\bottomrule
	\end{tabular}
	\begin{tabular}{lrrrrr}
		\toprule
		$|V|$ & $G=(V,E)$ & $\lamIII{G}$ & $G'$ & $\lamIII{G'}$ & Factor\\
		\midrule
		6  & 7935       &  8 & 260603         &   32 &  4 \\
		7  & 384510     & 16 & 49524604       &  128 &  8 \\
		8  & 74378103   & 32 & 6336901027     &  448 & 14 \\
		9  & 9459054061 & 64 & 12178932745662 & 1280 & 20 \\
		\bottomrule
	\end{tabular}
	\caption{Minimal and maximal increase of $\lamIIIfunc$ by a 2-extension of type E2Xs239 based on graphs up to nine vertices.}
	\label{tab:3dLaman_Increase_H3Xs239}
\end{table}

\begin{table}[ht]
	\centering\scriptsize
	\begin{tabular}{lrrrrr}
		\toprule
		$|V|$ & $G=(V,E)$ & $\lamIII{G}$ & $G'$ & $\lamIII{G'}$ & Factor\\
		\midrule
		6  & 7935        &   8 & 237563        &  16 & 2.00 \\
		7  & 515806      &  48 & 73914364      &  32 & 0.67 \\
		8  & 49724126    & 160 & 18848282483   &  64 & 0.40 \\
		9  & 11717490611 & 512 & 9634462543324 & 128 & 0.25 \\
		\bottomrule
	\end{tabular}
	\begin{tabular}{lrrrrr}
		\toprule
		$|V|$ & $G=(V,E)$ & $\lamIII{G}$ & $G'$ & $\lamIII{G'}$ & Factor\\
		\midrule
		6  & 7935       &  8 & 515806        &   48 &  6 \\
		7  & 382463     & 16 & 49724126      &  160 & 10 \\
		8  & 74100607   & 32 & 11717490611   &  512 & 16 \\
		9  & 5711063903 & 64 & 6130619373214 & 1408 & 22 \\
		\bottomrule
	\end{tabular}
	\caption{Minimal and maximal increase of $\lamIIIfunc$ by a 2-extension of type E2Vs239 based on graphs up to nine vertices.}
	\label{tab:3dLaman_Increase_H3Vs239}
\end{table}

\begin{table}[ht]
	\centering\scriptsize
	\begin{tabular}{lrrrrr}
		\toprule
		$|V|$ & $G=(V,E)$ & $\lamIII{G}$ & $G'$ & $\lamIII{G'}$ & Factor\\
		\midrule
		5  & 511        & 4  & 7935         & 8   & 2 \\
		6  & 7679       & 8  & 237311       & 16  & 2 \\
		7  & 237055     & 16 & 14917375     & 32  & 2 \\
		8  & 14917119   & 32 & 1893965567   & 64  & 2 \\
		9  & 1893965311 & 64 & 482930302719 & 128 & 2 \\
		\bottomrule
	\end{tabular}
	\begin{tabular}{lrrrrr}
		\toprule
		$|V|$ & $G=(V,E)$ & $\lamIII{G}$ & $G'$ & $\lamIII{G'}$ & Factor\\
		\midrule
		5  & 511        & 4  & 16350         & 16   & 4  \\
		6  & 7935       & 8  & 515806        & 48   & 6  \\
		7  & 382463     & 16 & 49724126      & 160  & 10 \\
		8  & 74100607   & 32 & 11717490611   & 512  & 16 \\
		9  & 6814169327 & 64 & 3112621267595 & 1664 & 26 \\
		\bottomrule
	\end{tabular}
	\caption{Minimal and maximal increase of $\lamIIIfunc$ by a vertex split based on graphs up to nine vertices.}
	\label{tab:3dLaman_Increase_Vsplit}
\end{table}

\begin{table}[ht]
	\centering\scriptsize
	\begin{tabular}{lrrrrr}
		\toprule
		$|V|$ & $G=(V,E)$ & $\lamIII{G}$ & $G'$ & $\lamIII{G'}$ & Factor\\
		\midrule
		5  & 511        & 4  & 7679         & 8   & 2 \\
		6  & 7679       & 8  & 237055       & 16  & 2 \\
		7  & 237055     & 16 & 14917119     & 32  & 2 \\
		8  & 14917119   & 32 & 1893965311   & 64  & 2 \\
		9  & 1893965311 & 64 & 482930302463 & 128 & 2 \\
		\bottomrule
	\end{tabular}
	\begin{tabular}{lrrrrr}
		\toprule
		$|V|$ & $G=(V,E)$ & $\lamIII{G}$ & $G'$ & $\lamIII{G'}$ & Factor\\
		\midrule
		5  & 511         & 4  & 7679          & 8    & 2  \\
		6  & 8187        & 8  & 515806        & 48   & 6  \\
		7  & 450271      & 16 & 49724126      & 160  & 10 \\
		8  & 52944111    & 32 & 7345971057    & 640  & 20 \\
		9  & 13236016767 & 64 & 2004558244619 & 2304 & 36 \\
		\bottomrule
	\end{tabular}
	\caption{Minimal and maximal increase of $\lamIIIfunc$ by a spider split based on graphs up to nine vertices.}
	\label{tab:3dLaman_Increase_Ssplit}
\end{table}

\end{document}